\documentclass[10pt]{amsart}

\usepackage{amsmath,amssymb,amsthm}
\usepackage{graphicx}
\usepackage{tikz}
\usetikzlibrary{arrows.meta, shapes.geometric}
\usepackage{url}
\IfFileExists{hyperref.sty}{\usepackage[colorlinks=true,linkcolor=blue,citecolor=blue,urlcolor=blue]{hyperref}}{}
\IfFileExists{todonotes.sty}{\usepackage{todonotes}}{\newcommand{\todo}[2][]{}}

\newcommand{\Q}{\mathbb{Q}}
\newcommand{\Z}{\mathbb{Z}}
\newcommand{\C}{\mathbb{C}}
\newcommand{\F}{\mathbb{F}}
\newcommand{\Qbar}{\overline{\Q}}
\DeclareFontEncoding{OT2}{}{}
\DeclareFontFamily{OT2}{wncyr}{}
\DeclareFontShape{OT2}{wncyr}{m}{n}{<->wncyr10}{}
\DeclareFontSubstitution{OT2}{wncyr}{m}{n}
\newcommand{\Sha}{\text{\usefont{OT2}{wncyr}{m}{n}Sh}}
\newcommand{\fp}{\mathfrak{p}}

\newcommand{\ff}{\mathfrak{f}}
\newcommand{\OK}{\mathcal{O}_K}

\DeclareMathOperator{\rank}{rank}
\DeclareMathOperator{\ord}{ord}
\DeclareMathOperator{\Reg}{Reg}
\DeclareMathOperator{\Sel}{Sel}
\DeclareMathOperator{\corank}{corank}
\DeclareMathOperator{\Gal}{Gal}
\DeclareMathOperator{\Tr}{Tr}
\DeclareMathOperator{\Nm}{N}
\DeclareMathOperator{\car}{char}

\newtheorem{theorem}{Theorem}[section]
\newtheorem{lemma}[theorem]{Lemma}
\newtheorem{proposition}[theorem]{Proposition}

\newtheorem{corollary}[theorem]{Corollary}

\newtheorem{thmintro}{Theorem}

\newtheorem{problem}[theorem]{Problem}
\newtheorem{question}{Question}
\newenvironment{qrestate}[1]{\par\medskip\noindent\textbf{Question~\ref{#1}.}\ \itshape\ignorespaces}{\par\medskip}
\theoremstyle{definition}
\newtheorem{definition}[theorem]{Definition}
\theoremstyle{remark}
\newtheorem{remark}[theorem]{Remark}

\title[Second derivatives of $p$-adic $L$-functions and $\Sha$]{Second derivatives of $p$-adic $L$-functions and the Shafarevich--Tate group of rank-two CM elliptic curves}

\author{Barinder S.~Banwait}
\address{London, UK}
\email{barinder.s.banwait@gmail.com}
\urladdr{\url{https://barindersbanwait.com}}


\begin{document}

\begin{abstract}
For an elliptic curve $E/\Q$ of rank two with complex multiplication, Coates, Liang and Sujatha gave a criterion for the vanishing of $\Sha(E/\Q)[p^\infty]$ at a good ordinary prime $p$ and applied it to five such curves for $p < 30{,}000$. We prove a cyclotomic criterion of the same kind: outside an explicit set of primes, the normalised second Taylor coefficient of the Mazur--Tate--Teitelbaum $p$-adic $L$-function at the central point is a $p$-adic unit if and only if the cyclotomic $p$-adic regulator is a unit and $\Sha(E/\Q)[p^\infty] = 0$, and, by the theory of Bannai and Kobayashi, if and only if an explicit combination of three critical Hecke $L$-values of weight $2p - 1$ has valuation exactly two. Following the algorithm of Stein and Wuthrich, we compute the regulator for the same five curves at every good ordinary prime below $30{,}000$: it is a unit at all but three of the $8{,}050$ primes outside the excluded set. The one case that the criterion of Coates, Liang and Sujatha left open, $p = 577$ for $y^2 = x^3 + 34x$, is settled by the new criterion. A Lean 4 formalisation of the first equivalence, assuming stated results from the literature, is provided.

\end{abstract}

\maketitle

\setcounter{tocdepth}{1}
\tableofcontents

\section{Introduction}\label{sec:intro}

\subsection{The problem}
For an elliptic curve $E/\Q$, the Shafarevich--Tate group $\Sha(E/\Q)$ is conjectured to be finite; and while this is known if $E$ has rank $0$ or $1$, it remains totally open for rank $\geq 2$, and to date there is not even a single known example of such a curve for which we unconditionally know finiteness of $\Sha(E/\Q)$.

Finiteness of $\Sha$ is equivalent to the following two statements:
\begin{enumerate}
  \item[(a)] there exists $M$ such that $\Sha(E/\Q)[p] = 0$ for all
        primes $p > M$;
  \item[(b)] $\Sha(E/\Q)[p^\infty]$ is finite for each $p \le M$.
\end{enumerate}
Statement (b) is \emph{vertical}: it concerns one prime at a time, and in the good ordinary case, cyclotomic Iwasawa theory relates it to the Selmer group over the cyclotomic $\Z_p$-extension. In particular, Coates,
Liang and Sujatha \cite{CLS1,CLS2} gave a criterion for the curves
$y^2 = x^3 - Dx$, which have complex multiplication by $\Z[i]$: at a
prime $p \equiv 1 \bmod 4$ with $p \nmid D$ this criterion compares the $p$-adic
valuation of one critical Hecke $L$-value with the Mordell--Weil rank,
and equality gives $\Sha(E/\Q)[p^\infty] = 0$.  Its input is Rubin's
main conjecture for $\Q(i)$ \cite{Rubin}, read through a
leading-coefficient criterion rather than by Iwasawa descent from a
computed $p$-adic $L$-function, and it is decidable prime by prime. They executed this criterion on five curves of rank two in the family and showed for these five curves that $\Sha(E/\Q)[p^\infty]$ is
finite at every good prime $p \equiv 1 \bmod 4$ below $30{,}000$. Later, the descent formalism of Perrin-Riou
\cite{PerrinRiou82} and Schneider \cite{Schneider} was made algorithmic
by Stein and Wuthrich \cite[Thm.~6.1]{SteinWuthrich}, and in practice it will either prove
$\Sha[p^\infty]$ finite (indeed, compute its conjectural order) or
report a genuine obstruction; the bound on $\#\Sha$ it produces rests
on Kato's divisibility \cite{Kato}, a theorem for curves without
complex multiplication, so the two methods apply to complementary classes
of curves.

Statement (a) is \emph{horizontal}: it
quantifies over all sufficiently large $p$ at once, and it is open for
every elliptic curve over $\Q$ of rank at least two. Even the significantly weaker statement that $\Sha(E/\Q)[p] = 0$ for merely infinitely many $p$ (as opposed to all but finitely many) remains open.

This paper does not prove finiteness of $\Sha$ for any elliptic curve of rank two. What it proves is of type (b). It takes the criterion of Coates, Liang and Sujatha as its starting point, proves a cyclotomic counterpart of it, extends their computations to the cyclotomic side, and states questions that the outcome raises. Its bearing on (a) is through those questions, and concerns the split primes only.

\subsection{What is known}
In rank $\le 1$ statement (a) is a theorem, and its proof rests on one
algebraic object computed once: Kato's Euler system bounds $\#\Sha(E/\Q)$
by an explicit multiple of $L(E,1)/\Omega_E$ \cite[Thm.~17.4]{Kato}, and
Kolyvagin's bounds it by the index of the Heegner point
\cite[Thm.~A]{Kolyvagin}, non-torsion by Gross--Zagier
\cite[Ch.~I, (6.3)]{GrossZagier}; the prime factorisation of that object
determines the finite set of primes excluded.  In rank two the following results are vertical.  Stein and Wuthrich prove $\Sha(E/\Q)[p] = 0$ for
$1{,}534{,}422$ pairs $(E,p)$ with $E$ non-CM of rank $\ge 2$ and
conductor $\le 30{,}000$, at good ordinary $5 \le p < 1000$ with
$\bar\rho_{E,p}$ surjective \cite[Thm.~1.1]{SteinWuthrich}, and for the
curve $389a$ at every good ordinary $p < 48{,}859$ but one
\cite[Thm.~12.3]{SteinWuthrich}; Coates, Liang and Sujatha prove
$\Sha(E/\Q)[p^\infty] = 0$ for five CM curves of rank two at every good $p \equiv 1 \bmod 4$ below $30{,}000$ outside four pairs \cite[Thm.~1.3]{CLS2}.  The Selmer-class constructions adapted to
rank two are vertical as well, one prime at a time and conditional on a
nonvanishing at that prime (\S\ref{ssec:ranktwo}).  Horizontally, Coates, Liang and Sujatha bound the $\Z_p$-corank of $\Sha(E/\Q)[p^\infty]$ for a CM curve by a quantity of order $p$ for all sufficiently large good ordinary $p$ \cite[Thm.~1.1]{CLS2}; they
note that their method cannot prove the vanishing of that corank for
infinitely many $p$ in any new case \cite[\S 1]{CLS1}.

\subsection{Results}\label{ssec:results}
There are two results in this paper: Theorem~\ref{intro:consequence} and Theorem~\ref{intro:criterion}. They both concern a CM curve $E/\Q$ of Mordell--Weil rank two whose $L$-function vanishes at $s = 1$ with root number $+1$, \[\rank E(\Q) = 2, \qquad L(E,1) = 0, \qquad w(E) = +1,\] where $w(E)$ is the sign of the functional equation of $L(E,s)$. These three conditions hold when the algebraic and analytic ranks are both two; the analytic rank itself, that is $L''(E,1) \ne 0$, is not assumed. Each result is stated below with its hypotheses in full, for which we require some definitions.

A prime $p$ is called \emph{split} if it splits in
$K$, in which case we write $p = \fp\bar\fp$. At each good split
prime the Mazur--Tate--Teitelbaum $p$-adic $L$-function $L_p(E,T)$
\cite{MTT} vanishes to order $\ge 2$ at $T = 0$,
and the quantity deciding both $\lambda_{\mathrm{an}}(\fp) = 2$ and the
triviality of $\Sha(E/\Q)[p^\infty]$ is the $\fp$-adic valuation of the
second Taylor coefficient $c_2(p)$. We work below with a normalised version of this, $\tilde c_2(p)$, normalised as in \cite{MTT,SteinWuthrich}. The set $S_E$ appearing in the theorem statements is the set of primes
given explicitly by \eqref{eq:Sexc}; it has density zero, and in some cases (including those arising in the computations), it is finite. We write $\Reg_\gamma$ for the cyclotomic $p$-adic regulator normalised as in \cite{SteinWuthrich,MST}.

\begin{thmintro}[The unit condition forces triviality]\label{intro:consequence}
Let $E/\Q$ be an elliptic curve with complex multiplication by the
maximal order of an imaginary quadratic field $K$, with
$\rank E(\Q) = 2$, $L(E,1) = 0$ and $w(E) = +1$.  Let $p = \fp\bar\fp \ge 5$ be a
split prime of good reduction lying outside the set $S_E$ of
\eqref{eq:Sexc}.  If $\tilde c_2(p) \in \Z_p^\times$, then
\[
   \lambda_{\mathrm{an}}(\fp) = 2, \qquad
   \corank_{\Z_p}\Sel_{p^\infty}(E/\Q) = 2, \qquad
   \Sha(E/\Q)[p^\infty] = 0 .
\]
\end{thmintro}

\noindent This is proved as Theorem~\ref{prop:consequence}, by Rubin's main
conjecture for $K$ together with a structure argument that replaces the
Schneider--Perrin-Riou descent formalism.

Theorem~\ref{intro:consequence} leaves the question of when its condition $\tilde c_2(p) \in \Z_p^\times$ (hereafter referred to simply as \emph{the unit condition}) holds. Theorem~\ref{intro:criterion} gives two answers. Arithmetically, the unit condition holds if and only if the cyclotomic $p$-adic regulator is a unit and $\Sha(E/\Q)[p^\infty] = 0$; this is the form the computations of \S\ref{ssec:intro_evidence} test. Analytically, it is decided by critical Hecke $L$-values, as in the criterion of Coates, Liang and Sujatha. At a split prime, $L_p(E,T)$ is, up to a comparison constant and a unit power series (Lemma~\ref{lem:comparison}), the restriction of the Katz measure \cite{Katz} to the cyclotomic line through the central character of the Grössencharacter $\psi_E$ of $E$ (\S\ref{ssec:katz}), and that measure interpolates the critical Hecke $L$-values of $K$. So $c_2(p)$ is, up to a unit and the factor $\log_p(1+p)^{-2}$, the second moment of the measure along that line: the integral of $\ell^2$, for $\ell$ the $p$-adic logarithm of the cyclotomic character. Since $\log_p z \equiv 1 - z^{p-1} \bmod p^2$ for a local unit $z$, this second moment is, modulo $p^3$, a combination of three polynomial moments of the measure in the local coordinates at $\fp$ and $\bar\fp$, of total degree $2p-2$. Bannai and Kobayashi \cite{BannaiKobayashi} evaluate the polynomial moments of the Katz measure as \emph{Eisenstein--Kronecker numbers} of the CM lattice: the Taylor coefficients of the Kronecker theta function at torsion points (\S\ref{ssec:BK}); a moment of degree $2p-2$ pairs with a number of weight $2p-1$. Summed over the ray class group modulo the conductor $\ff$ of $\psi_E$ (the class sums of \S\ref{sec:jet}), these numbers are, up to periods, the critical values $L_\ff(\bar\psi_E^{\,2p-1}, s)$ (Proposition~\ref{prop:classsumsL}). The number $B(\fp)$ of Theorem~\ref{intro:criterion} is the combination \eqref{eq:bracket} of the three critical values that arise: at $s = 1$, $s = p$ and $s = 2p-1$, with the Euler factors at $\fp$ and $\bar\fp$ and the factorials of the evaluation (Corollary~\ref{cor:bracketL}). Its $\fp$-adic valuation is at least $2$, and the criterion is that it equals $2$.

\begin{thmintro}[The exact criterion]\label{intro:criterion}
Let $E/\Q$ be an elliptic curve with complex multiplication by the
maximal order of an imaginary quadratic field $K$, with
$\rank E(\Q) = 2$, $L(E,1) = 0$ and $w(E) = +1$.  Let $p = \fp\bar\fp \ge 5$ be a
split prime of good reduction lying outside the set $S_E$ of
\eqref{eq:Sexc}, and let $B(\fp)$ be the combination \eqref{eq:bracket}
of three Eisenstein--Kronecker class sums of weight $2p-1$.  The
following are equivalent:
\begin{enumerate}
\item $\tilde c_2(p) \in \Z_p^\times$;
\item the cyclotomic $p$-adic height pairing on $E(\Q)$ is
      nondegenerate with $v_p\bigl(\Reg_\gamma\bigr) = 0$, and
      $\Sha(E/\Q)[p^\infty] = 0$;
\item $v_\fp\bigl(B(\fp)\bigr) = 2$.
\end{enumerate}
\end{thmintro}

\noindent This is proved as Theorem~\ref{thm:criterion}. The equivalence of (1) and (2) is Proposition~\ref{prop:dictionary}, from Rubin's main conjecture and Schneider's leading-term theorem. That of (1) and (3) runs through the Katz measure: the second Taylor coefficient is a second moment of the measure (Proposition~\ref{prop:jetformula}), its valuation is decided by one residue class (Proposition~\ref{prop:grading}), the comparison and Euler factors do not affect that class (Lemma~\ref{lem:decoupling}), and the class is computed exactly (Proposition~\ref{prop:exactreduction}). Since the class sums in $B(\fp)$ are, up to periods, critical values of the $(2p-1)$-th power of the conjugate Grössencharacter (Proposition~\ref{prop:classsumsL}, Corollary~\ref{cor:bracketL}), Theorem~\ref{intro:criterion} is the cyclotomic counterpart of the Coates--Liang--Sujatha criterion \cite[Thm.~2.2]{CLS1}, which tests one critical value of the $p$-th power; it decides the regulator as well as $\Sha$.

\subsection{The computations}\label{ssec:intro_evidence}
Part~\ref{part:evidence} computes the unit condition for five curves, through the equivalence of (1) and (2) in Theorem~\ref{intro:criterion}.

We work with the five curves of rank two of Coates, Liang and Sujatha \cite[Thm.~1.3]{CLS2},
\[
   y^2 = x^3 - Dx, \qquad D = -14,\ 17,\ -33,\ -34,\ -39,
\]
all with complex multiplication by $\Z[i]$; for the first we take the $2$-isogenous curve $E \colon y^2 = x^3 - 56x$, of conductor $12544 = 2^8 \cdot 7^2$ (\S\ref{ssec:testbed}), to which their result on $\Sha$ transfers (Lemma~\ref{lem:isogeny}). For these curves Theorem~1.3 of \cite{CLS2} gives $\Sha(E/\Q)[p^\infty] = 0$ at every split prime $p < 30{,}000$ outside four pairs, at which their criterion is inconclusive, and Theorem~\ref{intro:consequence} settles the four (Propositions~\ref{prop:threepairs} and~\ref{prop:e4}). The evidence discussed below is therefore free of any hypotheses on $\Sha$ in this range.

\emph{What makes a computation possible.}  What can be computed at a
split prime is not $\tilde c_2(p)$ itself, but rather the cyclotomic $\fp$-adic
regulator $\Reg_\fp$ of a fixed Mordell--Weil basis, and its
normalisation $\Reg_\gamma = \Reg_\fp/\log_p(1+p)^2$
(\S\ref{ssec:notation}).  Theorem~\ref{intro:consequence} converts the
unit condition into an arithmetic statement. Theorem~\ref{intro:criterion} converts it both ways: its statements (1) and (2) are equivalent.  The two directions are used differently.

$(1) \Rightarrow (2)$, read contrapositively, \emph{refutes}.  A
computed $v_p(\Reg_\gamma) > 0$ makes (2) false, hence $\tilde c_2(p) \notin \Z_p^\times$.  This uses nothing about $\Sha$, and is available at every such
prime, in any range.

$(2) \Rightarrow (1)$ \emph{verifies}, and a computation does not
establish (2) on its own.  A computed $v_p(\Reg_\gamma) = 0$ gives
the first clause of (2) but not the vanishing of $\Sha(E/\Q)[p^\infty]$; this is supplied by the Coates--Liang--Sujatha input for $p < 30{,}000$, completed at its four inconclusive pairs by Propositions~\ref{prop:threepairs} and~\ref{prop:e4}, so (2) holds in this range and (1) follows.

\emph{The result at $8{,}052$ primes.} The regulator was computed at every split prime $p < 30{,}000$ of good reduction for each of the five curves, $8{,}052$ primes in all (\S\ref{ssec:scan}, \S\ref{ssec:family}). Two of them, $p = 5$ for $D = 17$ and for $D = -33$, are anomalous and lie in $S_E$. At the other $8{,}050$ the cyclotomic regulator is a unit except at three: $p = 37$ for $y^2 = x^3 + 33x$, and $p = 5$ and $p = 15289$ for $y^2 = x^3 + 39x$, each outside $S_E$, each confirmed by a second implementation, and the one at $p = 5$ also by the modular-symbol $p$-adic $L$-function. Each of the three is therefore a failure of the unit condition, and it is the regulator that fails and not $\Sha$. Every other prime outside $S_E$ is a verified instance of the unit condition $\tilde c_2(p) \in \Z_p^\times$, with no assumption on $\Sha$; for $y^2 = x^3 - 56x$, whose $S_E$ contains no split prime, that is all $1611$ of its primes. Three failures in five curves is the rate at which a random residue vanishes, about $3.8$ expected (\S\ref{ssec:scanweight}), and at the four primes where the Coates--Liang--Sujatha criterion is inconclusive, two of them for the isogeny partner of $y^2 = x^3 - 56x$, the cyclotomic regulator is a unit. So the unit condition fails at split primes outside $S_E$ at a Wieferich-type rate, and the two criteria fail at different primes; \S\ref{sec:questions} states the questions this raises. At the one case their criterion and Wuthrich's computations left open, $p = 577$ for $y^2 = x^3 + 34x$, a modular-symbol computation gives the unit condition directly, hence $\Sha(E/\Q)[577^\infty] = 0$ by Theorem~\ref{intro:consequence} (Proposition~\ref{prop:e4}).

\subsection{The questions}\label{ssec:intro_questions}
The computations raise several questions, which we pose without answering.

\begin{question}\label{q:infinite}
Is the set of split primes at which the unit condition fails an infinite set?
\end{question}

\noindent If the residue of $\Reg_\gamma$ modulo $p$ behaved like a random element of $\F_p$, there would be about $\tfrac12\log\log X$ such primes below $X$; the five curves gave three below $30{,}000$ against about $3.8$ expected.

\begin{question}\label{q:independent}
Are the failures of the cyclotomic criterion and of the criterion of Coates, Liang and Sujatha independent?
\end{question}

\noindent Below $30{,}000$ the two criteria failed at disjoint sets of primes. If the two events at $p$ are independent with probability about $1/p$ each, the convergence of $\sum_p p^{-2}$ suggests that for all but finitely many split $p$ at least one of the two criteria gives $\Sha(E/\Q)[p^\infty] = 0$.

We speculate about the existence of a single algebraic object that governs the failure of the unit condition, analogous to the Heegner point in rank one.

\begin{question}\label{q:fixed}
Is there a fixed algebraic number whose prime divisors contain every split prime at which the unit condition fails?
\end{question}

\noindent The precursor to this question is Gillard's theorem (Theorem~\ref{thm:gillard}), which gives the vanishing of the $\mu$-invariant of every branch of the Katz measure at every split prime not dividing $6\Nm\ff$. The proof descends from Ferrero--Washington \cite{FerreroWashington} and Sinnott \cite{Sinnott} for Dirichlet $L$-functions, Sinnott's rational-function argument transplanted to the CM curve. The coefficient $c_2(p)$ is, up to a unit, a Taylor coefficient of a specialisation of the same object (\S\ref{ssec:BK}), and the question above asks whether a similar phenomenon occurs for that coefficient: the bracket $B(\fp)$ of Theorem~\ref{intro:criterion} is a different algebraic number at each prime, of weight $2p-1$, and the question is whether the primes at which its valuation exceeds $2$ are the prime divisors of one number independent of $p$.

\subsection{Software, data and reproduction}\label{ssec:software}
Most of the computations have been carried out with PARI/GP 2.17.2 \cite{PARI} on an Apple M1 Pro
(arm64) with 16\,GB RAM. The modular symbols computations have been done in SageMath 10.7 \cite{Sage} and its eclib modular symbols. Every script and data file quoted in this paper is in the repository
\begin{center}
\url{https://github.com/BarinderBanwait/rank2sha}
\end{center}
\noindent
It has two directories: \texttt{code/}, holding the
computations of Part~\ref{part:evidence}, and \texttt{formalisation/},
holding the Lean formalisation of Part~\ref{part:formal}.  Within
\texttt{code/}, the \textsc{gp} scripts are under \texttt{gp/}, the
SageMath and Python scripts under \texttt{sage/}, and the output of
every run under \texttt{data/}.  Each script resolves its own data paths
as \texttt{../data/}, so it is run from the directory it sits in.  The file
\texttt{README.md} there maps each number quoted in
Part~\ref{part:evidence} to the script that computes it and to the file
that records it, and \texttt{build.sh} checks a reader's machine for the software named above and runs a short smoke test.

\subsection{Formalisation}\label{ssec:formalintro}
The results of Part~\ref{part:theory} are accompanied by a formal
verification in Lean~4 \cite{Lean4}, in the \texttt{formalisation/}
directory of the repository just named.  It is a formalisation modulo the literature: the
classical theorems this paper takes from other authors are assumed, as
fields of a single Lean structure carrying their pinpoint citations,
and the deductions this paper adds to them are machine-checked.
Theorem~\ref{intro:consequence} and the equivalence of (1) and (2) in Theorem~\ref{intro:criterion} are proved in that sense, as are several of the subsidiary statements; the equivalence of (1) and (3) is not formalised.  The assumed
statements are also shown to be satisfiable, and not to imply
$\Sha(E/\Q)[p^\infty] = 0$ on their own.

A script in the repository checks that no proof in the imported tree
is incomplete and that the audited declarations use no axiom beyond
\texttt{propext}, \texttt{Classical.choice} and \texttt{Quot.sound};
there are no \texttt{axiom} declarations in the formalisation, so every
assumed input appears in the statement of each theorem that uses it.
Part~\ref{part:formal} gives the details: what is proved, what is
assumed, what is not formalised, and what a referee must read to check
that the assumed statements say what the cited papers say.

\subsection{Roadmap}
Part~\ref{part:theory} is the mathematics, stated for any curve satisfying the hypotheses and free of examples: Section~\ref{sec:background} collects the background, Section~\ref{sec:thmA} proves Theorem~\ref{intro:consequence}, Section~\ref{sec:basis} proves Proposition~\ref{prop:dictionary}, Section~\ref{sec:jet} fixes the Eisenstein--Kronecker data, and Section~\ref{sec:criterion} proves Theorem~\ref{intro:criterion}. Part~\ref{part:evidence} carries the computations out: Section~\ref{sec:evidence} the regulator scans of five curves, Section~\ref{sec:bracket} the criterion at sixteen pairs of a curve and a prime. Part~\ref{part:formal} concerns the formalisation, and Part~\ref{part:outlook} discusses the questions in greater detail.

\subsection*{Acknowledgements}
I am grateful to David Loeffler for taking the time to explain to me some of the broader questions of Iwasawa theory concerning rank 2 elliptic curves at the campus pub of the University of East Anglia, Norwich, UK, in July 2026, where the motivation to pursue the question of this paper arose. This was during the \emph{Bridging Lean and the LMFDB} workshop, organised by David Angdinata and Chris Birkbeck. I thank them also for the invitation to participate in that workshop.

\subsection*{Declaration on AI usage}
Claude \cite{Claude} has played an essential role in this project. If AI agents were permitted to be authors of mathematics papers, then Claude would be an author of this one.

I (BSB) describe the AI usage across different aspects of this project.

\emph{The mathematical results.}  Theorems~\ref{intro:consequence} and~\ref{intro:criterion} arose after a long discussion with Claude (Fable 5). This discussion was aimed at provably finding an example of a rank $2$ elliptic curve with finite $\Sha$. Claude suggested an approach for CM curves and proposed a candidate object answering Question~\ref{q:fixed}; the proposal was ultimately seen to be circular. Long story short: the parts of that approach that survived scrutiny became the results of this paper and those that did not became the questions.

\emph{The exposition.}  The writing of the paper has been a collaboration with Claude.

\emph{The computations and formalisation.}  The Sage, PARI/GP, and Lean code were developed across several sessions by several teams of Claude agents on Opus 5, project-managed by Claude on Fable 5 or 5.1. For the Sage and PARI/GP scripts, the paper of Stein--Wuthrich \cite{SteinWuthrich} was given as the computational model to follow. For the Lean formalisation, Claude was instructed to formalise the results of this paper only (i.e.\ modulo the existing results in the literature). They were reviewed and accepted by me.

I accept all responsibility for errors in any aspect of this work.

\part{Theorems}\label{part:theory}
\section{Background}
\label{sec:background}

This section collects the material the rest of the paper uses:
notation (\S\ref{ssec:notation}); the Mazur--Tate--Teitelbaum $p$-adic
$L$-function, the two of its properties used later, and the anomalous
primes (\S\ref{ssec:definitions}); the Katz measure and the
Grossencharacter of $E$ (\S\ref{ssec:katz}); the criterion of Coates,
Liang and Sujatha, of which Theorem~\ref{thm:criterion} is the
cyclotomic counterpart (\S\ref{ssec:cls}); the elliptic units the
measure is built from and the Eisenstein--Kronecker numbers that are its
polynomial moments (\S\S\ref{ssec:ellunits}--\ref{ssec:BK}); and
Gillard's theorem, the precedent for Question~\ref{q:fixed}
(\S\ref{ssec:gillardproof}).  Facts quoted from the literature are given
here as numbered displays or statements with pinpointed references, and
the sections that follow cite those rather than the sources.

\subsection{Notation}\label{ssec:notation}
Throughout, $K$ is an imaginary quadratic field with ring of integers
$\OK$, and $E/\Q$ is an elliptic curve with CM by $\OK$; since $E$ is
defined over $\Q$ its $j$-invariant is rational and generates the Hilbert
class field of $K$, so $K$ has class number one.  The letter $p$ denotes
a rational prime $\ge 5$ of good reduction for $E$ that is \emph{split}
in $K$; we fix once and for all an embedding
$\iota_p \colon \Qbar \hookrightarrow \overline{\Q}_p$ for each $p$, and
$\fp$ denotes the prime of $K$ (or of a larger number field, according
to context) above $p$ determined by $\iota_p$.  For such $p$ the curve
is ordinary, $a_p = \pi + \bar\pi$ with $p = \pi\bar\pi$ in $\OK$, and
$\alpha = \alpha_p$ denotes the unit root of $X^2 - a_pX + p$ under
$\iota_p$.  We write $\Q_\infty/\Q$ for the cyclotomic
$\Z_p$-extension, $\Gamma = \Gal(\Q_\infty/\Q)$, and fix the topological
generator $\gamma$ with $\chi_{\mathrm{cyc}}(\gamma) = 1+p$; the Iwasawa
algebra is $\Lambda = \Z_p[[\Gamma]] \cong \Z_p[[T]]$, $T = \gamma - 1$.
The Mazur--Tate--Teitelbaum $p$-adic $L$-function $L_p(E,T)$ is
normalised as in \cite{MTT,SteinWuthrich} with respect to the N\'eron
periods of $E$; its Taylor coefficients at $T=0$ are written $c_j(p)$,
and $\lambda_{\mathrm{an}}(\fp)$, $\mu_{\mathrm{an}}(\fp)$ denote its
Iwasawa invariants.  The cyclotomic $p$-adic height is that of \cite[(4.1)]{SteinWuthrich}:
for $P \in E(\Q)$ lying in the formal group at $p$ and reducing to the
identity component of the N\'eron model at every bad prime,
\[
   \hat h_p(P) \;=\; 2\log_p\bigl(e(P)/\sigma_p(P)\bigr),
\]
with $\sigma_p$ the canonical $p$-adic sigma function of \cite{MST} and
$e(P)$ the positive square root of the denominator of $x(P)$, extended to
$E(\Q)$ by $\hat h_p(P) = m^{-2}\hat h_p(mP)$; the height of
\cite[(1.1)]{MST} is $-\hat h_p/2p$.  For such $P$, with $t_P$ the
formal-group parameter, $\sigma_p(t) = t + O(t^2)$ and
$x(P) = t_P^{-2}\cdot(\text{unit})$ give $v_p(e(P)) = v_p(t_P) = v_p(\sigma_p(P))$,
so $e(P)/\sigma_p(P)$ is a $p$-adic unit and $v_p(\hat h_p(P)) \ge 1$.
We write $\Reg_\fp = \Reg_p$ for the discriminant of the associated
pairing on $E(\Q)/E(\Q)_{\mathrm{tors}}$, the cyclotomic $\fp$-adic
regulator, and $\Reg_\gamma = \Reg_\fp/\log_p(1+p)^{r}$ for its
normalisation \cite[(4.4)]{SteinWuthrich}, $r$ the Mordell--Weil rank;
the generic value of $v_\fp(\Reg_\fp)$ is $r$.  $X = X(\Q_\infty)$ denotes the Pontryagin dual of
$\Sel_{p^\infty}(E/\Q_\infty)$, a finitely generated $\Lambda$-module.
For $\beta \in \Qbar^\times$ we write $\fp \nmid \beta$ to mean
$v_\fp(\iota_p(\beta)) = 0$, the valuation normalised so that
$v_\fp(p) = 1$.  For $x \in \Q_p^\times$ one has $v_\fp(x) = v_p(x)$:
the choice of $\fp$ over $p$ carries content only for algebraic
arguments, through $\iota_p$.

\subsection{The $p$-adic $L$-function and the anomalous primes}\label{ssec:definitions}

We recall the $p$-adic $L$-function and the two of its properties used
later.  Let $E/\Q$ be an elliptic curve, let $p$ be a prime of good
ordinary reduction for $E$, let $\alpha_p$ be the unit root of
$X^2 - a_pX + p$, and let $\Omega_E = \int_{E(\mathbb{R})} |\omega_E|$ be
the real N\'eron period of $E$, $\omega_E$ its N\'eron differential, as
in \cite[\S 3.1]{SteinWuthrich}.
Then $L_p(E,T) \in \Lambda \otimes \Q$, $\Lambda = \Z_p[[T]]$, denotes the
Mazur--Tate--Teitelbaum $p$-adic $L$-function \cite[\S I.13]{MTT},
normalised with respect to the N\'eron periods as in
\S\ref{ssec:notation}; it lies in $\Lambda$ under the hypothesis of
Remark~\ref{rmk:integrality}.  Its Taylor coefficients at $T = 0$ are
written $c_j(p)$:
\begin{equation}\label{eq:Lpjets}
   L_p(E,T) \;=\; \sum_{j \ge 0} c_j(p)\, T^j .
\end{equation}
Interpolation at the trivial character gives
\begin{equation}\label{eq:interp}
   c_0(p) \;=\; L_p(E,0)
   \;=\; \bigl(1 - \alpha_p^{-1}\bigr)^{2}\,\frac{L(E,1)}{\Omega_E} .
\end{equation}
This is the Proposition of \cite[\S I.14]{MTT}, whose $p$-adic multiplier
$e_p(\alpha,\chi)$ specialises at weight $2$, trivial character and
$j = 0$ to $(1-\alpha_p^{-1})^2$.  Writing $T^\iota := (1+T)^{-1} - 1$ for
the involution induced by $\gamma \mapsto \gamma^{-1}$, the $p$-adic
functional equation reads
\begin{equation}\label{eq:padicfe}
   L_p\bigl(E,T^\iota\bigr) \;=\; w(E)\,U(T)\,L_p(E,T),
   \qquad U \in \Lambda^\times, \quad U(0) = 1 .
\end{equation}
This is \cite[\S I.17, Cor.~2]{MTT} in the form \cite[(18.3)]{MTT},
transported from the variable $s$ to $T$: the involution $s \mapsto -s$
about the central point becomes $T \mapsto T^\iota$, and the factor
$\langle N\rangle^{-s}$ becomes the unit power series $U$, so
$U(0) = 1$.  Its sign is the $\mathrm{sign}_p$ of
\cite[\S I.18]{MTT}, which differs from the sign $w(E)$ of the
classical functional equation exactly when the multiplier $e_p$
vanishes \cite[\S I.18, Prop.]{MTT}; at good ordinary $p$ one has
$e_p = (1-\alpha_p^{-1})^2 \ne 0$, so the two signs agree.

One class of primes is excluded throughout, and we name it here.

\begin{definition}[Anomalous primes]\label{def:anomalous}
Let $E/\Q$ be an elliptic curve and let $p$ be a prime of good reduction
for $E$.  Then $p$ is \emph{anomalous} for $E$ if
\[
   a_p \equiv 1 \pmod p ;
\]
equivalently, since $\#\widetilde E(\F_p) = p + 1 - a_p$, if $p$
divides $\#\widetilde E(\F_p)$.
\end{definition}

The terminology is due to Mazur \cite{Mazur}, who first set out the significance of
these primes for the Iwasawa theory of $E$; \cite[\S 4]{Greenberg}
discusses them in that setting, and \cite[\S 3.3]{SteinWuthrich} works
a numerical example at one.

Non-anomalousness enters three arguments below, always through the same
consequence, that $\#\widetilde E(\F_p)$ is prime to $p$.  It makes the
multiplier $(1-\alpha_p^{-1})^{2}$ of \eqref{eq:interp} a $p$-adic unit,
so that dividing it out changes no valuation
(Proposition~\ref{prop:normalisation}).  It makes the local term at $p$
in Mazur's control theorem vanish, since
$\widetilde E(\F_p)[p^\infty] = 0$ (Step~4 of
Theorem~\ref{prop:consequence}).  And it makes the index of $E^\circ(\Q)$
in $E(\Q)$ prime to $p$, which is what extends the regulator estimate of
Proposition~\ref{prop:dictionary} from $E^\circ(\Q)$ to all of $E(\Q)$.

For a curve with complex multiplication by $\Z[i]$ at most one prime
is anomalous.

\begin{lemma}[{Anomalous primes of a curve with CM by $\Z[i]$}]\label{lem:noanomalous}
Let $E/\Q$ be an elliptic curve with complex multiplication by $\Z[i]$
and let $p$ be a prime of good reduction for $E$.
\begin{enumerate}
\item If $p \not\equiv 1 \pmod 4$, then $p \mid a_p$; in particular $p$
      is not anomalous for $E$.
\item If $p \equiv 1 \pmod 4$, then $a_p$ is even, and $p$ is anomalous
      for $E$ if and only if $p = 5$ and $a_5 = -4$.
\end{enumerate}
\end{lemma}

\begin{proof}
Write $K = \Q(i)$.  A prime $p \not\equiv 1 \bmod 4$ is inert or
ramified in $K$, so the reduction of $E$ at $p$ is supersingular, by
Deuring's reduction criterion \cite{Deuring},
\cite[Ch.~13, \S 4]{LangEF}; and supersingular means $p \mid a_p$.
Then $a_p \equiv 1 \bmod p$ is impossible, and (1) follows from
Definition~\ref{def:anomalous}.

Let now $p \equiv 1 \bmod 4$, so that $p$ splits: $p = \pi\bar\pi$ with
$\pi = u + vi$, and $\fp = (\pi)$.  The Grossencharacter $\psi_E$ of
\S\ref{ssec:katz} has infinity type $(1,0)$, so $\psi_E(\fp)$ is a
generator of $\fp$, hence one of $\pm\pi$ and $\pm i\pi$, and
$a_p = \psi_E(\fp) + \overline{\psi_E(\fp)}$ is the trace of that
generator: one of $\pm 2u$ and $\pm 2v$.  So $a_p$ is even.

By Definition~\ref{def:anomalous}, $p$ is anomalous exactly when
$a_p = 1 + kp$ for some integer $k$.  Such a $k$ must be nonzero (to avoid contradicting $a_p$ being even) so $|a_p| \ge p - 1$; with the Hasse bound
$|a_p| \le 2\sqrt p$ this gives $p - 2\sqrt p - 1 \le 0$, hence
$\sqrt p \le 1 + \sqrt 2$ and $p \le 5$.  As $p \equiv 1 \bmod 4$, only
$p = 5$ survives, and there $a_5$ even, $a_5 \equiv 1 \bmod 5$ and
$|a_5| \le 2\sqrt 5 < 5$ leave $a_5 = -4$ alone.  Conversely
$a_5 = -4$ gives $\#\widetilde E(\F_5) = 5 + 1 + 4 = 10$, which $5$
divides.
\end{proof}

The word \emph{jet}, used throughout, means only a truncated Taylor
expansion: the $m$-th jet of a power series is the polynomial formed
by its coefficients in degrees $\le m$; it involves none of the
differential-geometric apparatus of jet bundles or jet schemes.  Since
the vocabulary is not standard in Iwasawa theory, we fix it now.

\begin{definition}[Jets]\label{def:jet}
Let $F(T) = \sum_{j \ge 0} c_j T^j$ be a power series over a $p$-adic
ring and let $m \ge 0$.  The \emph{$m$-th jet of $F$ at $T = 0$} is the
truncation $\sum_{j \le m} c_j T^j$, and $c_m$ is its \emph{$m$-th
coefficient}.  The $m$-th jet of a measure, in a specified variable, is
the $m$-th jet of its Iwasawa power series in that variable.
\end{definition}

\noindent Two instances occur in this paper.
\begin{enumerate}
\item[(i)] \emph{Jets of $L_p(E,T)$ at $T = 0$}, with coefficients
$c_j(p)$ as in \eqref{eq:Lpjets}.  When $L(E,1) = 0$ and $w(E) = +1$
one has $c_0(p) = c_1(p) = 0$ (Lemma~\ref{lem:c0c1}), so the second jet
is $c_2(p)\,T^2$; we call it the \emph{central second jet}, and the
assertion ``the second jet is a unit'' means
$\tilde c_2(p) \in \Z_p^\times$ for the normalisation $\tilde c_2$ of
Definition~\ref{def:c2tilde}.
\item[(ii)] \emph{Jets of the Katz measure} at the central character,
taken in the cyclotomic direction: the jets of the one-variable slice
of $\mu_{\ff\fp^\infty}$ through that character (\S\ref{ssec:katz}),
which agree with those in (i) up to the explicit unit factors of
Lemma~\ref{lem:comparison}.
\end{enumerate}
\subsection{The Katz measure}\label{ssec:katz}
Let $K$ be imaginary quadratic, $\ff$ an integral ideal of $\OK$, and
$p = \fp\bar\fp \ge 5$ split in $K$ and prime to $6\Nm\ff$.  Katz
\cite{Katz} constructed a measure $\mu = \mu_{\ff\fp^\infty}$ on
$G_\infty = \Gal(K(\ff p^\infty)/K)$, valued in
\[
   W \;:=\; \widehat{\Z_p^{\mathrm{ur}}} ,
\]
the completion of the maximal unramified extension of $\Z_p$, with
maximal ideal $\mathfrak{m}_W$ and residue field $\overline{\F}_p$.  It
interpolates the critical Hecke
$L$-values of $K$: for $\chi$ of conductor dividing $\ff p^\infty$ and
infinity type $(k,j)$ with $0 \le -j < k$,
\begin{equation}\label{eq:katzinterp}
   \int_{G_\infty} \chi \, d\mu
   \;=\;
   (\ast) \cdot
   \frac{\Omega_p^{\,k - j}}{\Omega_\infty^{\,k-j}}
   \cdot
   \mathcal{E}_\fp(\chi)
   \cdot
   \Gamma\text{-factor}
   \cdot
   L(\chi^{-1}, 0),
\end{equation}
with $\Omega_\infty$ a period of a fixed elliptic curve $\mathcal{A}/K$ with
complex multiplication by $\OK$,\footnote{The notation here is more
general than the sequel needs.  Katz's construction takes for
$\mathcal{A}$ any elliptic curve over $K$ with complex multiplication by
$\OK$, with no reference to $E$.  Immediately below, $\mathcal{A}$ is
fixed with lattice $\OK$-proportional to that of $E$, hence isogenous to
$E$ over $\Qbar$; for the curve of Part~\ref{part:evidence} one takes
$\mathcal{A} = E$.} so that
$\mathcal{A}(\C) = \C/\Omega_\infty\OK$;
$\Omega_p$ the corresponding $p$-adic period (a unit in
$\widehat{\Z_p^{\mathrm{ur}}}$, from a trivialisation of the formal group
of the reduction); $\mathcal{E}_\fp(\chi)$ Euler-type factors at $\fp$
and $\bar\fp$; and $(\ast)$ explicit algebraic constants.  The construction
is that of \cite{Katz}, of which \cite[II \S4]{deShalit} is an exposition;
throughout, $\mu$ is taken in the normalisation of Bannai and Kobayashi
\cite[Def.~3.8]{BannaiKobayashi}, for which
\cite[Thm.~3.11]{BannaiKobayashi} gives every factor of
\eqref{eq:katzinterp} explicitly.  That normalisation is relative to a
complex number $\Omega$ with $\ff\,\Omega = \Omega_\infty\OK$, the period
lattice of $\mathcal{A}$; we fix a generator $f_0$ of $\ff$ and take
$\Omega = \Omega_\infty/f_0$.  It assumes that $p \nmid \Nm\ff$, that
$\mathcal{A}$ has a Weierstrass model over $\OK$ good at $\fp$ whose
invariant differential has period lattice $\Omega_\infty\OK$, and that
$1$ is the only unit of $\OK$ congruent to $1$ modulo $\ff$, which holds
when $\ff$ is the conductor of a Hecke character of infinity type
$(1,0)$.  Katz's construction
does not require class number one.  Here $K$ has class number one
(\S\ref{ssec:notation}), which is what allows a single $\mathcal{A}$ and a
single $\Omega_\infty$ to serve.

The connection with an elliptic curve is through its Grossencharacter:
for $E/\Q$ with CM by $\OK$ there is a Hecke character $\psi_E$ of $K$,
of infinity type $(1,0)$ and conductor $\ff$ with
$N = |d_K| \cdot \Nm\ff$, such that $L(E/\Q, s) = L(\psi_E, s)$; this is
a theorem of Deuring, and we follow \cite[Ch.~II]{deShalit} for the
conventions attaching $\psi_E$ to $E$.  For an integer $n \ge 1$ we
write
\[
   L_\ff(\bar\psi_E^{\,n}, s)
   \;:=\;
   \sum_{(\mathfrak{b},\,\ff) = 1}
   \frac{\bar\psi_E(\mathfrak{b})^{n}}{\Nm\mathfrak{b}^{\,s}} ,
\]
the sum over the integral ideals of $\OK$ prime to $\ff$, for the Hecke
$L$-function of $\bar\psi_E^{\,n}$ without the Euler factors at the
primes dividing $\ff$; it converges for $\mathrm{Re}\,s > n/2 + 1$ and
continues to an entire function of $s$ \cite[\S1.1]{BannaiKobayashi},
and $L(\bar\psi_E^{\,n}, s)$ denotes the primitive $L$-function.  That $\psi_E$ and that $\ff$ are
the character and
conductor in force throughout
\S\S\ref{sec:jet}--\ref{sec:criterion}, and $\mathcal{A}$ is taken with
lattice $\OK$-proportional to that of $E$.  The interpolation range of
\eqref{eq:katzinterp} includes the characters adjacent to the central
character of $\psi_E$, so the Mazur--Tate--Teitelbaum $p$-adic
$L$-function of a CM curve is, up to the explicit comparison factors of
Lemma~\ref{lem:comparison}, a one-variable slice of $\mu$.

The two sides of that comparison are normalised by different periods:
$L_p(E,T)$ by the real N\'eron period $\Omega_E$ of $E$, the measure
$\mu$ by the CM period $\Omega_\infty$ of \eqref{eq:katzinterp}.  Since
$\mathcal{A}$ and $E$ have $\OK$-proportional lattices they are isogenous over
$\Qbar$, so their N\'eron differentials differ by an algebraic factor and
the two periods differ by a nonzero algebraic number.

\begin{definition}[Period ratio]\label{def:periodratio}
Let $E/\Q$ have complex multiplication by $\OK$ and let $\Omega_\infty$
be as in \eqref{eq:katzinterp}.  The \emph{period ratio} of $E$ is the
number $\omega_E \in \Qbar^\times$ with
\[
   \Omega_E \;=\; \omega_E \cdot \Omega_\infty ,
\]
and $\operatorname{supp}(\omega_E)$ is the finite set of rational primes
lying below a prime of the number field $\Q(\omega_E)$ that divides the
numerator or the denominator of $\omega_E$.
\end{definition}

\noindent The set $\operatorname{supp}(\omega_E)$ depends on $E$
and on the choice of $\Omega_\infty$ made above, both of which are fixed
once and for all.  It is the set of primes at which the comparison of
Lemma~\ref{lem:comparison} can fail to be a $p$-adic unit.

\subsection{The Coates--Liang--Sujatha criterion}
\label{ssec:cls}
The criterion this paper builds on decides $\Sha(E/\Q)[p^\infty]$ at a
split prime from the $p$-adic valuation of one critical Hecke
$L$-value.  Let $E/\Q$ have complex multiplication by $\OK$, let
$\Omega_0 \in \C$ generate the period lattice $\Omega_0\OK$ of the
N\'eron differential of $E$ (for $\mathcal{A} = E$ in \S\ref{ssec:katz},
$\Omega_0 = \Omega_\infty$), let $\Omega^+_\infty$ be the least positive
real period of $E$, and let $\alpha(E) \in \OK$ be the nonzero element
with $\Omega^+_\infty = \alpha(E)\,\Omega_0$.  For a prime $p$ put
\begin{equation}\label{eq:cpplus}
   c_p^+(E) \;:=\; (\Omega^+_\infty)^{-p}\, L\bigl(\bar\psi_E^{\,p},\, p\bigr) .
\end{equation}

\begin{theorem}[Coates--Liang--Sujatha {\cite[Thm.~2.2]{CLS1}}]\label{thm:clscriterion}
Let $E/\Q$ be an elliptic curve with complex multiplication by the
maximal order $\OK$ of an imaginary quadratic field $K$, let
$g = \rank E(\Q)$ and $r = \ord_{s=1} L(E,s)$, and let $c_p^+(E)$ be as
in \eqref{eq:cpplus}.  Then $c_p^+(E) \in \Q$.  Let $p$ be a prime such
that
\begin{enumerate}
\item $E$ has good reduction at $p$;
\item $p$ is prime to the number of roots of unity in $K$;
\item $p$ splits in $K$;
\item $p \nmid \alpha(E)$;
\item $r \equiv g \pmod 2$.
\end{enumerate}
If $\ord_p\bigl(c_p^+(E)\bigr) < g + 2$, then $\Sha(E/K)[p^\infty]$ is
finite.  If $\ord_p\bigl(c_p^+(E)\bigr) = g$, $p \nmid 6$ and
$\#\widetilde E(\F_p)$ is prime to $p$, then $\Sha(E/K)[p^\infty] = 0$.
\end{theorem}

\noindent Always $\ord_p(c_p^+(E)) \ge g$ \cite[(15)]{CLS1}, and
Coates, Liang and Sujatha call $p$ \emph{exceptional} for $E$ when the
inequality is strict.  The last hypothesis says that $p$ is not
anomalous (Definition~\ref{def:anomalous}).  For odd $p$ the restriction
map $\Sha(E/\Q)[p^\infty] \to \Sha(E/K)[p^\infty]$ is injective, its
kernel being killed by $[K:\Q] = 2$, so the conclusion gives
$\Sha(E/\Q)[p^\infty] = 0$.  The proof uses Rubin's main conjecture for
$K$ \cite[Thm.~4.1]{Rubin}.  The criterion is decidable: $c_p^+(E)$ is,
up to an explicit factor, a trace from a ray class field of $K$ of a
division value of a derivative of $\wp$, and $\ord_p(c_p^+(E))$ is
settled by exact integer arithmetic \cite[\S 4]{CLS2}.  The value tested
is $L(\bar\psi_E^{\,p}, p)$, at the end $s = p$ of the critical strip
$1 \le s \le p$ of $\bar\psi_E^{\,p}$, and its weight moves with $p$.
Theorem~\ref{thm:criterion} is the cyclotomic counterpart: it tests three
critical values of $\bar\psi_E^{\,2p-1}$ (Corollary~\ref{cor:bracketL})
and decides the cyclotomic regulator as well as $\Sha$.

\subsection{Elliptic units and the rational-function structure}
\label{ssec:ellunits}
The measure is constructed from elliptic units.  Fix a CM elliptic curve
$\mathcal{A}/K$ with $\mathcal{A}(\C) = \C/\Gamma$ and
$\Gamma = \Omega_\infty\OK$ (the N\'eron differential giving the
identification), and $\mathfrak{a}$ an auxiliary
ideal prime to $6\ff p$.  The classical theta quotient
\[
   \Theta_{\mathfrak{a}}(z) \;=\;
   \frac{\Delta(\Gamma)}{\Delta(\mathfrak{a}^{-1}\Gamma)}
   \prod_{0 \ne t \in \mathfrak{a}^{-1}\Gamma/\Gamma}
   \bigl(\wp(z;\Gamma) - \wp(t;\Gamma)\bigr)^{-1}
\]
is a \emph{rational function} on $\mathcal{A}$, defined over $K$, whose
values at
torsion points of order prime to $\mathfrak{a}$ are elliptic units,
norm-coherent along the $\fp$-power tower.  Kummer theory and Coleman
calculus convert the coherent system
$\{\Theta_{\mathfrak{a}}(\text{$\ff\fp^n$-torsion})\}_n$ into $\mu$: on
each branch the Amice transform of $\mu$ is the expansion of
$\partial\log \Theta_{\mathfrak{a}}$ along the formal group at $\fp$, up
to the period trivialisation and translation by $\ff$-torsion
\cite[Ch.~II]{deShalit}.

\subsection{The algebraic theta function and Eisenstein--Kronecker
numbers}\label{ssec:BK}
Bannai--Kobayashi \cite{BannaiKobayashi} put this structure in a form
that gives the higher jets as well as the values; it is the form
\S\ref{sec:jet} uses.  For a lattice $\Gamma \subset \C$ with
$\C/\Gamma$ CM, put $A(\Gamma) := \mathrm{area}(\C/\Gamma)/\pi$ (an
invariant of the lattice; the CM curve of \S\ref{ssec:katz} is written
$\mathcal{A}$) and
$\langle z, w\rangle_\Gamma := \exp\bigl((z\bar w - w\bar z)/A(\Gamma)\bigr)$;
for $z_0, w_0 \in \C$ and integers $a \ge 0$, $b \ge 1$ the
Eisenstein--Kronecker--Lerch series
\[
   K^*_a(z_0, w_0, s; \Gamma)
   \;=\;
   \sideset{}{'}\sum_{\gamma \in \Gamma}
   \frac{(\bar z_0 + \bar\gamma)^a}{|z_0 + \gamma|^{2s}}
   \,\langle \gamma, w_0\rangle_\Gamma
   \qquad (\mathrm{Re}\, s > a/2 + 1)
\]
continues meromorphically in $s$.  Following \cite[\S1]{BannaiKobayashi}
we set
\[e^*_{a,b}(z_0, w_0; \Gamma) := K^*_{a+b}(z_0, w_0, b; \Gamma).\]
Two properties of the series are used below.  The primed sum omits
$\gamma = -z_0$ when $z_0 \in \Gamma$, so for $w_0 = 0$ the series, and
hence $e^*_{a,b}(z_0, 0; \Gamma)$, depends on $z_0$ only modulo
$\Gamma$.  And for $\lambda \in \C^\times$ the substitution
$\gamma \mapsto \lambda\gamma$ multiplies the numerator by
$\bar\lambda^{\,a+b}$ and the denominator, at $s = b$, by
$|\lambda|^{2b}$, while $A(\lambda\Gamma) = |\lambda|^{2}A(\Gamma)$
leaves the pairing unchanged; so, first where the series converges and
then by continuation,
\begin{equation}\label{eq:EKhomogeneity}
   e^*_{a,b}(\lambda z_0,\, \lambda w_0;\, \lambda\Gamma)
   \;=\;
   \bar\lambda^{\,a}\,\lambda^{-b}\,
   e^*_{a,b}(z_0,\, w_0;\, \Gamma)
   \qquad (\lambda \in \C^\times) .
\end{equation}
At torsion parameters these \emph{Eisenstein--Kronecker numbers} are
algebraic after division by the appropriate power of $A(\Gamma)$
(Damerell's theorem; \cite[Cors.~2.11--2.12]{BannaiKobayashi} in this
normalisation), and they exhaust the critical Hecke $L$-values of $K$:
every $L(\chi^{-1},0)$ in \eqref{eq:katzinterp} is an explicit finite
$\Qbar$-linear combination of them, at parameters determined by the
conductor of $\chi$.  Three of the structural theorems of
\cite{BannaiKobayashi} are used below, in the following forms; the
constants are left unspecified here and are made explicit where they
are used, in Lemmas~\ref{lem:comparison} and~\ref{lem:moments}.

\begin{enumerate}
\item[(BK1)] (\emph{Generating function.})  The Kronecker theta
function $\Theta(z, w) = \theta(z+w)/\theta(z)\theta(w)$ --- $\theta$
the reduced theta function of $(\C/\Gamma, dz)$, normalised as in
\cite[\S1.2]{BannaiKobayashi} --- is a meromorphic section whose Laurent
expansion at any pair of torsion points $(z_0, w_0)$ has coefficients
the $e^*_{a,b}(z_0, w_0)$, up to explicit factorials and periods; all
Eisenstein--Kronecker numbers, of all indices, are Taylor coefficients
of this \emph{one} two-variable function
\cite[Thms.~1.13, 1.17]{BannaiKobayashi}.
\item[(BK2)] (\emph{Algebraicity.})  $\Theta$, suitably translated and
trivialised, is defined over $\Qbar$: its expansion in algebraic
formal-group parameters at torsion points has algebraic coefficients
with bounded denominators \cite[Thm.~2.9, Cor.~2.11]{BannaiKobayashi}.
\item[(BK3)] (\emph{$p$-adic interpolation at ordinary primes.})  For
$\fp$ split (ordinary), the pullback of the translated $\Theta$ to
$\widehat{\mathcal{A}} \times \widehat{\mathcal{A}}$ at $\fp$, expanded in
the parameters
$(S, S')$ of the trivialised formal group, has $\fp$-integral
coefficients, and the measure on $\Z_p \times \Z_p$ attached to it by
Amice--Mahler is, up to the explicit torsion-translation bookkeeping,
the Katz measure $\mu$.  In particular the \emph{polynomial moments}
satisfy
\[
   \int_{\Z_p\times\Z_p} x^a y^b \, d\mu_{(z_0,w_0)}
   \;=\;
   c_{a,b}\cdot
   \Omega_p^{\,\bullet}\cdot
   \frac{e^*_{a', b'}(z_0, w_0)}{A(\Gamma)^{\bullet}\,
   \Omega_\infty^{\,\bullet}} ,
\]
with $c_{a,b}$ an explicit product of factorials, signs and unit Euler
factors, $(a',b')$ an explicit re-indexing of $(a,b)$, and explicit
exponents $\bullet$ \cite[Prop.~3.3, Thm.~3.7]{BannaiKobayashi}.
\end{enumerate}
The content of (BK3) is that the polynomial moments of $\mu$ are
Eisenstein--Kronecker numbers, $\fp$-integral once divided by the period
trivialisation.  Indexing conventions differ among references, and we
fix ours here.  Following \cite[Prop.~3.3]{BannaiKobayashi}, which
matches a moment of total degree $d$ with index sum $a + b = d+1$, we
call $a+b$ the \emph{weight} of $e^*_{a,b}$; the moment of bidegree
$(k,l)$ pairs with $e^*_{l,\,k+1}$.
\subsection{Gillard's theorem and proof}
\label{ssec:gillardproof}
Gillard's theorem is the precedent for Question~\ref{q:fixed}: a
single statement about the Katz measure that holds at every split
prime.  Its proof is discussed in \S\ref{ssec:gillardmech}.

\begin{theorem}[Gillard {\cite[Th.~2.9]{Gillard}}]\label{thm:gillard}
Let $K$ be an imaginary quadratic field, let $\ff$ be an integral ideal
of $\OK$, and let $p = \fp\bar\fp \ge 5$ be a prime split in $K$ and
prime to $6\Nm\ff$.  Let $\mu_{\ff\fp^\infty}$ be the Katz measure
\cite{Katz} on $G_\infty = \Gal(K(\ff p^\infty)/K)$: the
$\widehat{\Z_p^{\mathrm{ur}}}$-valued measure interpolating the algebraic
parts of the critical Hecke $L$-values $L(\chi^{-1},0)$ of $K$, over the
characters $\chi$ of conductor dividing $\ff p^\infty$ and infinity type
$(k,j)$ with $0 \le -j < k$.  Then the $\mu$-invariant of every
branch of $\mu_{\ff\fp^\infty}$ vanishes: each of the finitely many power
series obtained by restricting $\mu_{\ff\fp^\infty}$ to a character of the
torsion subgroup of $G_\infty$ has a coefficient that is a $p$-adic unit.
\end{theorem}

\noindent For $K$ of class number one the measure of the theorem is the
$\mu$ of \eqref{eq:katzinterp}.  The same result was obtained independently by Schneps \cite{Schneps}, and both
proofs descend from Ferrero--Washington \cite{FerreroWashington} and Sinnott
\cite{Sinnott} for Dirichlet $L$-functions.  The $\mu$-invariants of $p$-adic
Hecke $L$-functions have since been determined in wider settings by Hida
\cite{Hida}, Finis \cite{Finis} and Hsieh \cite{Hsieh}, and at primes inert in
$K$ by Burungale--He--Kobayashi--Ota \cite{BHKO}; each of these determines a
$\mu$-invariant.

\section{Proof of Theorem~\ref{intro:consequence}}\label{sec:thmA}

Throughout this section $E/\Q$ has CM by the maximal order $\OK$ of an
imaginary quadratic field $K$, with
\[\rank E(\Q) = 2, \qquad L(E,1) = 0, \qquad w(E) = +1,\]
and $p = \fp\bar\fp \ge 5$ is a good prime for $E$ that splits in
$\OK$; every numbered statement below restates the hypotheses it uses.  Such a $p$ is ordinary, by Deuring's reduction criterion
\cite{Deuring}, \cite[Ch.~13, \S 4]{LangEF} (which says that for a curve with CM by
$\OK$ and good reduction at $p$, the reduction is ordinary if $p$
splits in $K$ and supersingular if $p$ is inert or ramified).

We write $L_p(E,T) \in \Lambda \otimes \Q$ for the
Mazur--Tate--Teitelbaum $p$-adic $L$-function of $E$ with $j$-th Taylor coefficient $c_j(p)$ at $T = 0$; the two properties of
$L_p$ used in this section are the interpolation formula
\eqref{eq:interp} and the functional equation \eqref{eq:padicfe}.

\subsection{The normalised second jet}\label{ssec:jet}

\begin{lemma}\label{lem:c0c1}
Assume $L(E,1) = 0$ and $w(E) = +1$.  Then for every good ordinary $p$,
\[
   c_0(p) = c_1(p) = 0 .
\]
\end{lemma}

\begin{proof}
By \eqref{eq:interp}, $c_0(p) = (1-\alpha_p^{-1})^2 L(E,1)/\Omega_E$,
which vanishes because $L(E,1) = 0$.  So
$L_p(E,T) = c_1T + c_2T^2 + \cdots$, and it remains to show
$c_1 = 0$.  Substitute $T^\iota = (1+T)^{-1} - 1 = -T + T^2 - \cdots$
into this expansion: every term $c_jT^j$ with $j \ge 2$ contributes
only in degrees $\ge 2$, so
\[
   L_p\bigl(E,T^\iota\bigr) \;=\; -c_1T + O(T^2).
\]
On the other side of \eqref{eq:padicfe}, $w(E) = +1$ and $U(0) = 1$
give
\[
   w(E)\,U(T)\,L_p(E,T) \;=\; c_1T + O(T^2).
\]
Comparing the coefficients of $T$ yields $-c_1 = c_1$, and $2$ is
invertible, so $c_1 = 0$.
\end{proof}

The $p$-adic Birch--Swinnerton-Dyer conjecture of
Mazur--Tate--Teitelbaum \cite[Ch.~II, \S 10]{MTT} predicts
\begin{equation}\label{eq:padicbsd}
   c_2(p)
   \;=\;
   \bigl(1 - \alpha_p^{-1}\bigr)^{2} \cdot
   \frac{\Reg_p}{\log_p(1+p)^{2}} \cdot
   \#\Sha(E/\Q)[p^\infty] \cdot
   \frac{\prod_v c_v}{(\#E(\Q)_{\mathrm{tors}})^{2}}.
\end{equation}
in the form given by \cite[Conj.~5.1]{SteinWuthrich}, whose
normalisations of $L_p$ and of the cyclotomic $p$-adic regulator are
the ones fixed in \S\ref{ssec:notation}; their conjecture carries the
full order $\#\Sha(E/\Q)$, whose prime-to-$p$ part is a $p$-adic unit,
and the display is its $p$-primary form.  Of the factors on the right
only $\Reg_p$ and $\#\Sha(E/\Q)[p^\infty]$
concern us; the others --- a $p$-adic multiplier, the Tamagawa
numbers, the torsion --- have nothing to do with $\Sha$, and we divide
them out once and for all.

\begin{definition}\label{def:c2tilde}
For $p$ as above, with $\alpha_p$ the unit root of $X^2 - a_pX + p$,
set
\[
   \tilde c_2(p)
   \;:=\;
   c_2(p)\cdot
   \bigl(1 - \alpha_p^{-1}\bigr)^{-2}\cdot
   \frac{(\#E(\Q)_{\mathrm{tors}})^2}{\prod_v c_v} .
\]
\end{definition}

At a split prime of good reduction that is not anomalous
(Definition~\ref{def:anomalous}) and does not
divide the Tamagawa or torsion terms, the factors just divided out are
themselves $p$-adic units, so the normalisation does not change the valuation.

\begin{proposition}[Normalisation]\label{prop:normalisation}
Let $p = \fp\bar\fp$ be a split prime of good reduction with
$p \nmid \prod_v c_v \cdot \#E(\Q)_{\mathrm{tors}}$, and suppose $p$ is
not anomalous.  Then
\[
   v_\fp\bigl(\tilde c_2(p)\bigr) \;=\; v_\fp\bigl(c_2(p)\bigr) ;
\]
in particular $\tilde c_2(p) \in \Z_p^\times$ if and only if
$v_\fp\bigl(c_2(p)\bigr) = 0$.
\end{proposition}

\begin{proof}
By Definition~\ref{def:c2tilde}, $\tilde c_2(p)$ is $c_2(p)$ times
$(\#E(\Q)_{\mathrm{tors}})^2/\prod_v c_v$ and
$(1-\alpha_p^{-1})^{-2}$.  The first factor is a $p$-adic unit by
hypothesis.  For the second, let $\beta_p$ be the second root of
$X^2 - a_pX + p$, so that $\alpha_p\beta_p = p$ with $\alpha_p$ a unit;
then $v_p(\beta_p) = 1$ and $1 - \beta_p \in \Z_p^\times$.  Since
$\#\widetilde E(\F_p) = p + 1 - a_p = (1-\alpha_p)(1-\beta_p)$ and
$\alpha_p$ is a unit,
\[
   v_p\bigl(1 - \alpha_p^{-1}\bigr) \;=\; v_p(\alpha_p - 1)
   \;=\; v_p\bigl(\#\widetilde E(\F_p)\bigr),
\]
which is $0$ because $p$ is not anomalous.  So
$(1-\alpha_p^{-1})^{-2} \in \Z_p^\times$ as well, and $\tilde c_2(p)$
and $c_2(p)$ differ by a $p$-adic unit.  The last assertion follows
because an element of $\Q_p^\times$ of valuation $0$ lies in
$\Z_p^\times$.
\end{proof}

With this normalisation \eqref{eq:padicbsd} reads
\[\tilde c_2(p) = \Reg_\gamma \cdot \#\Sha(E/\Q)[p^\infty] .\]
The generic valuation of $\Reg_p$ is $2$, so the generic expectation is
$\tilde c_2(p) \in \Z_p^\times$: unit regulator, trivial $\Sha$.  For CM
curves at split $p \notin S_E$ (the set \eqref{eq:Sexc} below) this
identity is a theorem up to $p$-adic
units (Proposition~\ref{prop:padicbsdunits}).

\begin{remark}\label{rmk:integrality}
Integrality of $c_2(p)$, hence of $\tilde c_2(p)$ up to the finitely
many bad factors, holds for all split $p$ outside an explicit finite
set.  It follows from integrality of the modular symbols of $E$, which
holds at every good ordinary $p > 2$ for which $\bar\rho_{E,p}$ is
irreducible: this is \cite[Prop.~3.7]{GreenbergVatsal} with trivial
twisting character; the elliptic-curve form is
\cite[Prop.~3.7]{SteinWuthrich}, whose ambient \S 3 assumes $E$ non-CM,
a hypothesis the underlying result does not carry.  And
$\bar\rho_{E,p}$ is reducible exactly when $E$ admits a rational
$p$-isogeny, which happens for only finitely many $p$ \cite{Mazur78}.
It is used in Step~1 of the proof of Theorem~\ref{prop:consequence}.
\end{remark}

\subsection{The two comparisons}\label{ssec:comparisons}

The next lemma compares the two $p$-adic $L$-functions in play: the
Mazur--Tate--Teitelbaum $L_p(E,T)$ of \eqref{eq:Lpjets}, normalised by
the N\'eron period, and the cyclotomic slice of the Katz measure of
\S\ref{ssec:katz}, normalised by $\Omega_\infty$.  It names the finite set
of primes at which the two can differ by more than a $p$-adic unit, which
is one of the five constituents of the excluded set \eqref{eq:Sexc}
defined in the next subsection.

\begin{lemma}[Comparison of $L_p$ with the Katz branch]\label{lem:comparison}
Let $E/\Q$ have CM by $\OK$ and conductor $N$, let
$p = \fp\bar\fp \ge 5$ be a good split prime, and let
$L^{\mathrm{Katz}}_\fp(T)$ be the restriction of the Katz measure
$\mu_{\ff\fp^\infty}$ of \S\ref{ssec:katz} to the cyclotomic line
through the central character of $\psi_E$.  Then
\[
   L_p(E, T) \;=\; c_p \cdot u(T)\cdot L^{\mathrm{Katz}}_\fp(T),
   \qquad c_p \;=\; \bigl(f_0\,\omega_E\,\Omega_p\bigr)^{-1},
\]
where $f_0$ is the generator of $\ff$ fixed in \S\ref{ssec:katz},
$\omega_E$ is the period ratio of Definition~\textup{\ref{def:periodratio}},
$\Omega_p \in W^\times$ is the $p$-adic period of \eqref{eq:katzinterp},
$W = \widehat{\Z_p^{\mathrm{ur}}}$, and $u(T) \in \Lambda^\times$ is a unit
power series with $u(0) = 1$.
Moreover $v_p(c_p) \;=\; 0$ for every split $p \notin S_{\mathrm{cmp}}$, where
\[
   S_{\mathrm{cmp}} \;:=\;
   \{\, p : p \mid 6N\,\Nm\ff \,\} \cup \operatorname{supp}(\omega_E)
\]
is a finite set ($\omega_E$ being the period ratio of
Definition~\textup{\ref{def:periodratio}}).
\end{lemma}

\begin{proof}
Let $\chi$ be a character of $\Gamma$ of finite order $p^{n} > 1$, viewed
also as a Dirichlet character of conductor $p^{n+1}$ and, through the
norm, as a character of $G_\infty$; and for a compatible system $\zeta$
of $p$-power roots of unity in $\overline{\Q}_p$ let $\tau_\zeta(\bar\chi)$
be the Gauss sum of $\bar\chi$ formed with it.  Two interpolation formulae
are in play.  On the Katz side, \cite[Cor.~3.12]{BannaiKobayashi}, with
$\Omega = \Omega_\infty/f_0$ the number of \S\ref{ssec:katz}, gives
\[
   L^{\mathrm{Katz}}_\fp(\chi)
   \;=\; \int_{G_\infty} \hat\psi\cdot(\chi\circ\Nm)\, d\mu
   \;=\; \Omega_p\,\frac{p^{n+1}}{\alpha_p^{\,n+1}}\,
   \frac{L(E,\bar\chi,1)}{\tau_\zeta(\bar\chi)\,\Omega}
\]
for the system $\zeta$ fixed there, where $\hat\psi$ is the $p$-adic
avatar of $\psi_E$, the continuous character of $G_\infty$ with
$\hat\psi(\sigma_{\mathfrak a}) = \psi_E(\mathfrak a)$ for every ideal
$\mathfrak a$ prime to $\ff p$, so that the left-hand side is the value
at $\chi$ of the restriction named in the statement (\S\ref{sec:criterion},
\eqref{eq:katzbranch}); the corollary is stated with a real
period $\Omega_E^+$ on both sides of its identity, which cancels, and its
proof specialises \cite[Thm.~3.11]{BannaiKobayashi} to $\psi_E\chi$ and
uses no hypothesis on the period ratio.  On the Mazur--Tate--Teitelbaum
side, \cite[\S I.14]{MTT} in the form \cite[\S 3.2]{SteinWuthrich} gives
\[
   L_p(E,\chi) \;=\; \frac{p^{n+1}}{\alpha_p^{\,n+1}}\,
   \frac{L(E,\bar\chi,1)}{\tau_{\zeta'}(\bar\chi)\,\Omega_E}
\]
for the system $\zeta'$ obtained from $e^{2\pi i/p^{n+1}}$ through $\iota_p$.
Two compatible systems differ by a fixed $b \in \Z_p^\times$,
$\zeta' = \zeta^{\,b}$, and then $\tau_{\zeta'}(\bar\chi) = \chi(b)\,\tau_\zeta(\bar\chi)$.
Since $\Omega_E = \omega_E\Omega_\infty$,
\[
   L_p(E,\chi) \;=\; \bigl(f_0\,\omega_E\,\Omega_p\bigr)^{-1}\,\chi(b)^{-1}\,
   L^{\mathrm{Katz}}_\fp(\chi)
\]
at every such $\chi$.  Both sides are bounded measures on $\Gamma$, hence
determined by their values at infinitely many characters
\cite[\S 7.6]{MSD}, and $\chi \mapsto \chi(b)^{-1}$ is the value at $\chi$
of the unit power series $u(T) = (1+T)^{-m(b)}$,
$m(b) = \log_p b/\log_p(1+p) \in \Z_p$, which has $u(0) = 1$.  This proves
the identity with $c_p = (f_0\omega_E\Omega_p)^{-1}$.  As $\Omega_p$ is a
unit of $W$, $v_p(c_p) = 0$ if and only if $f_0\omega_E$ is a $p$-adic
unit; $f_0$ generates $\ff$, so it is one for $p \nmid \Nm\ff$, and
$\omega_E$ is one for $p \notin \operatorname{supp}(\omega_E)$.  Hence
$v_p(c_p) = 0$ off $S_{\mathrm{cmp}}$.
\end{proof}

Rubin's form of the cyclotomic main conjecture, quoted in Step 2 of
Theorem~\ref{prop:consequence} below, is stated for the $p$-adic
$L$-series that Mazur and Swinnerton-Dyer construct from the modular
symbols of a Weil parametrisation \cite[\S\S 8--9]{MSD}, while every
$L$-function in this paper is normalised as in \cite{MTT,SteinWuthrich}
(\S\ref{ssec:notation}).  The two constructions differ by a fixed
rational constant, because the $p$-adic $L$-functions of isogenous curves
differ by a constant \cite[\S 3]{GreenbergVatsal}, so their ideals
in $\Lambda$ could differ at finitely many $p$.  The next lemma
identifies the constant and shows that the two ideals agree at every
good ordinary $p > 2$ at which $\bar\rho_{E,p}$ is irreducible.  The
isogeny $\pi$ it requires exists whenever $\bar\rho_{E,p}$ is
irreducible (since a rational $p$-isogeny anywhere in the isogeny class
would make the semisimplification of $\bar\rho_{E,p}$, hence
$\bar\rho_{E,p}$ itself, reducible); when $E$ is itself the strong Weil
curve of its isogeny class, $\pi$ is the identity.

\begin{lemma}[Comparison of the \cite{MSD} and \cite{MTT}
normalisations]\label{lem:msdmtt}
Let $E/\Q$ be an elliptic curve of conductor $N$, and let $p > 2$ be a
prime of good ordinary reduction for $E$ at which $\bar\rho_{E,p}$ is
irreducible.  Let $E_0$ be the strong Weil curve of the isogeny class
of $E$, let $\varphi_0 \colon X_0(N) \to E_0$ be its optimal
parametrisation, let $\pi \colon E_0 \to E$ be an isogeny of degree
prime to $p$, and let $\mathcal{L}_{\mathrm{MSD}} \in \Lambda$ be the
$p$-adic $L$-series of \cite[\S\S 8--9]{MSD} attached to the
parametrisation $\pi \circ \varphi_0$ of $E$ and to the unit root
$\alpha_p$.  Then
\[
   \mathcal{L}_{\mathrm{MSD}}\,\Lambda \;=\; \bigl(L_p(E,T)\bigr) .
\]
\end{lemma}

\begin{proof}
Write $\varphi = \pi \circ \varphi_0$ and let $f$ be the newform of
$E$.  The series $\mathcal{L}_{\mathrm{MSD}}$ is $c_\varphi$ times the
Mellin transform of the measure that \cite[\S 8]{MSD} attaches to the
$T_p$-eigenfunction $r \mapsto \varphi_*\{r\}^+ \in H_1(E,\Z)^+ \cong
\Z$, where $\{r\}$ is the modular-symbol path from $0$ to $r$ and
$c_\varphi$ is the constant with $\varphi^*\omega_E = c_\varphi \cdot 2\pi i\,f\,dz$
\cite[(2.1), (9.1)]{MSD}.  For $r = a/p^{n+1}$ the cusp $r$ is
$\Gamma_0(N)$-equivalent to $0$, its denominator being prime to $N$, so
$\{r\}$ is a closed cycle on $X_0(N)$ and $\varphi_*\{r\}$ lies in
$H_1(E,\Z)$ \cite[\S 6, (4)]{MSD}; the Riemann sums of the measure are
therefore integers divided by powers of the unit $\alpha_p$
\cite[\S 8.1, Lemma~1]{MSD}, and $\mathcal{L}_{\mathrm{MSD}} \in \Lambda$.  The
series $L_p(E,T)$ is built in the same way from
$[r]^+ = \lambda^+(r)/\Omega_E$, the period integrals of
$2\pi i\,f\,dz$ along paths from $i\infty$, normalised by the real
N\'eron period (\S\ref{ssec:notation}).

An Iwasawa function is determined by its values at the nontrivial
finite-order characters \cite[\S 7.6]{MSD}, and on both sides the
value at a character $\chi$ of conductor $p^{n+1}$ is one and the same
power of $\alpha_p^{-1}$ times the sum
$\sum_a \chi^{-1}(a)\cdot(\text{symbol at } a/p^{n+1})$
(\cite[\S\S 8.3, 9.7]{MSD}; \cite[\S I.13]{MTT}).  The two symbol
systems are proportional up to an additive constant: the paths differ
by the fixed arc $\{0 \to i\infty\}$, whose contribution to the sum
vanishes because $\sum_a \chi^{-1}(a) = 0$ for nontrivial $\chi$; and
pairing $\varphi_*\{r\}^+$ with the N\'eron differential $\omega_E$
gives
$\varphi_*\{r\}^+ = c_\varphi\,2^{a}\,[r]^+ + \text{constant}$ for a
fixed $a \in \Z$ depending only on the real components of
$E(\mathbb{R})$.  Comparing values at all nontrivial $\chi$, and
recalling the factor $c_\varphi$ in $\mathcal{L}_{\mathrm{MSD}}$,
\[
   \mathcal{L}_{\mathrm{MSD}} \;=\; u \cdot c_\varphi^{\,2} \cdot L_p(E,T),
\]
with $u$ a $p$-adic unit (a power of $2$, a sign, and a power of
$\alpha_p$), so the two ideals agree if and only if
$v_p(c_\varphi) = 0$.

To show $v_p(c_\varphi) = 0$: write $c_\varphi = c_{\varphi_0}\,c_\pi$
with $\pi^*\omega_E = c_\pi\,\omega_{E_0}$.  Mazur's theorem
\cite[Cor.~4.1]{Mazur78} gives
$p \mid c_{\varphi_0} \Rightarrow p^2 \mid 4N$, and $p$ is odd of good
reduction, so $p \nmid c_{\varphi_0}$.  Both curves have good reduction
at $p$ and $p \nmid \deg\pi$, so $\pi$ induces an isomorphism on formal
groups over $\Z_p$ and carries N\'eron differentials to unit multiples:
$v_p(c_\pi) = 0$.
\end{proof}

\noindent Any parametrisation of $E$ is $\pi' \circ \varphi_0$ for an
isogeny $\pi' \colon E_0 \to E$, and multiplies
$\mathcal{L}_{\mathrm{MSD}}$ by a nonzero rational constant, prime to
$p$ whenever $\deg\pi'$ is; under the hypotheses of the lemma the ideal
$\mathcal{L}_{\mathrm{MSD}}\Lambda$ is therefore the same for every
parametrisation whose degree ratio to $\varphi_0$ is prime to $p$.  The
lemma also gives $L_p(E,T) \in \Lambda$ under its hypotheses, since
$\mathcal{L}_{\mathrm{MSD}} \in \Lambda$ and $c_\varphi$ is a $p$-adic unit
there; this is Remark~\ref{rmk:integrality} again.

\subsection{The excluded set}\label{ssec:excluded}

We define
\begin{equation}\label{eq:Sexc}
   S_E \;:=\; S_{\mathrm{bad}} \cup S_{\mathrm{an}} \cup S_{\mathrm{red}}
   \cup S_{\mathrm{cmp}} \cup S_{\mathrm{cl}},
\end{equation}
where
\[
\begin{aligned}
   S_{\mathrm{bad}} &\;:=\; \bigl\{\, p \;:\;
      p \mid 6N \cdot {\textstyle\prod_v c_v} \cdot
      \#E(\Q)_{\mathrm{tors}} \cdot d_K \,\bigr\}, \\
   S_{\mathrm{an}}  &\;:=\; \bigl\{\, p \;:\;
      p \text{ is anomalous for } E \,\bigr\}, \\
   S_{\mathrm{red}} &\;:=\; \bigl\{\, p \;:\;
      \bar\rho_{E,p} \text{ is reducible} \,\bigr\}, \\
   S_{\mathrm{cmp}} &\;:=\; \text{the set of Lemma~\ref{lem:comparison}}, \\
   S_{\mathrm{cl}}  &\;:=\; \bigl\{\, p \;:\;
      p \mid \#\mathrm{Cl}_\ff(K) \,\bigr\}.
\end{aligned}
\]
`Anomalous' is as in Definition~\ref{def:anomalous} and
$\mathrm{Cl}_\ff(K)$ is the ray class group of $K$ modulo $\ff$. It depends only on $E$ and on the
choice of $\Omega_\infty$ fixed in \S\ref{ssec:katz}

For a CM curve over $\Q$ the sets $S_{\mathrm{bad}}$,
$S_{\mathrm{red}}$, $S_{\mathrm{cmp}}$ and $S_{\mathrm{cl}}$ are all
finite; only $S_{\mathrm{an}}$ can fail to be.
Whether $S_E$ contains a split prime at all is a question about the
individual curve.

\begin{proposition}[The anomalous primes]\label{prop:anomalous}
Let $E/\Q$ have CM by the maximal order $\OK$ of an imaginary quadratic
field $K$, let $S_{\mathrm{an}}$ be the set of primes of good reduction
that are anomalous for $E$, and let $S_E$ be the set \eqref{eq:Sexc}.
\begin{enumerate}
\item Every $p \in S_{\mathrm{an}}$ with $p > 5$ is split in $K$, has
      $a_p = 1$, and satisfies $4p - 1 = |d_K|\,f^2$ for some odd
      integer $f$.  Hence fewer than $\sqrt{4x/|d_K|} + 3$ elements of
      $S_{\mathrm{an}}$ lie below $x$, and $S_E$ has density zero among
      the primes.
\item If $|d_K| \not\equiv 3 \pmod 8$, that is if
      $|d_K| \in \{4, 7, 8\}$, then $S_{\mathrm{an}} \subseteq \{2, 3, 5\}$
      and $S_E$ is finite.
\item If the Lang--Trotter conjecture \cite{LangTrotter} holds for $E$
      with a nonzero constant at the value $a_p = 1$, then
      $S_{\mathrm{an}}$ is infinite; by \textup{(2)} this requires
      $|d_K| \equiv 3 \pmod 8$.
\end{enumerate}
\end{proposition}

\begin{proof}
(1) Let $p > 5$ be anomalous.  If $p$ is inert or ramified in $K$, the
reduction is supersingular by Deuring's criterion, recalled at the start
of this section, so $p \mid a_p$, and $a_p \equiv 1 \bmod p$ is
impossible.  So $p = \fp\bar\fp$ splits, and
$a_p = \psi_E(\fp) + \overline{\psi_E(\fp)}$ with $\psi_E(\fp)$ a
generator of $\fp$, as $K$ has class number one.  Write $a_p = 1 + kp$.
The Hasse bound $|a_p| \le 2\sqrt p$ excludes $k \ge 1$ and $k \le -2$
for every $p$, and $k = -1$ unless $p \le 5$; so $a_p = 1$.  Then
$\psi_E(\fp)$ is a root of $x^2 - x + p$, whose discriminant $1 - 4p$
is negative, so $\Q(\psi_E(\fp)) = K$ and $1 - 4p = d_K f^2$ with $f$
the conductor of the order $\Z[\psi_E(\fp)]$; that is
$4p - 1 = |d_K| f^2$.  As $4p - 1 \equiv 3 \bmod 4$, $f$ is odd.  Each
$f$ determines at most one $p$, and $p \le x$ forces
$f < \sqrt{4x/|d_K|}$, which with $2$, $3$ and $5$ gives the count.
The other four constituents of \eqref{eq:Sexc} are finite, so $S_E$
too has fewer than a constant multiple of $\sqrt x$ elements below $x$,
while the number of primes below $x$ exceeds a constant multiple of
$x/\log x$ (Chebyshev's bound); hence
$\#\{p \in S_E : p \le x\}/\pi(x) \to 0$, which is density zero among
the primes.

(2) With $f$ odd, $f^2 \equiv 1 \bmod 8$, so $|d_K| \equiv 4p - 1
\equiv 3 \bmod 8$ for every $p > 5$ in $S_{\mathrm{an}}$.  Among the
class-number-one discriminants this fails exactly for
$|d_K| \in \{4, 7, 8\}$, and then $S_{\mathrm{an}} \subseteq \{2, 3, 5\}$
and $S_E$ is finite.

(3) The conjecture asserts that the number of primes $p \le x$ with
$a_p = 1$ is asymptotic to $C\sqrt x/\log x$ for a constant
$C = C_{E,1} \ge 0$ attached to $E$.  If $C > 0$ there are infinitely
many such $p$, and each of them is anomalous, so $S_{\mathrm{an}}$ is
infinite; by (2), $|d_K| \equiv 3 \bmod 8$.
\end{proof}

\noindent Part~(2) contains Lemma~\ref{lem:noanomalous} for CM by
$\Z[i]$, and the split primes outside $S_E$ have density $\tfrac12$ in
every case.  For the six remaining class-number-one fields,
$|d_K| \equiv 3 \bmod 8$, anomalous primes do occur: the curve
$y^2 = x^3 + 2$, with CM by $\Z[\zeta_3]$, has $a_p = 1$ at $p = 61$,
$331$, $547$, $2437$ and $3571$, and the Lang--Trotter conjecture
predicts infinitely many for that curve; no case of the conjecture with
$a_p \ne 0$ is known.  The curve of Part~\ref{part:evidence} has CM by
$\Z[i]$.

\subsection{What the unit condition gives}\label{ssec:consequences}

\begin{theorem}[$=$ Theorem~\ref{intro:consequence}]\label{prop:consequence}
Let $E/\Q$ have CM by the maximal order $\OK$ of an imaginary
quadratic field $K$, with $\rank E(\Q) = 2$, $L(E,1) = 0$ and $w(E) = +1$.  Let
$p = \fp\bar\fp \ge 5$ be a split prime of good reduction with
$p \notin S_E$ and $\tilde c_2(p) \in \Z_p^\times$.  Then:
\begin{enumerate}
\item $\mu_{\mathrm{an}}(\fp) = 0$ and $\lambda_{\mathrm{an}}(\fp) = 2$;
\item $\corank_{\Z_p}\Sel_{p^\infty}(E/\Q) = 2$;
\item $\Sha(E/\Q)[p^\infty] = 0$.
\end{enumerate}
\end{theorem}

\begin{proof}
Fix such a $p$ and write $\tilde c_2 = \tilde c_2(p)$.  All modules are
over $\Lambda = \Z_p[[T]]$, and
$X = \Sel_{p^\infty}(E/\Q_\infty)^\vee$.  Nothing below is conjectural:
the hypothesis is the unit condition at the single prime $p$, and no
statement about the other primes is assumed.

\emph{Step 1: the analytic side.}  As $p \notin S_E$, the prime $p$ is
non-anomalous and prime to $\prod_v c_v \cdot \#E(\Q)_{\mathrm{tors}}$,
so the factors relating $c_2$ and $\tilde c_2$ are units
(Proposition~\ref{prop:normalisation}) and $v_\fp(c_2(p)) = 0$; and
$c_0 = c_1 = 0$ by Lemma~\ref{lem:c0c1}.  Since $\bar\rho_{E,p}$ is
irreducible, $L_p(E,T) \in \Lambda$ by
Remark~\ref{rmk:integrality}, so
$L_p(E,T) = T^2(c_2 + c_3T + \cdots)$ with $c_2 \in \Z_p^\times$ and $(c_2 + c_3T + \cdots) \in \Lambda^\times$.  Hence
$\mu_{\mathrm{an}} = 0$ and $\lambda_{\mathrm{an}} = 2$, proving (1).

\emph{Step 2: Rubin's main conjecture.}  What this step establishes is
the single identity
\[
   \car_\Lambda(X) \;=\; \bigl(L_p(E,T)\bigr)
   \qquad\text{for every split } p \notin S_E ,
\]
between the characteristic ideal of the Selmer dual and the ideal
generated by the $p$-adic $L$-function; i.e., the cyclotomic Iwasawa main
conjecture for $E$ over $\Q$, which for CM curves at good ordinary $p$
is a theorem of Rubin.  With Step 1 the right-hand side is $(T^2)$, so
$X$ is $\Lambda$-torsion with $\mu(X) = 0$ and $\lambda(X) = 2$: the
algebraic invariants match the analytic ones computed in Step 1.  The
rest of this step is the checking that Rubin's theorem applies in
exactly this form, and may be skipped on a first reading.

As $E$ has CM by $\OK$ and $p > 2$
is split, hence good ordinary, Rubin's theorem \cite[Thm.~12.3]{Rubin}
is stated in exactly this setting and over
$\Q$: writing $S_\infty = \Sel_{p^\infty}(E/\Q_\infty)$ for the Selmer
group he forms over $\Q_\infty$ relative to powers of the rational
prime $p$, so that $S_\infty^\vee = X$ by definition, its Pontryagin
dual over $\Z_p[[\Gal(\Q_\infty/\Q)]]$ has characteristic ideal
generated by the $p$-adic $L$-series $\mathcal{L}_{\mathrm{MSD}}$ of
Mazur and Swinnerton-Dyer \cite{MSD}.  That $X$ is $\Lambda$-torsion is
part of the quoted theorem; Rubin takes it from
\cite[Thm.~4.4]{Rubin88}, which rests on the nonvanishing theorem of
Rohrlich \cite{Rohrlich}.  It is Rubin himself who identifies $S_\infty$ with the
$\fp$-Selmer group of $E$ over $K\Q_\infty$ \cite[p.~66]{Rubin} --- for
odd $p$ restriction realises $S_\infty$ as the
$\Gal(K\Q_\infty/\Q_\infty)$-invariants of
$\Sel_{\fp^\infty}\oplus\Sel_{\bar\fp^\infty}$, which complex
conjugation interchanges, so the invariants project isomorphically onto
the $\fp$-branch --- and it is on that side that his $\fp$-adic
machinery runs; the Selmer identifications that machinery rests on are
proved in Perrin-Riou's memoir \cite[Ch.~II, Lemme~9, Prop.~12,
Th.~18]{PerrinRiou84}, whose standing hypotheses hold here over $K$.
The comparison of the Katz--Yager two-variable $L$-function
\cite[Thm.~1]{Yager} with $\mathcal{L}_{\mathrm{MSD}}$ on the
cyclotomic line is a step inside Rubin's proof \cite[p.~67]{Rubin},
used here through the statement of his theorem; his statement fixes no
parametrisation, and by the sentence following
Lemma~\ref{lem:msdmtt} the ideal $\mathcal{L}_{\mathrm{MSD}}\Lambda$ is
the same for every choice with degree ratio prime to $p$.
Lemma~\ref{lem:comparison} and the period $\Omega_\infty$ play no part
in this step.  Rubin's theorem thus gives
$\car_\Lambda(X) = \mathcal{L}_{\mathrm{MSD}}\Lambda$, and
Lemma~\ref{lem:msdmtt} --- applicable since $p \notin S_E$ is odd, good
ordinary, and $\bar\rho_{E,p}$ is irreducible --- converts this into
the identity announced.  Nothing in the step has used the unit
condition: the identity, the torsion of $X$ and the equalities
$\mu(X) = \mu_{\mathrm{an}}(\fp)$,
$\lambda(X) = \lambda_{\mathrm{an}}(\fp)$ hold for every split
$p \notin S_E$; the values $0$ and $2$ come from Step 1, hence from the
unit condition.

\emph{Step 3: freeness.}  As $\mu(X) = 0$ and $X$ is
$\Lambda$-torsion, $X$ is a finitely generated $\Z_p$-module whose
$\Z_p$-torsion submodule is finite and $\Lambda$-stable, hence a finite
$\Lambda$-submodule; Greenberg's theorem \cite[Prop.~4.14]{Greenberg}
--- applicable since $\Sel_{p^\infty}(E/\Q_\infty)$ is
$\Lambda$-cotorsion by Step 2 and $E(\Q)[p] = 0$, as
$p \nmid \#E(\Q)_{\mathrm{tors}}$ --- gives no nonzero such submodule.
So $X \cong \Z_p^2$, with $T$ acting by some $M \in M_2(\Z_p)$.

\emph{Step 4: the action is trivial.}  Here
$\car_\Lambda(X) = (\det(TI_2 - M))$, and comparing
$T^2 - (\Tr M)T + \det M$ with the generator $T^2v(T)$,
$v \in \Lambda^\times$, in degrees $0$ and $1$ gives
$\det M = \Tr M = 0$: $M$ is nilpotent by Cayley--Hamilton (and
$v = 1$), so $M = 0$ or $\mathrm{rank}_{\Z_p}(\mathrm{im}\,M) = 1$.
Mazur's control theorem (\cite{Mazur};
\cite[Thm.~1.2 and \S 3]{Greenberg}) now makes
$s\colon \Sel_{p^\infty}(E/\Q) \to \Sel_{p^\infty}(E/\Q_\infty)^\Gamma$
an isomorphism: $\ker s$ embeds in
$H^1(\Gamma, E(\Q_\infty)[p^\infty])$, which vanishes since a point of
order $p$ in $E(\Q_\infty)$ would generate over $\Q$ an extension of
degree at once a power of $p$ and --- $\Gal(\Q(E[p])/\Q)$ embedding in
the normaliser of a split Cartan subgroup, of order $2(p-1)^2$, since
$\Gal(\Qbar/K)$ acts on $E[p] \cong \OK/\fp \oplus \OK/\bar\fp$ through
the diagonal torus and complex conjugation exchanges the factors
\cite[Ch.~II, \S 1]{deShalit} --- prime to $p$, hence would lie in
$E(\Q)[p] = 0$; while
$\operatorname{coker} s$ is bounded by local terms, all vanishing here
--- at bad $v$ the $p$-part of $c_v$ is $1$ as $p \notin S_E$, at $p$ one
has $\widetilde E(\F_p)[p^\infty] = 0$ by non-anomalousness, and good
$v \ne p$ contribute nothing.  Dually
$X_\Gamma = X/MX \cong \Sel_{p^\infty}(E/\Q)^\vee$; since
$E(\Q)\otimes\Q_p/\Z_p \hookrightarrow \Sel_{p^\infty}(E/\Q)$ we get
$\mathrm{rank}_{\Z_p}(X/MX) \ge \rank E(\Q) = 2$, whereas $M \ne 0$
would give $2 - \mathrm{rank}(M) = 1$.  Hence $M = 0$ and
$\Sel_{p^\infty}(E/\Q)^\vee \cong X \cong \Z_p^2$.

\emph{Step 5: conclusions.}  Step 4 gives
$\corank\Sel_{p^\infty}(E/\Q) = 2$, which is (2).  In
$0 \to E(\Q)\otimes\Q_p/\Z_p \to \Sel_{p^\infty}(E/\Q) \to
\Sha(E/\Q)[p^\infty] \to 0$
the first term has corank $\rank E(\Q) = 2$, equal to that of the
middle, so $\Sha(E/\Q)[p^\infty]$ has corank $0$ and is finite; its
dual is a finite submodule of the torsion-free $\Z_p^2$, hence zero,
which is (3).
\end{proof}

\begin{remark}\label{rmk:nofinitesub}
The proof avoids the $p$-adic leading-term formalism deliberately; the
unit condition is strong enough that the structure argument of Steps
3 and 4 replaces the descent formulae of Schneider and Perrin-Riou.  Those
formulae re-enter in the converse direction of Proposition~\ref{prop:dictionary}.
\end{remark}

\section{The basis of computation}\label{sec:basis}

The identity \eqref{eq:padicbsd} predicted by Mazur--Tate--Teitelbaum
is, for the curves and primes of \S\ref{sec:thmA}, a theorem up to
$p$-adic units.  It combines Rubin's main conjecture for $K$
with Schneider's leading-term theorem.

\begin{proposition}\label{prop:padicbsdunits}
Let $E/\Q$ have CM by the maximal order $\OK$ of an imaginary
quadratic field $K$, with $\rank E(\Q) = 2$, $L(E,1) = 0$ and
$w(E) = +1$, and let $p \notin S_E$ be a split prime of good reduction.
Assume that $\Sha(E/\Q)[p^\infty]$ is finite and that the cyclotomic
$p$-adic height pairing on $E(\Q)$ is nondegenerate.  Then
\[
   v_\fp\bigl(\tilde c_2(p)\bigr)
   \;=\;
   v_p\bigl(\Reg_\gamma\bigr)
   + v_p\bigl(\#\Sha(E/\Q)[p^\infty]\bigr).
\]
\end{proposition}

\begin{proof}
Under the two assumptions, Schneider's leading-term theorem
\cite[Thms.~2, 2$'$]{Schneider} --- stated for good ordinary $p$ with
$X$ torsion, complex multiplication not excluded, and normalised as in
\cite[Thm.~6.1]{SteinWuthrich} --- gives
\[\ord_{T=0}\car(X) = \corank\Sel_{p^\infty}(E/\Q)\] and identifies the
leading coefficient of a characteristic power series of $X$, up to a
$p$-adic unit, with the right-hand side of \eqref{eq:padicbsd}; its
factor $(1-\alpha_p^{-1})^{2}$ arises algebraically as the square of
the $p$-part of the local term at $p$, a unit on both sides for
non-anomalous $p$.  Rubin's theorem \cite[Thm.~12.3]{Rubin} supplies
$\car_\Lambda(X) = \bigl(L_p(E,T)\bigr)$, as in Step 2 of
Theorem~\ref{prop:consequence} and with no unit condition assumed, so
that algebraic and analytic leading coefficients agree up to units.
Comparing the two and dividing by the unit normalising factors of
Proposition~\ref{prop:normalisation} gives the stated equality of
valuations.
\end{proof}

Proposition~\ref{prop:padicbsdunits} converts the unit condition into a
statement about the regulator and $\Sha$ in one direction only, since
its two hypotheses are themselves arithmetic.  At the primes of
\S\ref{sec:thmA} the conversion goes both ways.

\begin{proposition}[Analytic--arithmetic dictionary]\label{prop:dictionary}
Let $E/\Q$ have CM by the maximal order $\OK$ of an imaginary
quadratic field $K$, with $\rank E(\Q) = 2$, $L(E,1) = 0$ and $w(E) = +1$, and let
$p \notin S_E$ be a split prime of good reduction.  The following are
equivalent:
\begin{enumerate}
\item $\tilde c_2(p) \in \Z_p^\times$;
\item the cyclotomic $p$-adic height pairing on $E(\Q)$ is
nondegenerate with $v_p\bigl(\Reg_\gamma\bigr) = 0$, and
$\Sha(E/\Q)[p^\infty] = 0$.
\end{enumerate}
\end{proposition}

\begin{proof}
(1) $\Rightarrow$ (2): Steps 1--5 of Theorem~\ref{prop:consequence}
give $\Sha(E/\Q)[p^\infty] = 0$ and $\car(X) = (T^2)$ with
$X \cong \Z_p^2$ and trivial $T$-action.  Schneider's theorem
\cite[Thms.~2, 2$'$]{Schneider} gives
$\ord_{T=0}\car(X) \ge \corank\Sel_{p^\infty}(E/\Q)$, with equality if
and only if $\Sha[p^\infty]$ is finite and the cyclotomic height
pairing is nondegenerate.

Here $\ord_{T=0}\car(X) = 2 = \corank\Sel_{p^\infty}(E/\Q)$, so the
pairing is nondegenerate; both hypotheses of
Proposition~\ref{prop:padicbsdunits} therefore hold, and it gives
$v_\fp(\tilde c_2) = v_p(\Reg_\gamma) + v_p(\#\Sha[p^\infty])$ with both
summands $\ge 0$.  For the regulator, let $m$ be the exponent of
$E(\Q)/E^\circ(\Q)$, where $E^\circ(\Q)$ collects the points reducing
to the identity component at every bad place and to the identity in
$\widetilde E(\F_p)$; then $m$ divides
$\#E(\Q)_{\mathrm{tors}}\cdot\prod_v c_v\cdot\#\widetilde E(\F_p)$ and
is prime to $p$.  For $P \in E^\circ(\Q)$ one has $v_p(\hat h_p(P)) \ge 1$
(\S\ref{ssec:notation}), and bilinearity
with $\hat h_p(mP) = m^2\hat h_p(P)$, $p \nmid m$, propagates $v_p \ge 1$ to all
pairings on $E(\Q)$, whence $v_p(\Reg_p) \ge 2$ and
$\Reg_\gamma \in \Z_p$.  Two nonnegative summands with sum $0$ vanish,
which is (2).

(2) $\Rightarrow$ (1): the hypotheses of
Proposition~\ref{prop:padicbsdunits} are exactly (2), so
$v_\fp(\tilde c_2) = v_p(\Reg_\gamma) + v_p(\#\Sha[p^\infty]) = 0$.
\end{proof}

\begin{remark}[The CM route to nondegeneracy]\label{rmk:prroute}
The proof above takes nondegeneracy from Schneider's theorem, stated
over $\Q$.  For a CM curve the order statement is also available in a
form proved in the CM literature.  Perrin-Riou
\cite[Ch.~V, Thm.~8, p.~106]{PerrinRiou84}, over a base field
$F \supseteq K$ and a $\Z_p$-extension of $F$ inside
$F_\infty \subset F(E[p^\infty])$, bounds the order of vanishing at the
origin of the $\fp$-branch characteristic series below by $n_F + r_F$
--- the $\OK$-rank of $E(F)$ modulo torsion plus the $\Z_p$-rank of
$T_p\Sha(E/F)$ --- with equality if and only if the associated height
pairing is nondegenerate.  Its standing hypotheses are ours, applied
over $F = K$: $E$ is its own quadratic twist by $-1$, so $E(K)$ has
$\Z$-rank $4$ and $n_K = 2$, while part 1 of the theorem itself forces
$r_K = 0$; the $\fp$-branch series is $\car_\Lambda(X) = (T^2)$, of
order $2 = n_K + r_K$, whence nondegeneracy on $E(K)$, descending to
$E(\Q)$ by $\Gal(K/\Q)$-equivariance of the pairing (the two
eigenspaces are orthogonal and each is a copy of $E(\Q)\otimes\Q$).
Taking that route in place of Schneider's would require two links we
have not checked: that her bookkeeping is
$\OK$-linear --- $n_F$ is an $\OK$-rank --- so that the $\fp$-branch
series is what the theorem measures; and that restricting her
two-variable series to the cyclotomic line computes the characteristic
ideal of the Selmer dual over $K\Q_\infty$ itself, not that of the
coinvariants over $F_\infty$ \cite[Ch.~V, \S 3]{PerrinRiou84}.
\end{remark}

The equivalence expressed in Proposition~\ref{prop:dictionary} is what makes computational evidence possible.
Theorem~\ref{prop:consequence} converts the unit condition into
arithmetic but not back, whereas what can be computed at a given prime
is the $\fp$-adic regulator; only the converse direction turns such a
measurement into a statement about $\tilde c_2(p)$.

\section{Class sums of Eisenstein--Kronecker numbers}\label{sec:jet}

This section fixes the Eisenstein--Kronecker data in which the criterion of \S\ref{sec:criterion} is stated: a finite set $D_E$, functions $r_{a,b}$ on it, and their class sums $P_{a,b}$, which Proposition~\ref{prop:classsumsL} identifies with critical Hecke $L$-values.

Everything used here is in \S\ref{sec:background}: the Grossencharacter
$\psi_E$ and its conductor $\ff$ (\S\ref{ssec:katz}), and the
Eisenstein--Kronecker numbers $e^*_{a,b}$ with the lattice invariant
$A(\Gamma)$ (\S\ref{ssec:BK}).  The construction is archimedean and no
prime $p$ enters it.

\subsection{The ray class group and the class sums}
\label{ssec:package}

The construction has three ingredients: a set of parameters and two
normalisations.  Each is fixed by the
Bannai--Kobayashi calculus of \S\ref{ssec:BK}, and they are taken here in
that order.  No prime enters any of them.

\emph{The parameters.}\quad By (BK1) every Eisenstein--Kronecker
number, at every index, is a Taylor coefficient of the single Kronecker
theta function $\Theta(z,w)$ at a pair of torsion points, and the
construction fixes the pair.  The measure is generated by one rational
function along the $\fp$-power tower, translated by $\ff$-torsion
(\S\ref{ssec:ellunits}); in the Bannai--Kobayashi form the measure
attached to a class is the expansion of $\Theta$ translated by an
$\ff$-division value in the first variable and untranslated in the
second \cite[Prop.~3.1, Def.~3.6]{BannaiKobayashi}.  The second
parameter is therefore $0$ throughout, and the first runs over the
$\ff$-division values.  Since $K$ has class number one,
$\ff = (f_0)$ for the generator $f_0$ fixed in \S\ref{ssec:katz}, and with
$\Gamma = \Omega_\infty\OK$ the period lattice of
\S\ref{ssec:katz} the points of order dividing $\ff$ are the
division values
\begin{equation}\label{eq:divisionvalues}
   t_g \;:=\; g\,\Omega_\infty/f_0
   \;\in\; \tfrac{1}{f_0}\Gamma/\Gamma ,
   \qquad g \in \OK/\ff .
\end{equation}
The parameters are the $t_g$ of exact order $\ff$, that is those with
$g \in (\OK/\ff)^\times$.

\emph{The first normalisation, and the ray class group.}\quad A
division value depends on the residue $g$, while the Galois group acts
through ideals.  Since $K$ has class number one, every ray class modulo
$\ff$ contains a principal ideal $(g)$ with $g$ prime to $\ff$, and $g$
is determined by the class up to a unit of $\OK$; so
\begin{equation}\label{eq:DEdef}
   D_E \;:=\; \mathrm{Cl}_\ff(K)
   \;=\; (\OK/\ff)^\times\!/\mu_K ,
\end{equation}
and $\Gal(K(\ff)/K)$ acts simply transitively on $D_E$ through the
Artin map.  What is wanted, then, is a value attached to the class and
not to the residue.  Changing $g$ to $ug$ with $u \in \mu_K$ changes
$t_g$ to $u\,t_g$, and the effect on the Eisenstein--Kronecker number
is a root of unity: at second parameter $0$ the series of
\S\ref{ssec:BK} whose continuation defines $e^*_{a,b}(z_0,0)$ is
\[
   \sideset{}{'}\sum_{\gamma \in \Gamma}
   \frac{(\bar z_0 + \bar\gamma)^{a+b}}{|z_0 + \gamma|^{2s}} ,
   \qquad s = b ,
\]
multiplication by $u$ permutes $\Gamma$, and substituting
$\gamma \mapsto u\gamma$ multiplies the numerator by $\bar u^{\,a+b}$
and leaves the denominator fixed.  The resulting identity holds where
the series converges, hence identically in $s$ after continuation, and
in particular at $s = b$:
\begin{equation}\label{eq:EKscaling}
   e^*_{a,b}(u z_0,\, 0)
   \;=\;
   \bar u^{\,a+b}\, e^*_{a,b}(z_0,\, 0) ,
   \qquad u \in \mu_K .
\end{equation}
Write $\varepsilon$ for the unit character of $\psi_E$: the finite
character of $(\OK/\ff)^\times$ with
$\psi_E\bigl((\alpha)\bigr) = \varepsilon(\alpha)\,\alpha$ for
$\alpha \in \OK$ prime to $\ff$, so that $\varepsilon(u) = \bar u$ on
roots of unity.  Then $\varepsilon(u)^{-(a+b)} = u^{\,a+b}$ for
$u \in \mu_K$, which by \eqref{eq:EKscaling} cancels the scaling:
\begin{align*}
   \varepsilon(ug)^{-(a+b)}\, e^*_{a,b}(t_{ug},\, 0)
   &\;=\;
   u^{\,a+b}\,\bar u^{\,a+b}\;
   \varepsilon(g)^{-(a+b)}\, e^*_{a,b}(t_g,\, 0) \\
   &\;=\;
   \varepsilon(g)^{-(a+b)}\, e^*_{a,b}(t_g,\, 0) .
\end{align*}
So the weight $\varepsilon(g)^{-(a+b)}$ makes the value a function of
the class $[g] \in D_E$, and it is not a choice made here.
Bannai--Kobayashi attach to the class $[g]$ the parameter
$\varepsilon(g)\,t_g$ \cite[Def.~3.6]{BannaiKobayashi}; the values of
$\varepsilon$ are roots of unity in $K$, hence lie in $\mu_K$, so
\eqref{eq:EKscaling} with $u = \varepsilon(g)$ gives
\[
   e^*_{a,b}\bigl(\varepsilon(g)\,t_g,\, 0\bigr)
   \;=\;
   \varepsilon(g)^{-(a+b)}\, e^*_{a,b}(t_g,\, 0) ,
\]
and the weighted value is the value of $e^*_{a,b}$ at
Bannai--Kobayashi's parameter for the class.

\emph{The second normalisation.}\quad Division by $A(\Gamma)^{a}$ makes
the value algebraic: at torsion parameters the quotient
$e^*_{a,b}(t,0)/A(\Gamma)^{a}$ lies in $\Qbar$, and in fact in an
abelian extension of $K$ \cite[Thm.~2.9, Cor.~2.11]{BannaiKobayashi}.
With both normalisations in place the value at a class is well defined
and algebraic, at every index.

\begin{definition}[Class functions and class sums]\label{def:classsums}
Let $E/\Q$ have complex multiplication by the maximal order $\OK$ of an
imaginary quadratic field $K$ of class number one, let $\psi_E$ be its
Grossencharacter, $\ff = (f_0)$ the conductor of $\psi_E$ and
$\varepsilon$ the unit character of $\psi_E$, let
$\Gamma = \Omega_\infty\OK$ be the period lattice with invariant
$A(\Gamma)$, and for $g \in (\OK/\ff)^\times$ let
$t_g = g\,\Omega_\infty/f_0$.  For integers $a \ge 0$ and $b \ge 1$ the
\emph{class function} $r_{a,b}$ on the ray class group $D_E$ of
\eqref{eq:DEdef} and its \emph{class sum} $P_{a,b}$ are
\begin{equation}\label{eq:classsums}
   r_{a,b} \colon [g] \longmapsto
   \varepsilon(g)^{-(a+b)}\,
   \frac{e^*_{a,b}(t_g,\, 0)}{A(\Gamma)^{a}} ,
   \qquad
   P_{a,b} \;:=\; \sum_{[g] \in D_E} r_{a,b}\bigl([g]\bigr) .
\end{equation}
\end{definition}

\noindent That $r_{a,b}$ is well defined on $D_E$ is the computation
above, and its values are algebraic by the second normalisation.
The class sums are critical Hecke $L$-values up to explicit periods,
with $L_\ff$ as in \S\ref{ssec:katz}.  The identification is
Proposition~1.6(i) of \cite{BannaiKobayashi}, stated there on the
lattice $\ff \subset K \subset \C$, transported to $\Gamma$ by the
homogeneity \eqref{eq:EKhomogeneity}.

\begin{proposition}[Class sums as critical values]\label{prop:classsumsL}
Let $E/\Q$ have complex multiplication by the maximal order $\OK$ of an
imaginary quadratic field $K$ of class number one and discriminant
$d_K$, let $\psi_E$ be its Grossencharacter, of conductor
$\ff = (f_0)$, let $\Gamma = \Omega_\infty\OK$ be the period lattice,
and let $P_{a,b}$ be the class sums \eqref{eq:classsums}.  Then for all
integers $a \ge 0$ and $b \ge 1$,
\[
   P_{a,b}
   \;=\;
   \Bigl(\frac{f_0}{\Omega_\infty}\Bigr)^{a+b}
   \Bigl(\frac{2\pi}{\Nm\ff\,\sqrt{|d_K|}}\Bigr)^{a}\,
   L_\ff\bigl(\bar\psi_E^{\,a+b},\, b\bigr) .
\]
\end{proposition}

\begin{proof}
Write $\varepsilon$ for the unit character of $\psi_E$, and for
$[g] \in D_E$ put $z_g := \varepsilon(g)g \in \OK$, well defined modulo
$\ff$ by the computation following \eqref{eq:EKscaling}.  By
\eqref{eq:EKscaling} with $u = \varepsilon(g)$,
$r_{a,b}([g]) = e^*_{a,b}(\varepsilon(g)t_g, 0; \Gamma)/A(\Gamma)^{a}$,
and $\varepsilon(g)t_g = \lambda z_g$ with $\lambda := \Omega_\infty/f_0$,
while $\Gamma = \lambda\ff$.  By \eqref{eq:EKhomogeneity},
\[
   A(\Gamma)^{a}\,P_{a,b}
   \;=\;
   \sum_{[g] \in D_E} e^*_{a,b}(\lambda z_g,\, 0;\, \lambda\ff)
   \;=\;
   \bar\lambda^{\,a}\lambda^{-b}
   \sum_{[g] \in D_E} K^*_{a+b}(z_g,\, 0,\, b;\, \ff) .
\]

The sum on the right is $L_\ff(\bar\psi_E^{\,a+b}, b)$.  Let $I_K(\ff)$
be the group of fractional ideals prime to $\ff$ and $P_K(\ff)$ the
subgroup of principal ideals having a generator congruent to $1$ modulo
$\ff$ in the multiplicative sense, so that
$I_K(\ff)/P_K(\ff) = D_E$ by \eqref{eq:DEdef}.
Proposition~1.6(i) of \cite{BannaiKobayashi}, for a Hecke character
$\varphi$ of conductor $\ff$ and infinity type $(1,0)$, states
\[
   L_\ff(\bar\varphi^{\,m}, s)
   \;=\;
   \sum_{\mathfrak{b} \in I_K(\ff)/P_K(\ff)}
   K^*_{m}\bigl(\varphi(\alpha_{\mathfrak{b}}\mathfrak{b}),\, 0,\, s;\,
   \varphi(\mathfrak{b})\,\mathfrak{b}^{-1}\ff\bigr) ,
\]
where $\alpha_{\mathfrak{b}} \in \mathfrak{b}^{-1}$ is any element with
$(\alpha_{\mathfrak{b}}) \in P_K(\ff)$: the summand for
$\mathfrak{b}$ is the part of the $L$-series over the integral ideals in
the ray class of $\mathfrak{b}$, each such ideal being
$(\alpha_{\mathfrak{b}} + \gamma)\mathfrak{b}$ for exactly one
$\gamma \in \mathfrak{b}^{-1}\ff$, and
$\varphi\bigl((\alpha_{\mathfrak{b}}+\gamma)\mathfrak{b}\bigr) =
\varphi(\alpha_{\mathfrak{b}}\mathfrak{b}) + \varphi(\mathfrak{b})\gamma$
with $|\varphi(\mathfrak{c})|^{2} = \Nm\mathfrak{c}$.  Take
$\varphi = \psi_E$, $m = a+b$ and $s = b$.  For the class of
$\mathfrak{b} = (g)$ the lattice is
$\varepsilon(g)g \cdot g^{-1}\OK \cdot \ff = \ff$, since
$\varepsilon(g) \in \mu_K$; and
$\alpha_{\mathfrak{b}}\mathfrak{b} = (g')$ for an integral $g'$ with
$g' \equiv ug \bmod \ff$ for some $u \in \mu_K$, so that
$\psi_E\bigl((g')\bigr) = \varepsilon(g')g' \equiv
\varepsilon(u)u \cdot \varepsilon(g)g = z_g \bmod \ff$, as
$\varepsilon(u) = \bar u$ on $\mu_K$.  Since $K^*_{a+b}(z_0, 0, s; \ff)$
depends on $z_0$ only modulo $\ff$ (\S\ref{ssec:BK}), the summand is
$K^*_{a+b}(z_g, 0, s; \ff)$; the identity holds for
$\mathrm{Re}\,s > (a+b)/2 + 1$, where the $L$-series converges
absolutely, and at $s = b$ by continuation.

Finally $A(\Gamma) = |\lambda|^{2}A(\ff)$ and
$A(\ff) = \Nm\ff \cdot A(\OK) = \Nm\ff\,\sqrt{|d_K|}/2\pi$, so
$A(\Gamma)^{-a}\,\bar\lambda^{\,a}\lambda^{-b}
 = \lambda^{-(a+b)}A(\ff)^{-a}$, which is the stated factor.
\end{proof}

\section{Proof of Theorem~\ref{intro:criterion}}\label{sec:criterion}

This section proves Theorem~\ref{intro:criterion}, restated as Theorem~\ref{thm:criterion} at its end. Propositions~\ref{prop:jetformula} and~\ref{prop:grading} convert the unit condition into the nonvanishing of one element of $\overline{\F}_p$, the criterion class; Lemma~\ref{lem:decoupling} shows that the comparison and Euler factors do not affect it; Lemma~\ref{lem:moments} and Proposition~\ref{prop:exactreduction} compute it exactly, as the reduction of an explicit combination of three Eisenstein--Kronecker class sums of weight $2p-1$.

We recall existing, and fix new, notation. Every numbered statement below
writes its hypotheses in full, so this paragraph fixes symbols only.
$E/\Q$ has CM by the maximal order $\OK$ of an imaginary quadratic
field $K$, and $\psi = \psi_E$ is its Grossencharacter, of conductor
$\ff$ (\S\ref{ssec:katz}).  The field $K$ has class number one
(\S\ref{ssec:notation}), so the Katz measure
$\mu = \mu_{\ff\fp^\infty}$ of \S\ref{ssec:katz} is available for $E$;
we write $\hat\psi$ for the $p$-adic avatar of $\psi_E$, the continuous
character of $G_\infty$ with $\hat\psi(\sigma_{\mathfrak a}) = \psi_E(\mathfrak a)$
for every ideal $\mathfrak a$ prime to $\ff p$, $\sigma_{\mathfrak a}$ its
Artin symbol, and $\mu_\psi := \hat\psi\,\mu$ for the central twist of
$\mu$, its twist by $\hat\psi$.  The values of $\hat\psi$ are $p$-adic
units, so $\mu_\psi$ is $W$-valued, and $\hat\psi$ has infinite order: in
the coordinates of \cite[Def.~3.8]{BannaiKobayashi}, used in the proof of
Lemma~\ref{lem:moments}, its $\fp$-component is $x$.  The prime $p = \fp\bar\fp \ge
5$ is split in $K$ and of good reduction, $S_E$ is the set
\eqref{eq:Sexc}, $W$ and $\mathfrak{m}_W$ are as in
\S\ref{ssec:katz}, and
\[
   L^{\mathrm{K}} \;:=\; L^{\mathrm{Katz}}_\fp
\]
is the Katz-branch series of Lemma~\ref{lem:comparison}: writing
$m(g) := \log_p \chi_{\mathrm{cyc}}(g)/\log_p(1+p) \in \Z_p$, the
restriction to the cyclotomic line is, by definition,
\begin{equation}\label{eq:katzbranch}
   L^{\mathrm{K}}(T)
   \;=\;
   \int_{G_\infty} (1+T)^{m(g)}\, d\mu_\psi(g) .
\end{equation}

\subsection{From the central second jet to a criterion class}

The first two statements convert the unit condition
$\tilde c_2(p) \in \Z_p^\times$ into the nonvanishing of a single
element of $\overline{\F}_p$.  Proposition~\ref{prop:jetformula} does
the analytic half: it writes the second Taylor coefficient of the Katz
branch as an integral of $\ell^2$, where $\ell$ is the $p$-adic
logarithm of the cyclotomic character.  This is the point at which $\log_p$ enters; \S\ref{ssec:moments} replaces it by a polynomial.

\begin{proposition}[The central second jet as a second moment]
\label{prop:jetformula}
Let $E/\Q$ have CM by the maximal order $\OK$ of an imaginary
quadratic field $K$, with
\[
   L(E,1) = 0
   \qquad\text{and}\qquad
   w(E) = +1 ,
\]
and let $p = \fp\bar\fp \ge 5$ be a prime of good reduction split in
$K$.  Let $L^{\mathrm{K}}(T) \in W[[T]]$ be the restriction of the
Katz measure of $E$, twisted by the $p$-adic avatar $\hat\psi$ of $\psi_E$, to
the cyclotomic line, as in Lemma~\textup{\ref{lem:comparison}}, and let
$\ell \colon G_\infty \to p\Z_p$ be the composite of the cyclotomic
character with $\log_p$.  Then
\[
   c_0\bigl(L^{\mathrm{K}}\bigr)
   \;=\; \int_{G_\infty} d\mu_\psi \;=\; 0,
   \qquad
   c_1\bigl(L^{\mathrm{K}}\bigr)
   \;=\; \frac{1}{\log_p(1+p)}\int_{G_\infty} \ell \, d\mu_\psi
   \;=\; 0,
\]
and
\[
   c_2\bigl(L^{\mathrm{K}}\bigr)
   \;=\;
   \frac{1}{2\,\log_p(1+p)^2}
   \int_{G_\infty} \ell(g)^2 \, d\mu_\psi(g) .
\]
\end{proposition}

\begin{proof}
By \eqref{eq:katzbranch},
$L^{\mathrm{K}}(T) = \int_{G_\infty} (1+T)^{m(g)}\,d\mu_\psi(g)$
with $m(g) = \ell(g)/\log_p(1+p) \in \Z_p$.  Expanding
$(1+T)^m = \sum_k \binom{m}{k}T^k$ --- convergent in $W[[T]]$ with
bounded coefficients, $\mu$ being a measure --- gives
\[
   c_k\bigl(L^{\mathrm{K}}\bigr)
   \;=\; \int_{G_\infty} \binom{m}{k}\, d\mu_\psi ,
\]
which is the stated integral in each of the cases $k = 0$ and $k = 1$.
Both vanish: by Lemma~\ref{lem:comparison},
$L_p(E,T) = c_p\,u(T)\,L^{\mathrm{K}}(T)$ with $c_p \ne 0$ and $u(T)$ a
unit power series, so $L^{\mathrm{K}}$ and $L_p(E,T)$ have the same
order of vanishing at $T = 0$; and $L_p(E,T)$ has $c_0 = c_1 = 0$ by
Lemma~\ref{lem:c0c1}, whose hypotheses $L(E,1) = 0$ and $w(E) = +1$ are
ours and which applies because a split $p$ is good ordinary
(Deuring's criterion, \S\ref{sec:thmA}).  For $k = 2$,
\[
   c_2\bigl(L^{\mathrm{K}}\bigr)
   \;=\; \int \binom{m}{2}\, d\mu_\psi
   \;=\; \tfrac12 \int m^2\, d\mu_\psi
        - \tfrac12 \int m\, d\mu_\psi
   \;=\; \tfrac12 \int m^2\, d\mu_\psi ,
\]
and $m^2 = \ell^2/\log_p(1+p)^2$.
\end{proof}

The next step replaces $\ell$ by a polynomial, at the cost of an error
of valuation $3$.  Class field theory splits the cyclotomic logarithm
into the two CM directions.  Let $G^{(1)} \cong (1+p\Z_p)^2$ be the
image under the Artin map of the principal local units at $\fp$,
$\bar\fp$, with coordinates $(x,y)$.  For $p$ prime to
$\#\mathrm{Cl}_\ff(K)$ (the ray-class clause of \eqref{eq:Sexc}) the whole pro-$p$ part of $G_\infty$ is $G^{(1)}$, so that
\begin{equation}\label{eq:tamesplit}
   G_\infty \;=\; \Delta \times G^{(1)},
   \qquad
   \Delta \;=\; \text{the tame part of the $\ff p$-ray class group} .
\end{equation}
Since $\chi_{\mathrm{cyc}} = \Nm_{K/\Q}$ on id\`eles, restricted to the
split completions,
\begin{equation}\label{eq:ellsplit}
   \ell(g) \;=\; \log_p x(g) + \log_p y(g)
   \qquad (g \in G^{(1)}) .
\end{equation}
The sign of \eqref{eq:ellsplit} depends on the normalisation of the
Artin map; it is immaterial: only $\ell^2$ and the vanishing of
$\int_{G_\infty} \ell \, d\mu_\psi$ are used below.

Proposition~\ref{prop:grading} now truncates $\log_p$ to its linear
term.  Its second part is the statement the rest of the section works
with: the unit condition at $p$ is the nonvanishing of one element of
$\overline{\F}_p$.

\begin{proposition}[Grading of the second moment; the criterion class]
\label{prop:grading}
Let $E/\Q$ have CM by the maximal order $\OK$ of an imaginary
quadratic field $K$, with $L(E,1) = 0$ and $w(E) = +1$, let $\ff$ be
the conductor of $\psi_E$, and let $p = \fp\bar\fp \ge 5$ be a prime of
good reduction split in $K$ with $p \nmid \#\mathrm{Cl}_\ff(K)$.  Write
$G_\infty = \Delta \times G^{(1)}$ as in \eqref{eq:tamesplit}; for
$\delta \in \Delta$ let $\mu_\delta$ be the restriction of $\mu_\psi$
to the coset $\delta \times G^{(1)}$, viewed as a measure on
$(1+p\Z_p)^2$ in the coordinates $(x,y)$; put $u_1 = x-1$ and
$u_2 = y-1$, so that $v_p(u_1), v_p(u_2) \ge 1$ on the support; and set
\[
   M_2(\fp) \;:=\; \sum_{\delta \in \Delta}
   \int_{(1+p\Z_p)^2} (u_1 + u_2)^2 \, d\mu_\delta .
\]
Then:
\begin{enumerate}
\item $2\log_p(1+p)^2\, c_2\bigl(L^{\mathrm{K}}\bigr) \equiv M_2(\fp)
      \pmod{p^3 W}$;
\item $v_p\bigl(M_2(\fp)\bigr) \ge 2$, and
      $v_p\bigl(c_2(L^{\mathrm{K}})\bigr) = 0$ if and only if the
      \emph{criterion class}
      \[
         \mathfrak{c}(\fp) \;:=\; p^{-2} M_2(\fp) \bmod \mathfrak{m}_W
         \;\in\; \overline{\F}_p
      \]
      is nonzero.
\end{enumerate}
\end{proposition}

\begin{proof}
The homomorphism $\ell = \log_p \circ \chi_{\mathrm{cyc}}$ kills
$\Delta$: the order of $\Delta$ is prime to $p$, so
$\chi_{\mathrm{cyc}}(\Delta)$ lies in the group of roots of unity of
$\Z_p^\times$, on which $\log_p$ vanishes.  Hence
$\ell(\delta g) = \ell(g)$ for $g \in G^{(1)}$, and \eqref{eq:ellsplit}
applies on every coset.

For (1), expand $\log_p x = u_1 - u_1^2/2 + \cdots$ and likewise in
$u_2$.  On $1+p\Z_p$ the term $u_1^k/k$ has $v_p \ge k - v_p(k)$, which
is $\ge 2$ for $k \ge 2$ and $p \ge 5$; so
\[
   \bigl(\log_p x + \log_p y\bigr)^2
   \;=\; (u_1 + u_2)^2 + R ,
\]
where every monomial of $R$ has total degree $\ge 3$ and, by the same
estimate, $v_p(R) \ge 3$ pointwise; the lowest terms of $R$ are
$-(u_1+u_2)(u_1^2 + u_2^2)$, of coefficient $\pm 1$, so no denominator
weakens the bound, and $p \ge 5$ makes $2$ and $3$ units.  As $\mu_\psi$
is $W$-valued, integrating termwise and summing over the finitely many
cosets gives
\[
   \int_{G_\infty} \ell^2\, d\mu_\psi
   \;\equiv\; M_2(\fp) \pmod{p^3 W} ,
\]
which is (1) by Proposition~\ref{prop:jetformula}.

For (2), $v_p\bigl((u_1+u_2)^2\bigr) \ge 2$ pointwise gives
$v_p(M_2) \ge 2$.  Since $v_p\bigl(\log_p(1+p)\bigr) = 1$, the left
side of (1) has valuation $2 + v_p(c_2(L^{\mathrm{K}}))$, and the two
sides of (1) have equal valuation whenever either is $< 3$.  So
$v_p(c_2(L^{\mathrm{K}})) = 0$ if and only if $v_p(M_2) = 2$, which is
the nonvanishing of $\mathfrak{c}(\fp)$.
\end{proof}

We fix the following two conventions in the sequel.
\begin{enumerate}
\item The reduction defining $\mathfrak{c}(\fp)$ is modulo the maximal
ideal $\mathfrak{m}_W$, whose residue field is $\overline{\F}_p$
(equivalently modulo $p$, as $W/pW = \overline{\F}_p$); the prime $\fp$
of $K$ plays no role in it.
\item The $\fp$ in the arguments of $M_2$ and $\mathfrak{c}$ records
the branch along which the Katz series was formed.  Exchanging
$\iota_p$ for its conjugate exchanges the two coordinates, which
fixes $(u_1+u_2)^2$, and conjugates the values of $\mu_\psi$; so it
carries $M_2(\fp)$ and $\mathfrak{c}(\fp)$ to their images under the
induced automorphism of $W$, respectively of $\overline{\F}_p$.  The
\emph{vanishing} of $\mathfrak{c}(\fp)$ therefore depends on $p$
alone.
\end{enumerate}

\subsection{Decoupling of the Euler and comparison factors}

Besides $M_2(\fp)$, two factors enter $c_2(p)$: the interpolation
factor $(1-\alpha_p^{-1})^2$, whose logarithmic derivative in the
cyclotomic variable involves $\log_p\alpha_p$, and the comparison factor
$c_p\,u(T)$ of Lemma~\ref{lem:comparison}.  The next lemma shows that
neither affects the valuation of $c_2(p)$.  Its first part is an identity about
power series, its second the arithmetic consequence.

\begin{lemma}[Decoupling]\label{lem:decoupling}
\begin{enumerate}
\item Let $R$ be a commutative ring and let $A, B \in R[[T]]$ satisfy
$c_0(B) = c_1(B) = 0$.  Then
\[
   c_2(AB) \;=\; A(0)\, c_2(B) .
\]
\item Let $E/\Q$ have CM by the maximal order $\OK$ of an imaginary
quadratic field $K$, with $L(E,1) = 0$ and $w(E) = +1$, and let
$p = \fp\bar\fp \ge 5$ be a prime of good reduction split in $K$ with
$p \notin S_E$.  Then
\[
   v_p\bigl(c_2(p)\bigr) \;=\; v_p\bigl(c_2(L^{\mathrm{K}})\bigr) .
\]
\end{enumerate}
\end{lemma}

\begin{proof}
(1) $c_2(AB) = c_0(A)c_2(B) + c_1(A)c_1(B) + c_2(A)c_0(B)$, and the
last two terms vanish.

(2) By Lemma~\ref{lem:comparison},
$L_p(E,T) = c_p\,u(T)\,L^{\mathrm{K}}(T)$ with
$c_p \in \mathrm{Frac}(W)^\times$ and $u$ a unit power series, and
$v_p(c_p) = 0$ because $p \notin S_E \supseteq S_{\mathrm{cmp}}$.
Apply (1) with $A = c_p\,u$ and $B = L^{\mathrm{K}}$, whose
$c_0$ and $c_1$ vanish by Proposition~\ref{prop:jetformula}:
$c_2(p) = c_p\,u(0)\,c_2(L^{\mathrm{K}})$, and $c_p\,u(0)$ is a unit.
\end{proof}

\begin{remark}[What the decoupling excludes]\label{rmk:fq}
The comparison and Euler
factors enter $c_2(p)$ only through $c_0(L^{\mathrm{K}})$ and
$c_1(L^{\mathrm{K}})$, both of which vanish because the central zero of
$L_p(E,T)$ is double; so the coefficients $c_1(A)$ and
$c_2(A)$, which carry $\log_p\alpha_p$ and with it
$(\alpha_p^{p-1}-1)/p$, do not appear.  At a nonvanishing central
value, or at a simple zero, this fails: $c_0(B)$ or $c_1(B)$ is then
nonzero and the derivatives of $A$ enter the leading coefficient.
\end{remark}

\subsection{The second moment as Eisenstein--Kronecker class sums}
\label{ssec:moments}

This subsection computes the criterion class exactly: $\mathfrak{c}(\fp)$ is a unit multiple of the reduction of $p^{-2}B(\fp)$, where $B(\fp)$ is an explicit combination of three Eisenstein--Kronecker class sums of weight $2p-1$, of Euler factors and of factorials (Proposition~\ref{prop:exactreduction}).

\emph{From $M_2(\fp)$ to twisted moments.}\quad Write $(x, y)$ for the
local-unit coordinates at $\fp$, $\bar\fp$ and
$\langle x\rangle = x\,\omega(x)^{-1}$ for the principal-unit part, so
that the coordinates of Proposition~\ref{prop:grading} are
$(\langle x\rangle, \langle y\rangle)$ and
$u_1 + u_2 = \langle x\rangle + \langle y\rangle - 2$.  For
$a + b \le 2$ put
\[
   m_{a,b}
   \;:=\;
   \int_{G_\infty} \langle x\rangle^a \langle y\rangle^b \, d\mu_\psi
   \;=\;
   \int_{G_\infty} x^a y^b\,\omega^{-a}(x)\,\omega^{-b}(y)\,
   d\mu_\psi ,
\]
the \emph{twisted polynomial moments}: polynomial moments of
$\mu_\psi$ twisted by the Teichm\"uller powers that strip the tame
parts of the coordinates.  Expanding the square,
\begin{equation}\label{eq:M2moments}
   M_2(\fp)
   \;=\;
   m_{2,0} + 2\, m_{1,1} + m_{0,2}
   - 4\, m_{1,0} - 4\, m_{0,1} + 4\, m_{0,0} .
\end{equation}
The coefficients are integers, units for $p \ge 5$; each $m_{a,b}$ is
the integral of a unit-valued continuous function against the
$W$-valued measure $\mu_\psi$, so it lies in $W$; and
$m_{0,0} = \int d\mu_\psi$ vanishes here
(Proposition~\ref{prop:jetformula}).  The display therefore needs no
coefficient bookkeeping.  The cost sits in the twists.  Their
conductor divides $\fp\bar\fp$, and
resolving a twist by $\omega^{-a}$ produces either Gauss sums of
conductor $p$, which are not $\fp$-units
\cite[Thm.~3.11]{BannaiKobayashi}, or sums over $\fp$-power torsion
translates \cite[Prop.~3.10]{BannaiKobayashi} with root-of-unity
weights \cite[Prop.~2.20, Cor.~2.22]{BannaiKobayashi}.  That cost
belongs to this presentation and not to $M_2(\fp)$:
Proposition~\ref{prop:exactreduction} below reaches
$\mathfrak{c}(\fp)$ along a single direction in the exponents, where
there is no twist to resolve and no limit to take.  Nor does either
route supply the divisibility by $p^2$ that defines
$\mathfrak{c}(\fp)$: that divisibility is
Proposition~\ref{prop:grading}(2), and it comes from the level
structure, $\log_p$ of a principal unit lying in $p\Z_p$ in each of
the two coordinates.  The tame data of conductor dividing
$\ff$ behaves differently: the unit character $\varepsilon$ of
$\psi_E$ defined below, and Gauss sums of characters of that
conductor, are $\fp$-units at every split $p$ of good reduction, that
conductor being prime to $p$.

\emph{The ray class group and the class sums.}\quad By \eqref{eq:katzinterp}
each $m_{a,b}$ is, up to the explicit factors of that formula, a
critical value of a Teichm\"uller twist of $\psi^{-1}$ at a parameter
shifted by $(a,b)$, and by \S\ref{ssec:BK} such values are carried by
Eisenstein--Kronecker data at torsion parameters.  The ray class group
$D_E$ \eqref{eq:DEdef} and the class sums $P_{a,b}$ of
\eqref{eq:classsums} are that data.

\emph{The moments in closed form.}\quad The next lemma computes the
polynomial moments of (BK3) in \S\ref{ssec:BK}, with every constant
explicit.  Part~(1) is the
identity before the restriction to the local units, one class
parameter at a time; part~(2) is the identity after that restriction,
for $\mu_\psi$ itself, and it is where the Euler factors and the
normalisation constant enter.  The moment of index $(k,l)$ pairs with
the class function $r_{l,k+1}$, of weight $k+l+1$.  The Katz Galois
coordinates invert the second leg
\cite[Def.~3.8]{BannaiKobayashi}, which is why the second exponent in
part~(2) is negative.

\begin{lemma}[Polynomial moments as Eisenstein--Kronecker values]
\label{lem:moments}
Let $E/\Q$ have CM by the maximal order $\OK$ of an imaginary
quadratic field $K$, let $\ff$ be the conductor of $\psi_E$, and let
$\Gamma = \Omega_\infty\OK$ be the period lattice.  Assume:
\begin{enumerate}
\item[(a)] $p = \fp\bar\fp \ge 5$ is a prime of good reduction split
      in $K$;
\item[(b)] $p$ is prime to $\Nm\ff$.
\end{enumerate}
Let $\mu_\psi$ be the central twist of the Katz
measure $\mu_{\ff\fp^\infty}$ and $(x,y)$ the local-unit coordinates
at $\fp$, $\bar\fp$.  For $[g] \in D_E$ let $\nu_{[g]}$ be the measure on
$\Z_p \times \Z_p$ whose power series, in the sense of
\textup{\cite[Prop.~3.1]{BannaiKobayashi}}, is the expansion of
$\Theta_{\varepsilon(g)t_g,\,0}(z, \Nm\ff\,w)$ with its polar term
removed, taken in the coordinates of the formal group at $\fp$
trivialised by the $p$-adic period $\Omega_p$ of
\eqref{eq:katzinterp}; it is $W$-valued
\textup{\cite[Cor.~2.18]{BannaiKobayashi}}.  Put $\pi := \psi_E(\fp)$
and, for integers $k, l \ge 0$,
\[
   \mathcal{E}(k,l)
   \;:=\;
   \Bigl(1 - \frac{\pi^{\,k+l+1}}{p^{\,l+1}}\Bigr)
   \Bigl(1 - \frac{\pi^{\,k+l+1}}{p^{\,k+1}}\Bigr) .
\]
Then, for all integers $k, l \ge 0$:
\begin{enumerate}
\item for every $[g] \in D_E$,
\[
   \int_{\Z_p \times \Z_p} x^{k}\,y^{l}\, d\nu_{[g]}
   \;=\;
   \Omega_p^{\,k+l}\,(-1)^{k+l}\,k!\,(\Nm\ff)^{l}\,
   r_{l,\,k+1}\bigl([g]\bigr) ;
\]
\item with $\kappa := \Omega_p$, the $p$-adic period of \eqref{eq:katzinterp},
\[
   \int_{G_\infty} x^{k}\,y^{-l}\, d\mu_\psi
   \;=\;
   \kappa\,\mathcal{E}(k,l)\,\Omega_p^{\,k+l}\,(-1)^{k+l}\,k!\,
   (\Nm\ff)^{l}\, P_{l,\,k+1} .
\]
\end{enumerate}
\end{lemma}

\begin{proof}
(1)  The moments of a measure are the logarithmic derivatives of its
power series at the origin \cite[Prop.~3.1]{BannaiKobayashi}, and each
of the two logarithmic derivations is $\Omega_p$ times the
corresponding complex-analytic one, both legs being trivialised by the
same $p$-adic period; so the $(k,l)$ moment of $\nu_{[g]}$ is
$\Omega_p^{k+l}\,k!\,l!$ times the coefficient of $z^{k}w^{l}$ in the
defining expansion.  In that expansion the second parameter is $0$,
and the first, $\varepsilon(g)t_g$, is not in $\Gamma$, $g$ being prime
to $\ff$; so the only polar term of \cite[Thm.~1.17]{BannaiKobayashi}
is the one removed, and the coefficient of $z^{k}w^{l}$ is
\[
   (-1)^{k+l}\,\frac{(\Nm\ff)^{l}}{l!}\,
   \frac{e^*_{l,\,k+1}\bigl(\varepsilon(g)t_g,\, 0\bigr)}
        {A(\Gamma)^{l}} ,
\]
the factor $(\Nm\ff)^{l}$ coming from $w \mapsto \Nm\ff\,w$.  By the
identity displayed above, that is
$(-1)^{k+l}(\Nm\ff)^{l}\,r_{l,k+1}([g])/l!$, which gives (1).

(2)  Write $\nu := \sum_{[g] \in D_E} \nu_{[g]}$, the measure
\cite[Def.~3.6]{BannaiKobayashi} attaches to $\psi_E$, and $\nu^\times$
for its restriction to $\Z_p^\times \times \Z_p^\times$.

\emph{The Euler factors.}  Restriction to
$\Z_p^\times \times \Z_p^\times$ is the alternating sum of the four
averages over the $p$-torsion in the two formal-group variables
\cite[Lem.~3.4]{BannaiKobayashi}, and
\cite[Prop.~3.5]{BannaiKobayashi} evaluates the three averaged terms as
theta functions at the lattices $\bar\fp\Gamma$, $\bar\fp\Gamma$ and
$\bar\fp^{2}\Gamma$, with the parameter and the variable multiplied by
$p$ in the first leg, in the second, and in both.  (Removing the polar
term in one normalisation rather than another changes $\nu$ by a
measure supported on $\{0\} \times \Z_p$, which the restriction
annihilates, so the choice made in (1) is immaterial here.)  Apply the
homogeneity \eqref{eq:EKhomogeneity} with
$\lambda = \bar\pi$, $\bar\pi$ and $\bar\pi^{2}$, using
$\pi\bar\pi = p$: the three terms return to the lattice $\Gamma$, at
the class parameters of $[g\fp]$, $[g\bar\fp^{-1}]$ and
$[g\fp\bar\fp^{-1}]$, and carry the factors $\pi^{k+l+1}/p^{\,l+1}$,
$\pi^{k+l+1}/p^{\,k+1}$ and the product of the two.  Those three
substitutions permute $D_E$, so they leave the class sums unchanged,
and summing (1) over $D_E$ gives
\[
   \int_{\Z_p^\times \times \Z_p^\times} x^{k}\,y^{l}\, d\nu^\times
   \;=\;
   \mathcal{E}(k,l)\,\Omega_p^{\,k+l}\,(-1)^{k+l}\,k!\,
   (\Nm\ff)^{l}\,P_{l,\,k+1} .
\]

\emph{The transfer.}  The Katz Galois coordinate at $(\fp, \bar\fp)$ is
$(x, y^{-1})$ in the coordinates of $\nu$, and the Katz measure is
$\Omega_p\,x^{-1}$ times $\nu^\times$ in those coordinates
\cite[Def.~3.8]{BannaiKobayashi}.  Twisting by $\hat\psi$ multiplies
the integrand by that character, whose
$\fp$-component is $x$; the factor $x^{-1}$ cancels, and the finite
part of the twist is the character $\varepsilon$, already carried by
the class parameters of \cite[Def.~3.6]{BannaiKobayashi}; this is the
first display in the proof of \cite[Prop.~3.13]{BannaiKobayashi}.  So
$\mu_\psi = \Omega_p\,\nu^\times$ in these coordinates, and
$\kappa = \Omega_p$, a unit of $W$.  Taking
$x^{k}y^{-l}$ as the integrand gives (2).
\end{proof}

\emph{The reduction.}\quad The criterion class can now be computed
exactly.  Only one direction in the exponents survives the passage from
$M_2(\fp)$ to $\mathfrak{c}(\fp)$, and along it the moments are plain
polynomial moments at nonnegative indices: no twist is resolved, no
limit is taken, and no torsion translate occurs.

\begin{proposition}[Exact reduction of the criterion class]
\label{prop:exactreduction}
Let $E/\Q$ have CM by the maximal order $\OK$ of an imaginary
quadratic field $K$, and let $\ff$ be the conductor of $\psi_E$.
Assume:
\begin{enumerate}
\item[(a)] $L(E,1) = 0$ and $w(E) = +1$;
\item[(b)] $p = \fp\bar\fp \ge 5$ is a prime of good reduction split
      in $K$;
\item[(c)] $p \notin S_E$, the set of \eqref{eq:Sexc}.
\end{enumerate}
Let $\mathfrak{c}(\fp)$ be the criterion class of
Proposition~\textup{\ref{prop:grading}}, let $\mathcal{E}(k,l)$ be as
in Lemma~\textup{\ref{lem:moments}}, let $P_{a,b}$ be the class sums
\eqref{eq:classsums}, and put
\begin{equation}\label{eq:bracket}
\begin{aligned}
   B(\fp) \;:=\;
   &\mathcal{E}(2p-2,\,0)\,(2p-2)!\; P_{0,\,2p-1} \\
   &{}-\; 2\,\mathcal{E}(p-1,\,p-1)\,(p-1)!\,(\Nm\ff)^{\,p-1}\,
        P_{p-1,\,p} \\
   &{}+\; \mathcal{E}(0,\,2p-2)\,(\Nm\ff)^{\,2p-2}\, P_{2p-2,\,1}
   \;\in\; \Qbar ,
\end{aligned}
\end{equation}
with $v_\fp$ the valuation of \S\textup{\ref{ssec:notation}} on
$\Qbar^\times$ through $\iota_p$, normalised by $v_\fp(p) = 1$.  Then:
\begin{enumerate}
\item $v_\fp\bigl(B(\fp)\bigr) \ge 2$;
\item $\mathfrak{c}(\fp)$ is the reduction modulo $\mathfrak{m}_W$ of
      $u\,p^{-2}B(\fp)$ for some $u \in W^\times$;
\item $\mathfrak{c}(\fp) \ne 0$ if and only if
      $v_\fp\bigl(B(\fp)\bigr) = 2$.
\end{enumerate}
\end{proposition}

\begin{proof}
Write $N(k,l) := \int_{G_\infty} x^{k}y^{-l}\, d\mu_\psi$ for $k, l \ge
0$, the moments of Lemma~\ref{lem:moments}(2).

\emph{Step 1: collapse to one direction.}  Summing
Proposition~\ref{prop:grading}'s definition over the cosets,
$M_2(\fp) = \int_{G_\infty}(u_1+u_2)^2\,d\mu_\psi$ with
$u_1 + u_2 = \langle x\rangle + \langle y\rangle - 2$, as above.  On
the support of $\mu_\psi$ the coordinates are local units, and
$\log_p z = \log_p\langle z\rangle$.  The estimate in the proof of
Proposition~\ref{prop:grading} gives $\langle z\rangle - 1 = \log_p z
+ O(p^{2})$, so by \eqref{eq:ellsplit}
\[
   u_1 + u_2 \;=\; \ell + O(p^{2}) ,
   \qquad
   \ell \;=\; \log_p x + \log_p y \;\in\; p\Z_p ,
\]
and hence $(u_1+u_2)^2 = \ell^2 + O(p^{3})$, pointwise and uniformly.

\emph{Step 2: a polynomial for $\ell$.}  For a local unit $z$ one has
$z^{p-1} = \langle z\rangle^{p-1} \in 1 + p\Z_p$ and
$(p-1)\log_p z = \log_p\bigl(z^{p-1}\bigr) = (z^{p-1} - 1) + O(p^{2})$,
so, as $(p-1)^{-1} \equiv -1 \bmod p$ and $z^{p-1} - 1 \in p\Z_p$,
$\log_p z = (1 - z^{p-1}) + O(p^{2})$; and
$(1 - z^{p-1}) + (1 - z^{-(p-1)}) = -(1-z^{p-1})^{2}z^{-(p-1)}$ lies in
$p^{2}\Z_p$, so also $\log_p z = -\bigl(1 - z^{-(p-1)}\bigr) +
O(p^{2})$.  Applying the first form in $x$ and the second in $y$,
\[
   \ell \;=\; y^{-(p-1)} - x^{p-1} + O(p^{2}) ,
\]
and since $\ell \in p\Z_p$, squaring gives, uniformly on the support,
\[
   \ell^{2}
   \;=\;
   x^{2(p-1)} - 2\,x^{p-1}y^{-(p-1)} + y^{-2(p-1)} + O(p^{3}) .
\]
With Step 1 and the $W$-valuedness of $\mu_\psi$, integration gives
\[
   M_2(\fp)
   \;\equiv\;
   N(2p-2,\,0) - 2\,N(p-1,\,p-1) + N(0,\,2p-2)
   \pmod{p^{3}W} .
\]

\emph{Step 3: the three moments.}  The three indices have the same sum
$k + l = 2p-2$, so $(-1)^{k+l} = 1$ at each and the same factor
$\kappa\,\Omega_p^{2p-2}$ occurs in each; by
Lemma~\ref{lem:moments}(2), which applies as $p \nmid \Nm\ff$ for $p \notin S_E$,
\[
   N(2p-2,\,0) - 2\,N(p-1,\,p-1) + N(0,\,2p-2)
   \;=\;
   \kappa\,\Omega_p^{\,2p-2}\, B(\fp) ,
\]
the three terms being those of \eqref{eq:bracket}.  Put
$u := \kappa\,\Omega_p^{2p-2} = \Omega_p^{2p-1}$, a unit of $W$.

For (1): $v_p(M_2(\fp)) \ge 2$ by Proposition~\ref{prop:grading}(2),
so the displayed congruence forces $v_p(u\,B(\fp)) \ge 2$, and $u$ is a
unit.  For (2) and (3): dividing the congruence of Step 2 by $p^2$,
\[
   \mathfrak{c}(\fp)
   \;=\; p^{-2}M_2(\fp) \bmod \mathfrak{m}_W
   \;=\; u\,p^{-2}B(\fp) \bmod \mathfrak{m}_W ,
\]
and $p^{-2}B(\fp)$ lies in $W$, so a unit multiple of it is nonzero
modulo $\mathfrak{m}_W$ exactly when $v_\fp(B(\fp)) = 2$.
\end{proof}

\begin{theorem}[The exact criterion; $=$ Theorem~\ref{intro:criterion}]\label{thm:criterion}
Let $E/\Q$ have CM by the maximal order $\OK$ of an imaginary quadratic
field $K$, with $\rank E(\Q) = 2$, $L(E,1) = 0$ and $w(E) = +1$, and let
$p = \fp\bar\fp \ge 5$ be a prime of good reduction split in $K$ with
$p \notin S_E$.  Let $B(\fp)$ be the bracket \eqref{eq:bracket}.  The
following are equivalent:
\begin{enumerate}
\item $\tilde c_2(p) \in \Z_p^\times$;
\item the cyclotomic $p$-adic height pairing on $E(\Q)$ is
      nondegenerate with $v_p\bigl(\Reg_\gamma\bigr) = 0$, and
      $\Sha(E/\Q)[p^\infty] = 0$;
\item $v_\fp\bigl(B(\fp)\bigr) = 2$.
\end{enumerate}
\end{theorem}

\begin{proof}
The equivalence of (1) and (2) is Proposition~\ref{prop:dictionary}.
For (1) and (3): by Proposition~\ref{prop:normalisation}, (1) holds if
and only if $v_p(c_2(p)) = 0$, since $p \notin S_E$ is not anomalous
and is prime to the Tamagawa and torsion terms; by
Lemma~\ref{lem:decoupling}(2), if and only if
$v_p(c_2(L^{\mathrm{K}})) = 0$; by Proposition~\ref{prop:grading}(2),
if and only if $\mathfrak{c}(\fp) \ne 0$; and by
Proposition~\ref{prop:exactreduction}(3), if and only if
$v_\fp(B(\fp)) = 2$.
\end{proof}

Proposition~\ref{prop:classsumsL} converts the bracket \eqref{eq:bracket}
into Hecke $L$-values.  The three class sums in $B(\fp)$ have the same
weight $2p-1$, and the factors $(\Nm\ff)^{l}$ of \eqref{eq:bracket}
cancel against the $(\Nm\ff)^{-a}$ of the proposition, so one period
remains.

\begin{corollary}[The bracket as critical values]\label{cor:bracketL}
Let $E/\Q$ have complex multiplication by the maximal order $\OK$ of an
imaginary quadratic field $K$ of class number one and discriminant
$d_K$, let $\psi_E$ be its Grossencharacter, of conductor
$\ff = (f_0)$, let $\Gamma = \Omega_\infty\OK$ be the period lattice, let
$p = \fp\bar\fp \ge 5$ be a prime of good reduction split in $K$, let
$\mathcal{E}(k,l)$ be as in Lemma~\textup{\ref{lem:moments}}, and let
$B(\fp)$ be the bracket \eqref{eq:bracket}.  Put
$L(s) := L_\ff(\bar\psi_E^{\,2p-1}, s)$.  Then
\[
\begin{aligned}
   B(\fp) \;=\; \Bigl(\frac{f_0}{\Omega_\infty}\Bigr)^{2p-1}
   \Bigl[\;&\mathcal{E}(2p-2,\,0)\,(2p-2)!\;L(2p-1) \\
   &{}- 2\,\mathcal{E}(p-1,\,p-1)\,(p-1)!\,
     \Bigl(\frac{2\pi}{\sqrt{|d_K|}}\Bigr)^{p-1} L(p) \\
   &{}+ \mathcal{E}(0,\,2p-2)\,
     \Bigl(\frac{2\pi}{\sqrt{|d_K|}}\Bigr)^{2p-2} L(1)\;\Bigr] .
\end{aligned}
\]
\end{corollary}

\begin{proof}
Proposition~\ref{prop:classsumsL} at $(a,b) = (0,\,2p-1)$,
$(p-1,\,p)$ and $(2p-2,\,1)$ gives
\[
   (\Nm\ff)^{a}\,P_{a,b}
   \;=\;
   \Bigl(\frac{f_0}{\Omega_\infty}\Bigr)^{2p-1}
   \Bigl(\frac{2\pi}{\sqrt{|d_K|}}\Bigr)^{a}\,
   L(b)
\]
in each case, and these are the three class sums of \eqref{eq:bracket}
with their factors $(\Nm\ff)^{a}$.
\end{proof}

So condition (3) of Theorem~\ref{thm:criterion} is a condition on three
critical values of $\bar\psi_E^{\,2p-1}$, at $s = 2p-1$, $p$ and $1$:
the two ends and the centre of its critical strip $1 \le s \le 2p-1$, of
a weight that moves with $p$.  The criterion of Coates, Liang and
Sujatha \cite[Thm.~2.2]{CLS1} tests one critical value,
$L(\bar\psi_E^{\,p}, p)$, at the end $s = p$ of the critical strip of
$\bar\psi_E^{\,p}$; up to the period factor of
Proposition~\ref{prop:classsumsL} that value is $P_{0,p}$, the class sum
paired by Lemma~\ref{lem:moments}(2) with the moment $x^{p-1}$ along the
$\fp$-direction alone.  The comparison is taken up in
\S\ref{sec:questions}.

The unit factors in Proposition~\ref{prop:exactreduction}(2) are explicit,
and they reduce modulo $\mathfrak{m}_W$ to a quantity that
Part~\ref{part:evidence} computes.

\begin{corollary}[The residue of the criterion class]\label{cor:residue}
Let $E/\Q$ have CM by the maximal order $\OK$ of an imaginary quadratic
field $K$, with $L(E,1) = 0$ and $w(E) = +1$, let $f_0$ be the generator
of the conductor $\ff$ of $\psi_E$ fixed in \S\ref{ssec:katz}, let
$\omega_E$ be the period ratio of Definition~\textup{\ref{def:periodratio}},
and let $p = \fp\bar\fp \ge 5$ be a prime of good reduction split in $K$
with $p \notin S_E$, the set of \eqref{eq:Sexc}.  Let $a_p$ be the trace
of Frobenius of $E$ at $p$ and $B(\fp)$ the bracket \eqref{eq:bracket}.
Then
\[
   2\,c_2(p) \;\equiv\; (f_0\,\omega_E)^{-1}\,a_p^{-2}\;p^{-2}B(\fp)
   \pmod{\mathfrak{m}_W},
\]
the right-hand side read in $W$ through $\iota_p$.
\end{corollary}

\begin{proof}
By Lemma~\ref{lem:comparison}, $L_p(E,T) = c_p\,u(T)\,L^{\mathrm{K}}(T)$
with $c_p = (f_0\omega_E\Omega_p)^{-1}$ and $u(0) = 1$, and
$c_0(L^{\mathrm{K}}) = c_1(L^{\mathrm{K}}) = 0$ by
Proposition~\ref{prop:jetformula}; so Lemma~\ref{lem:decoupling}(1) gives
$c_2(p) = (f_0\omega_E\Omega_p)^{-1}\,c_2(L^{\mathrm{K}})$.  By
Proposition~\ref{prop:grading}(1) and Steps~2 and~3 of the proof of
Proposition~\ref{prop:exactreduction},
\[
   2\log_p(1+p)^2\,c_2(L^{\mathrm{K}})
   \;\equiv\; \Omega_p^{\,2p-1}\,B(\fp) \pmod{p^3 W},
\]
and $\log_p(1+p) \equiv p \pmod{p^2}$, so dividing by $p^2$,
\[
   2\,c_2(p) \;\equiv\; (f_0\omega_E)^{-1}\,\Omega_p^{\,2p-2}\;p^{-2}B(\fp)
   \pmod{\mathfrak{m}_W}.
\]
It remains to show $\Omega_p^{\,p-1} \equiv a_p^{-1} \pmod{\mathfrak{m}_W}$.
The period $\Omega_p$ of \cite[\S 3.3]{BannaiKobayashi} satisfies
$\Omega_p^{\sigma_{\mathfrak a}} = \Lambda(\mathfrak a)\,\Nm\mathfrak{a}^{-1}\,\Omega_p$
for every ideal $\mathfrak a$ prime to $\ff\bar\fp$, where
$\sigma_{\mathfrak a}$ is the Artin symbol of $\mathfrak a$ in the Galois
group of $K(\ff, E[\bar\fp^\infty])$ over $K$, extended by continuity to $W$,
and $\Lambda(\mathfrak a)$ is the scalar by which the isogeny attached
there to $\mathfrak a$ acts on $\omega_E$; for $E$ defined over $\Q$ that
isogeny is the endomorphism $\psi_E(\mathfrak a)$, so
$\Lambda(\mathfrak a) = \psi_E(\mathfrak a)$.  Take $\mathfrak a = \fp$.
The extension is unramified at $\fp$, so $\sigma_\fp$ acts on $W$ as the
Frobenius, and $\Lambda(\fp)/\Nm\fp = \pi/p = \bar\pi^{-1}$, whose image
under $\iota_p$ is $\alpha_p^{-1}$.  Reducing modulo $\mathfrak{m}_W$, the
residue $\bar\omega \in \overline{\F}_p^\times$ of $\Omega_p$ satisfies
$\bar\omega^{\,p} = \bar\alpha_p^{-1}\bar\omega$, that is
$\Omega_p^{\,p-1} \equiv \alpha_p^{-1} \pmod{\mathfrak{m}_W}$; and
$\alpha_p \equiv a_p \pmod p$, the other root of $X^2 - a_pX + p$ being
divisible by $p$.
\end{proof}

\part{Computations}\label{part:evidence}
\section{The unit condition at the split primes below $30{,}000$}
\label{sec:evidence}

This section computes the unit condition for the five elliptic curves
of rank two with CM by $\Z[i]$ of Coates, Liang and Sujatha
\cite[Thm.~1.3]{CLS2}, at every split prime of good reduction below
$30{,}000$.  It follows
Algorithm~11.1 of Stein and Wuthrich \cite[\S 11]{SteinWuthrich}: the
Mordell--Weil group, the $p$-adic regulator, the leading coefficient of
$L_p(E,T)$ from modular symbols with the error bounds of their \S 3,
and the bookkeeping of \eqref{eq:padicbsd}.  Their algorithm closes with
Kato's divisibility, which is not available for a CM curve; here
Theorem~\ref{prop:consequence} takes its place.

What is computed is the regulator.  It bears on the unit condition
through Proposition~\ref{prop:dictionary}, and the direction
$(2) \Rightarrow (1)$, the one that turns a regulator measurement into a
statement about $\tilde c_2(p)$, requires $\Sha(E/\Q)[p^\infty] = 0$.
For the five curves that hypothesis is a theorem at every split prime of
good reduction below $30{,}000$ (\S\ref{ssec:clsinput}), so for such a
prime outside $S_E$
\[
   v_\fp\bigl(\Reg_\fp\bigr) = 2
   \quad\text{and}\quad
   \Sha(E/\Q)[p^\infty] = 0
   \qquad\Longrightarrow\qquad
   \tilde c_2(p) \in \Z_p^\times ,
\]
the equality being measured and the vanishing proved.  The scan of
\S\ref{ssec:scan} covers $8{,}052$ primes, $8{,}050$ of them outside
$S_E$, and the implication is stated as Proposition~\ref{prop:scaneq}.
Above $30{,}000$ only the opposite direction of
Proposition~\ref{prop:dictionary} remains, which assumes nothing about
$\Sha$: there a scan can refute the unit condition but not verify it.

\subsection{The five curves}\label{ssec:testbed}
Coates, Liang and Sujatha \cite[Thm.~1.3]{CLS2} work with the five
curves of rank two
\[
   E_i \colon y^2 = x^3 - D_i x, \qquad
   D_1 = -14,\ D_2 = 17,\ D_3 = -33,\ D_4 = -34,\ D_5 = -39,
\]
all with CM by $\Z[i]$.  We work with the same five, except that in
place of $E_1$ we take the $2$-isogenous curve
\[
   E \colon y^2 = x^3 - 56x ,
\]
of the same conductor $12{,}544 = 2^8\cdot 7^2$; the isogeny class of
$E_1$ is $\{E_1, E\}$.  The curve $E$ satisfies the hypotheses of
Theorems~\ref{prop:consequence} and~\ref{thm:criterion} itself, and the
input on $\Sha$ transfers from $E_1$ by Lemma~\ref{lem:isogeny}.
Table~\ref{tab:curves} lists the invariants used below; they are
quoted from the LMFDB \cite{LMFDB}, whose page for each curve is linked
from its label.

\begin{table}[ht]
\centering
\caption{The five curves $y^2 = x^3 - Dx$ and the invariants used in this
paper (LMFDB).  $\omega_1$ is the least positive real period
and $\Omega_\infty$ generates the period lattice
$\Lambda_E = \Omega_\infty\Z[i]$.}\label{tab:curves}
\small
\begin{tabular}{@{}l@{\;\;}c@{\;\;}c@{\;\;}c@{\;\;}c@{\;\;}c@{}}
\hline
$D$ & $56$ & $17$ & $-33$ & $-34$ & $-39$ \\
\hline
LMFDB label & \href{https://www.lmfdb.org/EllipticCurve/Q/12544/g/1}{\texttt{12544.g1}}
  & \href{https://www.lmfdb.org/EllipticCurve/Q/9248/c/1}{\texttt{9248.c1}}
  & \href{https://www.lmfdb.org/EllipticCurve/Q/69696/i/2}{\texttt{69696.i2}}
  & \href{https://www.lmfdb.org/EllipticCurve/Q/73984/d/2}{\texttt{73984.d2}}
  & \href{https://www.lmfdb.org/EllipticCurve/Q/48672/n/2}{\texttt{48672.n2}} \\
Cremona label & \texttt{12544b2} & \texttt{9248g1} & \texttt{69696eb1} & \texttt{73984m1} & \texttt{48672i1} \\
conductor $N$ & $2^{8}\cdot 7^{2}$ & $2^{5}\cdot 17^{2}$ & $2^{6}\cdot 3^{2}\cdot 11^{2}$ & $2^{8}\cdot 17^{2}$ & $2^{5}\cdot 3^{2}\cdot 13^{2}$ \\
$E(\Q)_{\mathrm{tors}}$ & \multicolumn{5}{c}{$\Z/2\Z = \langle (0,0)\rangle$} \\
$\prod_v c_v$ & $4$ & $4$ & $4$ & $4$ & $8$ \\
$w(E)$ & $+1$ & $+1$ & $+1$ & $+1$ & $+1$ \\
$\rank E(\Q)$ & $2$ & $2$ & $2$ & $2$ & $2$ \\
$P_1$ & $(8,8)$ & $(-1,4)$ & $(4,14)$ & $(8,28)$ & $(3,12)$ \\
$P_2$ & $(9,15)$ & $(-4,2)$ & $(16,68)$ & $(32,184)$ & $(27,144)$ \\
$L(E,1)$ & \multicolumn{5}{c}{$0$ (exactly)} \\
analytic $\Sha$ & \multicolumn{5}{c}{$1.000\ldots$ ($28$ digits)} \\
$a_5$ & $-2$ & $-4$ & $-4$ & $-2$ & $-2$ \\
rational isogenies & \multicolumn{5}{c}{degrees $1$ and $2$ only} \\
$\ff$ & $(56)$ & $(1+i)^3(17)$ & $(132)$ & $(136)$ & $(1+i)^3(39)$ \\
$\Nm\ff = N/4$ & $3136$ & $2312$ & $17424$ & $18496$ & $12168$ \\
$\#\mathrm{Cl}_\ff(K)$ & $384 = 2^{7}\cdot 3$ & $256 = 2^{8}$ & $1920 = 2^{7}\cdot 3\cdot 5$ & $2048 = 2^{11}$ & $1152 = 2^{7}\cdot 3^{2}$ \\
$\Omega_\infty$ & $\omega_1$ & $\omega_1$ & $\tfrac{1+i}{2}\,\omega_1$ & $\tfrac{1+i}{2}\,\omega_1$ & $\tfrac{1+i}{2}\,\omega_1$ \\
\hline
\end{tabular}
\end{table}

Each curve satisfies the hypotheses of Theorems~\ref{prop:consequence}
and~\ref{thm:criterion}: $\rank E(\Q) = 2$ is \cite[Thm.~1.3]{CLS2};
$w(E) = +1$ is the product of the local root numbers, recomputed by
PARI/GP; and $L(E,1) = 0$ exactly, since $L(E,1)/\Omega_E$ is a
modular symbol, an exact rational computed by linear algebra, and it is
$0$.  The points $P_1, P_2$ generate $E(\Q)$ modulo torsion.

The conductor $\ff$ of $\psi_E$ is determined by $N = |d_K|\,\Nm\ff$
(\S\ref{ssec:katz}) and by the stability of $\ff$ under complex
conjugation, $E$ being defined over $\Q$: writing $N/4 = 2^{e}m^2$
with $m$ odd, the only conjugation-stable ideal of $\Z[i]$ of that norm
is $(1+i)^{e}(m)$, a split prime $q = \pi\bar\pi$ dividing $m$
contributing $\pi^{a}\bar\pi^{a} = (q)^a$.  This agrees with
\cite[Lem.~3.2]{CLS1}, and with the unit character $\varepsilon$ of
\S\ref{ssec:package}, read off from $a_p = \psi_E(\fp) + \overline{\psi_E(\fp)}$
at the split primes below $10^6$: it is well defined modulo $\ff$ and
modulo no proper divisor of $\ff$, and equals $(D/\cdot)_4^{-1}$ at
every one of those primes.

\subsection{The excluded sets}\label{ssec:excludedsets}
For a curve with CM by $\Z[i]$ only $p = 5$ can be anomalous
(Lemma~\ref{lem:noanomalous}), and it is anomalous exactly when
$a_5 = -4$: so $S_{\mathrm{an}} = \{5\}$ for $D = 17$ and $-33$, and
$S_{\mathrm{an}} = \emptyset$ for the other three curves.  At a
non-anomalous split prime of good reduction the multiplier
$(1-\alpha_p^{-1})^2$ of \eqref{eq:interp} is a unit, so $c_2(p)$ and
$\tilde c_2(p)$ have the same valuation, and $\Reg_\gamma \in \Z_p$ by
the proof of Proposition~\ref{prop:dictionary}; at an anomalous prime
both can fail, and the scan of \S\ref{ssec:scan} finds that they do.

Table~\ref{tab:excluded} bounds the five sets of \eqref{eq:Sexc} by
exact computation.  $S_{\mathrm{bad}}$ is the set of primes dividing
$6N\prod_v c_v\cdot\#E(\Q)_{\mathrm{tors}}\cdot d_K$, which is the set
of bad primes together with $3$.  $S_{\mathrm{red}} \subseteq \{2\}$,
as the only rational isogenies of each curve have degree $1$ and $2$.
$S_{\mathrm{cmp}}$ is contained in the set of primes dividing
$6N\,\Nm\ff$, which is $S_{\mathrm{bad}}$ again, together with
$\operatorname{supp}(\omega_E)$; and $\operatorname{supp}(\omega_E)$
lies in the set of bad primes, because the Katz measure is taken to be
that of $E$ itself, the choice $\mathcal{A} = E$ in force throughout
\S\ref{sec:bracket}, so that $\Omega_E$ and the $\Omega_\infty$ of
Definition~\ref{def:periodratio} differ by a factor supported on the
bad primes.  $S_{\mathrm{cl}}$ is the set of primes dividing
$\#\mathrm{Cl}_\ff(K)$.  Hence
\[
   S_E \subseteq \{2,3,7\},\quad \{2,3,5,17\},\quad \{2,3,5,11\},\quad
   \{2,3,17\},\quad \{2,3,13\}
\]
for $D = 56$, $17$, $-33$, $-34$, $-39$.  The primes $13$ and $17$ are
$\equiv 1 \bmod 4$ but are bad primes of the curves concerned, so the
only split primes of good reduction that $S_E$ excludes are $p = 5$ for
$D = 17$ and for $D = -33$; for $y^2 = x^3 - 56x$ it excludes none.

\begin{table}[ht]
\centering
\caption{The excluded sets of the five curves: bounds on the components
of $S_E$ in \eqref{eq:Sexc}, and the split primes of good reduction that
$S_E$ contains.}\label{tab:excluded}
\begin{tabular}{@{}l l l l l l l@{}}
\hline
$D$ & $S_{\mathrm{bad}}$ & $S_{\mathrm{an}}$ & $S_{\mathrm{red}}$ & $S_{\mathrm{cmp}}$ & $S_{\mathrm{cl}}$ & split good $p$ in $S_E$ \\
\hline
$56$ & $\{2,3,7\}$  & $\emptyset$ & $\subseteq\{2\}$ & $\subseteq\{2,3,7\}$  & $\{2,3\}$   & none \\
$17$  & $\{2,3,17\}$ & $\{5\}$     & $\subseteq\{2\}$ & $\subseteq\{2,3,17\}$ & $\{2\}$     & $5$ \\
$-33$ & $\{2,3,11\}$ & $\{5\}$     & $\subseteq\{2\}$ & $\subseteq\{2,3,11\}$ & $\{2,3,5\}$ & $5$ \\
$-34$ & $\{2,3,17\}$ & $\emptyset$ & $\subseteq\{2\}$ & $\subseteq\{2,3,17\}$ & $\{2\}$     & none \\
$-39$ & $\{2,3,13\}$ & $\emptyset$ & $\subseteq\{2\}$ & $\subseteq\{2,3,13\}$ & $\{2,3\}$   & none \\
\hline
\end{tabular}
\end{table}

\subsection{The Coates--Liang--Sujatha input}\label{ssec:clsinput}
The hypothesis $\Sha(E/\Q)[p^\infty] = 0$ that direction
$(2) \Rightarrow (1)$ of Proposition~\ref{prop:dictionary} needs is a
theorem below $30{,}000$.

\begin{theorem}[Coates--Liang--Sujatha {\cite[Thm.~1.3]{CLS2}}]\label{thm:cls}
Let $E_i \colon y^2 = x^3 - D_ix$ with $D_1 = -14$, $D_2 = 17$,
$D_3 = -33$, $D_4 = -34$, $D_5 = -39$.  For each $i$ and every prime
$p \equiv 1 \pmod 4$ of good reduction for $E_i$ with $p < 30{,}000$,
the group $\Sha(E_i/\Q)[p^\infty]$ is finite, and
\[
   \Sha(E_i/\Q)[p^\infty] \;=\; 0
   \qquad\text{unless } (i,p) \in \{(1,29),\ (1,277),\ (4,577),\ (5,17)\} .
\]
\end{theorem}

\noindent Coates, Liang and Sujatha prove this by applying
Theorem~\ref{thm:clscriterion} at each prime; at the four excepted pairs
$\ord_p(c_p^+) = g + 1$ and the criterion is inconclusive.  The paragraph
following their theorem reports $p$-adic height computations of Wuthrich
at $17$, $29$ and $277$ giving $\Sha(E_i/\Q)[p^\infty] = 0$ there, and
records that $(E_4, 577)$ remained open.  The cyclotomic criterion settles
all four pairs.

\begin{proposition}\label{prop:e4}
Let $E_4 \colon y^2 = x^3 + 34x$.  Then $\Sha(E_4/\Q)[577^\infty] = 0$.
\end{proposition}

\begin{proof}
$E_4$ satisfies the hypotheses of Theorem~\ref{prop:consequence}
(\S\ref{ssec:testbed}), and $577 \equiv 1 \bmod 4$ is a split prime of
good reduction outside $S_E \subseteq \{2,3,17\}$
(\S\ref{ssec:excludedsets}).  The modular-symbol $577$-adic
$L$-function of $E_4$, computed from eclib modular symbols with the
error bounds of \cite[\S 3]{SteinWuthrich}, is
\[
   L_{577}(E_4, T) \;=\; O(577^4) + O(577)\,T + \bigl(529 + O(577)\bigr)\,T^2 + O(T^3),
\]
so $v_{577}(c_2) = 0$ and, by Proposition~\ref{prop:normalisation},
$\tilde c_2(577) \in \Z_{577}^\times$.  Theorem~\ref{prop:consequence}
gives the claim.
\end{proof}

\noindent Proposition~\ref{prop:dictionary} adds that the cyclotomic
$577$-adic regulator of $E_4$ is a unit, as the scan of
\S\ref{ssec:scan} finds.  The transport from $E_1$ to
$y^2 = x^3 - 56x$ is the following standard fact.

\begin{lemma}[Degree-two isogenies preserve odd-primary
$\Sha$]\label{lem:isogeny}
Let $\varphi \colon A \to B$ be an isogeny of degree $2$ between
elliptic curves over $\Q$, and let $p$ be an odd prime.  Then the
induced map
\[
   \varphi_* \colon \Sha(A/\Q)[p^\infty] \longrightarrow
   \Sha(B/\Q)[p^\infty]
\]
is an isomorphism.
\end{lemma}

\begin{proof}
Let $\hat\varphi \colon B \to A$ be the dual isogeny, so that
$\hat\varphi\varphi = [2]$ on $A$ and $\varphi\hat\varphi = [2]$ on
$B$, whence $\hat\varphi_*\varphi_* = [2]_*$ and
$\varphi_*\hat\varphi_* = [2]_*$ on the respective groups.  For odd $p$
multiplication by $2$ is an automorphism of a $p$-primary torsion
group, so the first relation makes $\varphi_*$ injective and the second
makes it surjective.
\end{proof}

The other three excepted pairs are settled in the same way as
$(E_4, 577)$, for $E_5$ directly and for $E_1$ through its isogenous
curve.

\begin{proposition}\label{prop:threepairs}
Let $E \colon y^2 = x^3 - 56x$ and $E_5 \colon y^2 = x^3 + 39x$.  Then
\[
   \Sha(E/\Q)[29^\infty] \;=\; \Sha(E/\Q)[277^\infty] \;=\; 0
   \qquad\text{and}\qquad
   \Sha(E_5/\Q)[17^\infty] \;=\; 0 .
\]
\end{proposition}

\begin{proof}
Both curves satisfy the hypotheses of Theorem~\ref{prop:consequence}
(\S\ref{ssec:testbed}); $S_E$ contains no split prime for $E$ and is
contained in $\{2,3,13\}$ for $E_5$ (\S\ref{ssec:excludedsets}), so $29$
and $277$ lie outside $S_E$ for $E$, and $17$ outside $S_E$ for $E_5$.
The computation of \S\ref{ssec:unitms}, from eclib modular symbols with
the error bounds of \cite[\S 3]{SteinWuthrich}, certifies
$v_p(c_2(p)) = 0$ at every split prime of good reduction below $1000$
for $E$ and at every one but $5$ for $E_5$ (Table~\ref{tab:unitms}), in
particular at the three pairs concerned; by
Proposition~\ref{prop:normalisation} the unit condition holds there, and
Theorem~\ref{prop:consequence} gives the claim.
\end{proof}

\begin{corollary}\label{cor:shavanishing}
Let $E \colon y^2 = x^3 - Dx$ with $D \in \{56, 17, -33, -34, -39\}$.
Then
\[
   \Sha(E/\Q)[p^\infty] \;=\; 0
\]
for every prime $p \equiv 1 \pmod 4$ of good reduction for $E$ with
$p < 30{,}000$.
\end{corollary}

\begin{proof}
For $D = 17$ and $-33$ this is Theorem~\ref{thm:cls}; for $D = -34$ it
is Theorem~\ref{thm:cls} with Proposition~\ref{prop:e4} at $577$, and
for $D = -39$ Theorem~\ref{thm:cls} with Proposition~\ref{prop:threepairs}
at $17$.  For $D = 56$ the isogeny class of $E$ is $\{E, E_1\}$, the two
curves joined by $2$-isogenies and with the same bad primes; every
$p \equiv 1 \bmod 4$ is odd, so Lemma~\ref{lem:isogeny} identifies
$\Sha(E_1/\Q)[p^\infty]$ with $\Sha(E/\Q)[p^\infty]$, and
Theorem~\ref{thm:cls} gives the vanishing except at $29$ and $277$,
where it is Proposition~\ref{prop:threepairs}.
\end{proof}

For $K = \Q(i)$ the split primes are the primes $p \equiv 1 \bmod 4$,
so below $30{,}000$ the unit condition reduces to a condition on the
regulator alone.

\begin{proposition}[Below $30{,}000$ the unit condition is a regulator
condition]\label{prop:scaneq}
Let $E \colon y^2 = x^3 - Dx$ with $D \in \{56, 17, -33, -34, -39\}$,
and let $p \equiv 1 \pmod 4$ be a prime of good reduction for $E$ with
$p < 30{,}000$ and $p \notin S_E$.  Then
\[
   \tilde c_2(p) \in \Z_p^\times
   \iff
   v_p\bigl(\Reg_\gamma\bigr) = 0
   \iff
   v_\fp\bigl(\Reg_\fp\bigr) = 2 .
\]
\end{proposition}

\begin{proof}
The curve has $\rank E(\Q) = 2$, $L(E,1) = 0$ and $w(E) = +1$
(\S\ref{ssec:testbed}), and $p$ is a split prime of good reduction with
$p \notin S_E$, so Proposition~\ref{prop:dictionary} applies:
$\tilde c_2(p) \in \Z_p^\times$ if and only if the cyclotomic $p$-adic
height pairing on $E(\Q)$ is nondegenerate with $v_p(\Reg_\gamma) = 0$
and $\Sha(E/\Q)[p^\infty] = 0$.  The condition on $\Sha$ holds by
Corollary~\ref{cor:shavanishing}, and the nondegeneracy follows from the
valuation: $\Reg_\gamma$ is a nonzero multiple of the determinant of the
pairing on $P_1, P_2$, so $v_p(\Reg_\gamma) = 0$ forces that
determinant to be nonzero.  For the second equivalence,
$\Reg_\gamma = \Reg_\fp/\log_p(1+p)^{2}$ (\S\ref{ssec:notation}) and
$v_p(\log_p(1+p)) = 1$.
\end{proof}

\subsection{The regulator scan}\label{ssec:scan}
The quantity of Proposition~\ref{prop:scaneq} is the valuation of the
determinant of the cyclotomic $\fp$-adic height pairing on
$\{P_1, P_2\}$, obtained through the $p$-adic sigma function of
Mazur--Stein--Tate \cite[Alg.~3.1]{MST} as implemented in PARI/GP's
\texttt{ellpadicregulator} \cite{PARI}, which computes the height
$\hat h_p$ of \S\ref{ssec:notation}: in the notation of its manual,
$\hat h_p = f - s_2\,g$ with $f = \log_p(\mathrm{denominator}(x(P))) - 2\log_p\sigma(P)$.
SageMath's \texttt{padic\_regulator}, used in \S\ref{ssec:control},
computes the same height.  It admits two readings.  If
$v_\fp(\Reg_\fp) > 2$ at a split prime $p \notin S_E$ then
$\tilde c_2(p) \notin \Z_p^\times$, by direction $(1) \Rightarrow (2)$
of Proposition~\ref{prop:dictionary} read contrapositively, with no
input on $\Sha$: the computation \emph{refutes}.  If
$v_\fp(\Reg_\fp) = 2$ at such a prime below $30{,}000$ then
$\tilde c_2(p) \in \Z_p^\times$ by Proposition~\ref{prop:scaneq}: the
computation \emph{verifies}, and a clean prime is a verified instance
of the unit condition, not a prime at which no counterexample was found.
The standard of verification is that of a measurement: one
implementation of the $p$-adic height, at the precision stated next,
reimplemented independently only at the primes of \S\ref{ssec:control}
and \S\ref{ssec:family}.

The regulator was computed at a working precision of $n$ $\fp$-adic
digits, with
\[
   n = 6 \ \ (p < 2000), \qquad
   n = 5 \ \ (2000 \le p < 10^4), \qquad
   n = 4 \ \ (p \ge 10^4),
\]
at least two digits above the generic valuation $2$.  If the computed
valuation $v$ satisfied $v \ge n-1$, where finite precision could mask
a larger true valuation, the prime was recomputed at precision $n+4$
and its output line flagged.  No line of any of the five scans carries
a flag, so a valuation reported as $2$ is the true valuation.

The scanned set is, for each curve, every split prime of good reduction
below $30{,}000$: the $1611$ primes $p \equiv 1 \bmod 4$ with
$5 \le p \le 29{,}989$, less $17$ for $D = 17$ and $-34$ and $13$ for
$D = -39$; $8{,}052$ primes in all.  Table~\ref{tab:scan} gives the
outcome.  At every one of the $8{,}050$ primes outside $S_E$ the
valuation is $v_\fp(\Reg_\fp) = 2$ except at three, where it is $3$:
$p = 37$ for $y^2 = x^3 + 33x$, and $p = 5$ and $p = 15289$ for
$y^2 = x^3 + 39x$.  At the two anomalous primes it is $0$, so
$\Reg_\gamma \notin \Z_p$ there.  The three exceptions are examined in
\S\ref{ssec:family}.  The output files were re-parsed for primality,
for $p \equiv 1 \bmod 4$, for completeness of the range and for absence
of flags.  The cost per prime depends on $p$ and not on the curve:
$5$\,ms at $p = 101$, $0.55$\,s at $p = 5003$ and $4.9$\,s at
$p = 29{,}989$.

\begin{table}[ht]
\centering
\caption{The regulator scan of the five curves at every split prime of
good reduction below $30{,}000$.}\label{tab:scan}
\small
\begin{tabular}{@{}l r l l@{}}
\hline
curve & primes scanned & $v_\fp(\Reg_\fp) \ne 2$ outside $S_E$ & excluded, $p \in S_E$ \\
\hline
$y^2 = x^3 - 56x$ & $1611$ & none & \\
$y^2 = x^3 - 17x$ & $1610$ & none & $p = 5$ (anomalous), $v = 0$ \\
$y^2 = x^3 + 33x$ & $1611$ & $p = 37$, $v = 3$ & $p = 5$ (anomalous), $v = 0$ \\
$y^2 = x^3 + 34x$ & $1610$ & none & \\
$y^2 = x^3 + 39x$ & $1610$ & $p = 5$ and $p = 15289$, $v = 3$ & \\
\hline
\end{tabular}
\end{table}

The LMFDB records, for each of the five isogeny classes, the valuation
of the $p$-adic regulator at the split primes below $100$ and the
Iwasawa invariants of $L_p(E,T)$ at the primes below $50$
\cite{LMFDB}, without stated error bounds.  The $52$ regulator
valuations agree with the scan files line for line.  The invariants are
$\mu_{\mathrm{an}} = 0$ throughout and $\lambda_{\mathrm{an}} = 2$ at
every split prime below $50$ except $\lambda_{\mathrm{an}} = 4$ at
$(D,p) = (-39, 5)$ and $(-33, 37)$; since $c_0 = c_1 = 0$
(Lemma~\ref{lem:c0c1}), $\lambda_{\mathrm{an}} = 2$ is the unit
condition, and $\lambda_{\mathrm{an}} = 4$ says that $c_2$ and $c_3$
are divisible by $p$ and $c_4$ is a unit, as \S\ref{ssec:family} finds
at $(-39, 5)$.

\subsection{The control on the normalisation}\label{ssec:control}
At two or three primes of each curve the two sides of
\eqref{eq:padicbsd} were computed by disjoint methods and compared
digit for digit: the left-hand side from modular symbols at the level
$N$ of the curve, with the proved coefficient error bounds of
\cite[\S 3]{SteinWuthrich}, and the right-hand side from the cyclotomic
$p$-adic heights of $P_1, P_2$ through the sigma function \cite{MST},
in two implementations (PARI/GP and SageMath) agreeing digit for digit,
with $\#\Sha(E/\Q)[p^\infty] = 1$ as Corollary~\ref{cor:shavanishing}
gives.  Table~\ref{tab:control} lists the runs: the modular-symbol side
sums $(p-1)p^{n-1}$ modular symbols and certifies $c_2(p)$ modulo
$p^{n-1}$, the height side is computed to $14$ digits, and the two
agree on every certified digit above the common valuation.  Two of the
runs are at the exceptional pairs $(-39, 5)$ and $(-33, 37)$: there
both sides have valuation $1$, the height side because
$v_\fp(\Reg_\fp) = 3$ and the $L$-side because $v_p(c_2(p)) = 1$, and
the in-house $\lambda_{\mathrm{an}} = 4$ agrees with the LMFDB.  The
comparison tests the interpolation factor, the $\log_p(1+p)^2$
scaling, the saturation of the basis and the Tamagawa and torsion
bookkeeping at once; it tests nothing else, \eqref{eq:padicbsd} itself
being conjectural.  The checks are mechanical, on the $p$-adic
expansions the three stages print, and a run is recorded only if every
check passes.

\begin{table}[ht]
\centering
\caption{The control on \eqref{eq:padicbsd}: $n$ terms of the
modular-symbol series, $c_2$ certified modulo $p^{k}$, and the number of
$p$-adic digits on which the two sides agree.}\label{tab:control}
\begin{tabular}{@{}l r r r r r@{}}
\hline
curve & $p$ & $n$ & $k$ & $v_p(c_2)$ & digits agreeing \\
\hline
$y^2 = x^3 - 56x$ & $5$  & $8$ & $7$ & $0$ & $7$ \\
                  & $13$ & $6$ & $5$ & $0$ & $5$ \\
$y^2 = x^3 - 17x$ & $13$ & $5$ & $4$ & $0$ & $4$ \\
                  & $29$ & $4$ & $3$ & $0$ & $3$ \\
$y^2 = x^3 + 33x$ & $13$ & $4$ & $3$ & $0$ & $3$ \\
                  & $17$ & $4$ & $3$ & $0$ & $3$ \\
                  & $37$ & $3$ & $2$ & $1$ & $1$ \\
$y^2 = x^3 + 34x$ & $5$  & $8$ & $7$ & $0$ & $7$ \\
                  & $13$ & $4$ & $3$ & $0$ & $3$ \\
$y^2 = x^3 + 39x$ & $5$  & $9$ & $8$ & $1$ & $7$ \\
                  & $17$ & $4$ & $3$ & $0$ & $3$ \\
\hline
\end{tabular}
\end{table}

The error bounds are those of \cite[\S 3]{SteinWuthrich}, a section
whose stated purpose is a recipe for computing $L_p(E,T)$ ``in all
cases, including proven error bounds on each coefficient''.  It opens
by assuming that $E$ has no complex multiplication, which our curves
have.  That hypothesis is a running one for their article, used where
they invoke Serre's surjectivity theorem (\cite[\S 7]{SteinWuthrich},
which flags the dependence) and Kato's Euler system
(\cite[Thm.~8.1]{SteinWuthrich}); it does not enter \S 3, whose bound
reduces to the modular symbol $[x]^+$ having bounded denominator, and
that follows from the finiteness of the cuspidal subgroup of $J_0(N)$
(Manin \cite[Thm.~2.7]{Manin}, Cremona \cite[Ch.~2]{Cremona}) with the
Abel--Jacobi theorem, for any modular curve.  Coates--Liang--Sujatha
supply no alternative here: their criterion never constructs
$L_p(E,T)$, so it bears on the arithmetic side of \eqref{eq:padicbsd}
and not on the analytic one.

\subsection{The chance model}\label{ssec:scanweight}
Model the events $v_\fp(\Reg_\fp) > 2$, one per scanned split prime
outside $S_E$, as independent, each with the probability $1/p$ that a
random element of $\Z_p$ is divisible by $p$; the element in question
is $\Reg_\gamma$.  The number of exceptions for one curve is then
approximately Poisson with mean
\[
   \lambda_E \;=\; \sum_{p} \frac{1}{p} ,
\]
which is $0.880$, $0.621$, $0.680$, $0.821$, $0.803$ for the five
curves in the order of Table~\ref{tab:curves}, and $3.81$ in all; the
two anomalous fives, excluded, would each have contributed $0.2$.  A
clean scan of $y^2 = x^3 - 56x$ has probability $e^{-0.880} = 0.41$
under this model, and the scan of one curve neither confirms nor
refutes it.  The five scans found three exceptions against $3.8$
expected.  Question~\ref{q:infinite} takes this count as its
calibration.

\subsection{The three exceptions}\label{ssec:family}
Each of the three exceptions of Table~\ref{tab:scan} was recomputed at
precisions $6$ to $14$, and $6$ to $16$ at $p = 15289$, with the same
valuation, and by SageMath's implementation of the $p$-adic height with
its own saturated basis, which agrees.  None of $5$, $37$, $15289$ is
anomalous for its curve, divides $2D$, or divides $\#\mathrm{Cl}_\ff(K)$
(Table~\ref{tab:excluded}), so each is outside $S_E$, and each is a
failure of the unit condition by the refuting reading of
\S\ref{ssec:scan}: it is the regulator that fails, and not $\Sha$.  At
$p = 5$ for $y^2 = x^3 + 39x$ the modular-symbol $p$-adic $L$-function
shows the failure directly,
\[
\begin{aligned}
   L_5(E,T) \;=\; O(5^6) + O(5^3)\,T &+ \bigl(4\cdot 5 + 2\cdot 5^2 + O(5^3)\bigr)T^2 \\
   &+ \bigl(3\cdot 5 + 4\cdot 5^2 + O(5^3)\bigr)T^3 \\
   &+ \bigl(4 + 2\cdot 5^2 + O(5^3)\bigr)T^4 + O(T^5),
\end{aligned}
\]
so $c_0 = c_1 = 0$ and $v_5(c_2) = 1$, in agreement with
$v_5(\Reg_\gamma) = 1$ and $\Sha[5^\infty] = 0$ through
\eqref{eq:padicbsd}; the control of \S\ref{ssec:control} repeats the
comparison at $(-33, 37)$.

The criterion of Coates, Liang and Sujatha is inconclusive at four
pairs below $30{,}000$: $p = 29$ and $277$ for $y^2 = x^3 + 14x$, the
isogeny partner of $y^2 = x^3 - 56x$, $p = 17$ for $y^2 = x^3 + 39x$,
and $p = 577$ for $y^2 = x^3 + 34x$ \cite[Thm.~1.3]{CLS2}.  At all
four the cyclotomic regulator is a unit, and at the three exceptions
above their criterion succeeds: the two criteria fail at different
primes.  At all four the vanishing of $\Sha(E/\Q)[p^\infty]$ is
Theorem~\ref{prop:consequence} (Propositions~\ref{prop:threepairs}
and~\ref{prop:e4}); at $17$, $29$ and $277$ this agrees with the height
computations of Wuthrich reported in \cite{CLS2}.  Exceptions of the same kind for non-CM
curves, from Stein and Wuthrich, are recorded with
Question~\ref{q:infinite}.

\subsection{The unit condition by modular symbols below $1000$}\label{ssec:unitms}
Below $1000$ the unit condition was also read off $L_p(E,T)$ directly,
as Stein and Wuthrich do for their $1{,}534{,}422$ non-CM pairs
\cite[\S 12.4]{SteinWuthrich}: for each of the five curves and each
split prime $p < 1000$ of good reduction, $L_p(E,T)$ was computed
modulo $(p, T^3)$ from eclib modular symbols with the error bounds of
\cite[\S 3]{SteinWuthrich}, which certifies $c_2(p)$ modulo $p$, and
modulo $p^2$ at the two primes where $c_2(p) \equiv 0 \bmod p$;
$c_0$ and $c_1$ vanish to the certified precision.
Table~\ref{tab:unitms} gives the outcome at the $397$ pairs.  At every
pair outside $S_E$, $v_p(c_2) = 0$ holds if and only if the scan found
$v_\fp(\Reg_\fp) = 2$, and the two failures $(-39, 5)$ and $(-33, 37)$
show $c_2 \equiv 0 \bmod p$ directly.  At the two anomalous pairs $c_2$
is a unit while $v_\fp(\Reg_\fp) = 0$: in \eqref{eq:padicbsd} the
factor $(1-\alpha_5^{-1})^2$ has valuation $2$, as $5$ divides
$\#\widetilde E(\F_5) = 10$, and cancels the valuation $-2$ of
$\Reg_\gamma$.

\begin{table}[ht]
\centering
\caption{The unit condition read off $L_p(E,T) \bmod (p, T^3)$ at every
split prime of good reduction below $1000$.}\label{tab:unitms}
\begin{tabular}{@{}l r r l@{}}
\hline
curve & primes & $v_p(c_2) = 0$ & $v_p(c_2) > 0$ \\
\hline
$y^2 = x^3 - 56x$ & $80$ & $80$ & none \\
$y^2 = x^3 - 17x$ & $79$ & $79$ & none \\
$y^2 = x^3 + 33x$ & $80$ & $79$ & $p = 37$ \\
$y^2 = x^3 + 34x$ & $79$ & $79$ & none \\
$y^2 = x^3 + 39x$ & $79$ & $78$ & $p = 5$ \\
\hline
\end{tabular}
\end{table}

The consequence is independent of \S\ref{ssec:clsinput}.  At each of
the $393$ pairs of Table~\ref{tab:unitms} with $p \notin S_E$ and
$v_p(c_2) = 0$, Proposition~\ref{prop:normalisation} gives
$\tilde c_2(p) \in \Z_p^\times$ and Theorem~\ref{prop:consequence}
gives $\Sha(E/\Q)[p^\infty] = 0$, with no input from
Theorem~\ref{thm:cls}: below $1000$ the cyclotomic criterion reproduces
the result of Coates, Liang and Sujatha, including the three pairs
$(E_1, 29)$, $(E_1, 277)$ and $(E_5, 17)$ at which their criterion is
inconclusive (Proposition~\ref{prop:threepairs}).  At the two pairs where the cyclotomic
criterion fails, $(-39, 5)$ and $(-33, 37)$, it is their criterion that
gives $\Sha(E/\Q)[p^\infty] = 0$.

\section{The criterion of Theorem~\ref{thm:criterion} on the five curves}\label{sec:bracket}\label{ssec:bracketcert}
Proposition~\ref{prop:exactreduction} makes the criterion class
computable one prime at a time: $\mathfrak{c}(\fp) \ne 0$ if and only if
$v_\fp(B(\fp)) = 2$, where $B(\fp)$ is the combination
\eqref{eq:bracket} of three Eisenstein--Kronecker class sums of weight
$2p-1$, the critical values $L_\ff(\bar\psi_E^{\,2p-1}, s)$ at
$s = 2p-1$, $p$ and $1$ (Corollary~\ref{cor:bracketL}), of the Euler
factors $\mathcal{E}(k,l)$ and of factorials.  For the five curves of
\S\ref{ssec:testbed} the proposition applies at every split prime of
good reduction outside $S_E$, which excludes only $p = 5$ for
$y^2 = x^3 - 17x$ and $y^2 = x^3 + 33x$ (\S\ref{ssec:excludedsets}).

We computed $B(\fp)$ at sixteen pairs $(E, p)$: at $p = 5$, $13$, $17$,
$29$, $37$ for $y^2 = x^3 - 56x$, and at the three smallest split primes
of good reduction outside $S_E$ for each of the other four curves,
except that for $y^2 = x^3 + 33x$ only $13$ and $17$ are reported.
Each class sum lies in $K = \Q(i)$ and is recognised there exactly from
a floating-point value at $600$ decimal digits, under the gates
described in \texttt{code/README.md}; the Euler factors are computed in
$\Z_p$.  The cost grows like $p^4$ per ray class, which is why a few
primes are reported rather than a range.  Table~\ref{tab:bracket} lists
the results.

\begin{table}[ht]
\centering
\caption{The bracket $B(\fp)$ of \eqref{eq:bracket} at sixteen pairs.
The fifth column is $(f_0\omega_E)^{-1}a_p^{-2}\,p^{-2}B(\fp) \bmod \mathfrak{m}_W$,
which Corollary~\ref{cor:residue} identifies with
$\kappa(p) = 2c_2(p) \bmod p$; the last column is $\kappa(p)$ computed
from the modular-symbol $L_p(E,T)$, independently.  The two agree at
every pair.}\label{tab:bracket}
\small
\begin{tabular}{@{}l r r r r r@{}}
\hline
curve & $p$ & $v_\fp(B(\fp))$ & $p^{-2}B(\fp) \bmod \mathfrak{m}_W$ & from $B(\fp)$ & $\kappa(p)$ \\
\hline
$y^2 = x^3 - 56x$ & $5$  & $2$ & $1$  & $2$  & $2$ \\
                  & $13$ & $2$ & $4$  & $2$  & $2$ \\
                  & $17$ & $2$ & $12$ & $15$ & $15$ \\
                  & $29$ & $2$ & $15$ & $12$ & $12$ \\
                  & $37$ & $2$ & $6$  & $17$ & $17$ \\
$y^2 = x^3 - 17x$ & $13$ & $2$ & $3$  & $1$  & $1$ \\
                  & $29$ & $2$ & $17$ & $13$ & $13$ \\
                  & $37$ & $2$ & $27$ & $19$ & $19$ \\
$y^2 = x^3 + 33x$ & $13$ & $2$ & $5$  & $9$  & $9$ \\
                  & $17$ & $2$ & $14$ & $1$  & $1$ \\
$y^2 = x^3 + 34x$ & $5$  & $2$ & $4$  & $2$  & $2$ \\
                  & $13$ & $2$ & $2$  & $9$  & $9$ \\
                  & $29$ & $2$ & $17$ & $21$ & $21$ \\
$y^2 = x^3 + 39x$ & $5$  & $3$ & $0$  & $0$  & $0$ \\
                  & $17$ & $2$ & $10$ & $9$  & $9$ \\
                  & $29$ & $2$ & $10$ & $2$  & $2$ \\
\hline
\end{tabular}
\end{table}

At the fifteen pairs with $v_\fp(B(\fp)) = 2$ exactly,
$\mathfrak{c}(\fp) \ne 0$, and the reductions
$p^{-2}B(\fp) \bmod \mathfrak{m}_W$ lie in $\F_p$.  At $(-39, 5)$ the
valuation is $3$ exactly, so $\mathfrak{c}(\fp) = 0$; the three terms of
\eqref{eq:bracket} have valuation $2$ each there, and the valuation $3$
of their combination is a cancellation.  This is the smallest pair at
which the unit condition fails, and it tests the equivalence of (1) and
(3) in Theorem~\ref{thm:criterion} in the direction the other fifteen
do not: the bracket detects, through three Hecke $L$-values of weight
$2p-1$, the failure that the scan found in the regulator.

The values agree with an independent computation.  The second
coefficient $c_2(p)$ of $L_p(E,T)$ \eqref{eq:Lpjets}, computed from the
modular symbols, is a $p$-adic unit at each of the fifteen pairs and is
divisible by $5$ at $(-39, 5)$ (\S\ref{ssec:control}), which is what
Proposition~\ref{prop:exactreduction} with
Proposition~\ref{prop:grading} and Lemma~\ref{lem:decoupling}
requires; and its residue $\kappa(p) = 2c_2(p) \bmod p$ is, at each of
the sixteen pairs, the one Corollary~\ref{cor:residue} predicts from
$B(\fp)$, $a_p$ and $f_0\omega_E$, with $\omega_E = 2$ for $D > 0$ and
$\omega_E = 1 - i$ for $D < 0$ (Table~\ref{tab:curves}).  The two
computations share no data, one being a sum of
complex division values of the lemniscatic $\wp$-function and the other
a modular symbol.

Two pairs at which the unit condition fails have no bracket.  At
$(-33, 37)$, with $1920$ ray classes, the run at $600$ digits, of $46$
minutes, was inconclusive: two of the three class sums were recognised
in $\Q(i)$, but the third, $P_{2p-2,1}$, was not, its evaluation losing
about $240$ digits at this class count, so no valuation was obtained; a
run at higher precision was not made.  At $(-39, 15289)$ the run was not attempted: the cost per
class grows like $p^4$, and the working precision with it, so it would
take about $10^{10}$ times the six minutes of $(17, 37)$.  At both pairs
the unit condition is decided by the regulator alone,
$v_\fp(\Reg_\fp) = 3$ giving $\tilde c_2(p) \notin \Z_p^\times$ by the
refuting reading of \S\ref{ssec:scan}, and at $(-33, 37)$ also by
$v_{37}(c_2) = 1$; Theorem~\ref{thm:criterion} then predicts
$v_\fp(B(\fp)) \ge 3$ at both, which has not been checked.

The evidence is computational: each number quoted is an exact element
of $\Q(i)$ whose exactness rests on those gates and not on an algebraic
identity.

\part{Formalisation}\label{part:formal}
The results of Part~\ref{part:theory} are accompanied by a formal
verification in Lean~4, in the \texttt{formalisation/} directory of the
repository linked to in \S\ref{ssec:software}.  Its purpose is to reduce the work of
checking this paper: a referee who accepts a short list of statements
quoted from the literature need not check the deductions that run from
them to the results.

A blueprint, stating each result informally beside the declaration that
formalises it, is in \texttt{formalisation/blueprint/} of that
repository, with the two commands that build it.

\section{Why the formalisation is partial}\label{sec:formalpartial}
A complete formalisation would construct, inside Lean, the Selmer group
$\Sel_{p^\infty}(E/\Q)$, the group $\Sha(E/\Q)$, the Katz two-variable
$p$-adic $L$-function, and the cyclotomic $p$-adic height pairing, and
would prove Rubin's main conjecture for $K$, Mazur's control theorem,
Schneider's leading-term theorem and the Bannai--Kobayashi comparison.
Mathlib, the standard Lean mathematical library \cite{Mathlib},
contains none of these.  It has no Shafarevich--Tate group, no Weil--Ch\^atelet group,
and no Selmer group of an elliptic curve; it has no $p$-adic
$L$-function of any kind.  Supplying them is a formalisation project far beyond the scope of this project.

What \emph{is} achievable is a formalisation \emph{modulo the literature}:
the classical theorems this paper uses are assumed, in explicitly
stated form, and everything the paper adds to them is machine-checked.
The rest of this part makes that precise.

\section{What is proved}\label{sec:formalproved}
The classical inputs are collected in one Lean structure,
\texttt{ClassicalInputs}, whose fields are the statements this paper
takes from the literature.  Each field carries the precise citation it
renders.  The formal results are then literal implications: they take a
value of that structure as a hypothesis and assert a conclusion.  Two
are proved: they are the rows above the double rule below.  Under it
is a selection of subsidiary statements of the paper that are proved as well, several
of them ring-generic lemmas that take no classical input at all.

\begin{center}
\begin{tabular}{l l}
\hline
This paper & Lean declaration \\
\hline
Theorem~\ref{intro:consequence} $=$ Theorem~\ref{prop:consequence}
  & \texttt{prop\_consequence} \\
Proposition~\ref{prop:dictionary}
  & \texttt{prop\_dictionary} \\
\hline\hline
Lemma~\ref{lem:noanomalous}(2)
  & \texttt{anomalous\_iff\_five} \\
Lemma~\ref{lem:c0c1}
  & \texttt{c0\_eq\_zero}, \texttt{c1\_eq\_zero} \\
Proposition~\ref{prop:normalisation}
  & \texttt{isPUnit\_c2tilde\_iff} \\
Proposition~\ref{prop:grading}(2)
  & \texttt{isUnit\_iff\_residue\_ne\_zero\_of\_grading\_congr} \\
Lemma~\ref{lem:decoupling}
  & \texttt{Decoupling.coeff\_two\_mul} \\
\hline
\end{tabular}
\end{center}

\noindent Where several declarations together render one statement, the
table names the one that states it; \texttt{formalisation/README.md}
lists the others.

Two properties of the verification are checked by the
script \texttt{scripts/audit.sh}:
\begin{enumerate}
  \item No proof in the imported tree is incomplete.
  \item The $73$ audited declarations
depend on no axiom beyond the three standard ones of the Lean
library:
\[
   \texttt{propext}, \qquad \texttt{Classical.choice}, \qquad
   \texttt{Quot.sound};
\]
and in particular on no assumption introduced by this project.  
\end{enumerate}
There are no \texttt{axiom} declarations anywhere in the formalisation: every
assumed input is a field of a structure, and therefore appears in the
statement of any theorem that uses it.  No conjectural statement is
assumed anywhere: the unit condition enters
Theorem~\ref{prop:consequence} as an ordinary hypothesis at a single
prime, and every other hypothesis is a classical theorem with its
citation.  The version of mathlib is
recorded in the repository; at the time of writing it is
\texttt{v4.33.1}.

That the assumptions are consistent, and that they do not already
contain the conclusion, are also proved.  The repository constructs a
value of \texttt{ClassicalInputs} in which every field holds, which
shows the assumption set is satisfiable; and a second value in which
the Shafarevich--Tate object is nontrivial, which shows the assumptions
alone do not force $\Sha(E/\Q)[p^\infty] = 0$.  Without the first, the
assumptions could be contradictory and every theorem above vacuously
true; without the second, they could already contain the conclusion.

\section{What a referee has to check}\label{sec:formalreferee}
Lean certifies that the conclusions follow from the assumptions.  It
cannot certify that the assumptions are true, or that they say what the
cited papers say.  That is the check the formalisation leaves to a
human, and it is finite: read the fields of \texttt{ClassicalInputs} and
its constituent structures, in the directory
\texttt{formalisation/FinShaRank2/Interface/}, and judge whether each is
a faithful statement of the result it cites.  There are six such files,
one for each layer: the Mazur--Tate--Teitelbaum $L$-function, the
Selmer and Iwasawa modules, the $p$-adic height pairing, the Katz
measure, the Bannai--Kobayashi jet, and the aggregate.  The
guide \texttt{formalisation/README.md} lists them statement by statement
against the labels of this paper, and records which are proved, which are
assumed, and which are not formalised.

A referee who accepts those statements, and who trusts the Lean kernel,
can take the results in Part~\ref{part:theory} as verified.

\section{Limits of the formalisation}\label{sec:formallimits}
Beyond the standing limitation that the existing literature results are assumed to be correct (rather than formally verified to be so), we indicate the following further limitations.

\begin{enumerate}
   \item Part~\ref{part:evidence} is not formalised.  The computations are checked by the means described in \S\ref{ssec:software}, and the formalisation says nothing about them.

  \item One statement is not formalised as the paper states it:
Lemma~\ref{lem:noanomalous}(1) rests on Deuring's reduction criterion,
which mathlib does not have, and is not formalised; part~(2) is.

  \item Theorem~\ref{thm:criterion} is partly formalised.  The
equivalence of its statements (1) and (2) is
Proposition~\ref{prop:dictionary}, in the table above.  That of (1) and
(3) has no Lean counterpart: Lemma~\ref{lem:moments} and
Proposition~\ref{prop:exactreduction} are not formalised, and a
formalisation of Theorem~\ref{thm:criterion} would begin there.

  \item The Eisenstein--Kronecker data of Definition~\ref{def:classsums}
is assumed --- the set $D_E$, six of the class functions $r_{a,b}$, and a
valuation at each rational prime --- but no theorem reads a section
value, the criterion those functions feed being the half of
Theorem~\ref{thm:criterion} that is not formalised.  The Lean rendering
carries the six indices with $a + b \le 3$; the paper fixes no finite
index set.

\end{enumerate}

\part{Questions}\label{part:outlook}
This part records what the paper does not settle.  The limits of the
formalisation are recorded in \S\ref{sec:formallimits} and not
repeated here.

\section{Questions}\label{sec:questions}\label{ssec:ranktwo}\label{ssec:gillardmech}
The computations of Part~\ref{part:evidence} raise the three questions
of \S\ref{ssec:intro_questions}, repeated here with what bears on them.
Throughout, $E/\Q$ has CM by the maximal order $\OK$ of an imaginary
quadratic field $K$, with $\rank E(\Q) = 2$, $L(E,1) = 0$ and
$w(E) = +1$, and $p$ runs over the split primes of good reduction
outside $S_E$.  By Theorem~\ref{thm:criterion} the unit condition
$\tilde c_2(p) \in \Z_p^\times$ fails at $p$ exactly when the
cyclotomic regulator or $\Sha(E/\Q)[p^\infty]$ acquires a factor of
$p$, and exactly when $v_\fp(B(\fp)) \ge 3$.

\begin{qrestate}{q:infinite}
Is the set of split primes at which the unit condition fails an
infinite set?
\end{qrestate}

If the residue of $\Reg_\gamma$ modulo $p$ behaved like a random
element of $\F_p$, the unit condition would fail with probability $1/p$
at each split prime, at about $\tfrac12\log\log X$ primes below $X$:
infinitely many, at a rate too slow to see.  The scans of
\S\ref{ssec:scan} found three failures in five curves below $30{,}000$
against $3.8$ expected (\S\ref{ssec:scanweight}).  The same count is
available for non-CM curves from Stein and Wuthrich: for the curve of
conductor $389$, one prime, $16231$, among the $5{,}005$ good ordinary
primes below $48{,}859$ has $\ord_p(\Reg_p) = 3$ against the rank $2$
\cite[Thm.~12.3, Rem.~12.4]{SteinWuthrich}, where the model expects
about $1.8$; and among the curves of rank at least $2$ and conductor
below $130{,}000$, at primes below $1000$, they list excesses of up to
$5$ in the valuation of $\Reg_p$ over the rank, matched by the
valuation of the leading coefficient of $L_p(E,T)$, with
$\Sha[p^\infty] = 0$ \cite[Rem.~12.1]{SteinWuthrich}.  The question for
the value $L(\bar\psi_E^{\,p}, p)$ of Theorem~\ref{thm:clscriterion} has
the same shape, and the four pairs below $30{,}000$ at which Coates,
Liang and Sujatha found $\ord_p(c_p^+) = g + 1$ are its instances.

\begin{qrestate}{q:independent}
Are the failures of the cyclotomic criterion and of the criterion of
Coates, Liang and Sujatha independent?
\end{qrestate}

Below $30{,}000$ the two criteria failed at disjoint sets of primes: at
each of the four pairs where theirs is inconclusive the cyclotomic
regulator is a unit, and at each of the three where the regulator is
not a unit theirs succeeds (\S\ref{ssec:family}); at $(-39, 5)$ the
bracket shows the failure as well (\S\ref{sec:bracket}).  The two test
critical values of different weights, $L(\bar\psi_E^{\,p}, p)$ and the
three values of $\bar\psi_E^{\,2p-1}$ in $B(\fp)$
(Corollary~\ref{cor:bracketL}), and no implication between them is
known.  If the two events at $p$ are independent with probability about
$1/p$ each, the expected number of primes at which both occur is
bounded by $\sum_p p^{-2}$, and a Borel--Cantelli heuristic suggests
that for all but finitely many split $p$ at least one of the two
criteria gives $\Sha(E/\Q)[p^\infty] = 0$.  For the five curves of
Part~\ref{part:evidence} the two together decide every split prime of
good reduction below $30{,}000$, the last case by
Proposition~\ref{prop:e4}.

We speculate about the existence of a single algebraic object that
governs the failure of the unit condition, analogous to the Heegner
point in rank one.

\begin{qrestate}{q:fixed}
Is there a fixed algebraic number whose prime divisors contain every
split prime at which the unit condition fails?
\end{qrestate}

For the zeroth jet the answer is yes: $c_0(L^{\mathrm{K}})$ is a unit
multiple of the class sum $P_{0,1}$ (Lemma~\ref{lem:moments}), one
algebraic number, and Gillard's theorem (Theorem~\ref{thm:gillard}) is
the statement of that shape for the whole measure.  Its proof is
Sinnott's argument \cite[Thm.~1]{Sinnott} transplanted to the CM curve:
for every split $\fp$ the same fixed rational function
$\partial\log\Theta_{\mathfrak a}$ of \S\ref{ssec:ellunits} generates
the measure, only the specialisation data moving with $\fp$; a positive
$\mu$-invariant would force an algebraic identity among the torsion
translates of that function over $\overline{\F}_p$, which Gillard's
linear-independence theorem \cite[Th.~1.5.1]{Gillard} forbids, its
divisor being explicit and nonzero.  One algebraic nonvanishing thus
excludes a positive $\mu$-invariant at every split prime at once.  For
the second jet no such object is in sight.  What the Bannai--Kobayashi
theory of \S\ref{ssec:BK} delivers are the polynomial moments of the
measure, values of fixed algebraic quantities; the second jet in the
cyclotomic direction is an integral against $\log_p$, and
Proposition~\ref{prop:exactreduction} reaches it through a polynomial of
degree $2p-2$, paired with class sums of weight $2p-1$, which vary with
$p$.  The Kummer congruence of the measure relates those class sums to
the ones of weight at most $3$ only modulo $p$, while the criterion
class needs the bracket modulo $p^3$.

The answer can be no.  For the Kubota--Leopoldt functions
$L_p(s,\omega^k)$ the $\mu$-statement is the theorem of
Ferrero--Washington \cite{FerreroWashington}, parallel to Gillard's, but
the $\lambda$-statement is false: $\lambda > 0$ occurs exactly at the
irregular primes, governed by $p \mid B_k$ with $k$ taken modulo $p-1$,
so that the Bernoulli number involved depends on $p$ and the
exceptional set is detected by no fixed algebraic number.  The
criterion of Coates, Liang and Sujatha is of the same kind, testing a
critical value of $\bar\psi_E^{\,p}$ whose weight moves with $p$, and
their bound on the $\Z_p$-corank of $\Sha(E/\Q)[p^\infty]$, of order
$p$ for all large good ordinary $p$ \cite[Thm.~1.1]{CLS2}, shows the
ceiling of that input: the bound is obtained by applying the product
formula to a nonzero algebraic integer of $K$ built from
$L(\bar\psi_E^{\,p}, p)$, of logarithmic height $\tfrac12 p\log p + O(p)$
\cite[\S 4]{CLS2}, and any such argument yields a bound of order
$h(p)/\log p$.  Finiteness of $\Sha(E/\Q)[p^\infty]$ for large $p$ by
that route needs an input of height $o(\log p)$, which is to say one
fixed algebraic number rather than a family indexed by $p$.

The rank-two Selmer constructions do not supply one.  The generalised
Kato classes of Darmon--Rotger \cite{DarmonRotger16, DarmonRotgerJAMS},
built from $p$-adic families of diagonal cycles, and the criterion of
Castella--Hsieh that the nonvanishing of one of them forces
$\dim_{\Q_p}\Sel(\Q, V_pE) = 2$ \cite[Cor.~C]{CastellaHsieh}, hence
$\Sha(E/\Q)[p^\infty]$ finite in rank two, are constructed one prime at
a time, with no rational structure to compute once and factor; the
nonvanishing is a machine verification, twenty-one non-CM pairs
\cite[\S 6]{CastellaHsieh}, and the criterion requires a prime of
multiplicative reduction \cite[\S 1]{CastellaCM}, which a CM curve over
$\Q$ does not have.  Castella's CM analogue constructs
$\kappa_p \in \Sel(\Q, V_pE)$ with the same implication
\cite[Thm.~B]{CastellaCM}, but its nonvanishing is not known at a single
prime, and the statement there that produces a basis of the Selmer group
assumes $\Sha(E/\Q)[p^\infty]$ finite \cite[Thm.~C]{CastellaCM}.
Darmon and Rotger say of their own Euler system that it does not yield
finiteness of the Selmer or Shafarevich--Tate group
\cite[Introduction]{DarmonRotgerJAMS}.  The other routes stop at the
same place: Kolyvagin's conjecture on derived Heegner classes, proved
for large classes of non-CM curves
(\cite[Thms.~9.1, 9.3]{ZhangKolyvagin}; \cite[Thm.~A]{Sweeting}),
yields in rank $\ge 2$ an identity between vanishing orders and Selmer
coranks rather than a bound (\cite[Thm.~11.2]{ZhangKolyvagin};
\cite[Cor.~B]{Sweeting}), with finiteness of $\Sha[p^\infty]$ requiring
a further computation at each prime
\cite[Prop.~3.10, Rmk.~3.11]{JLS}; the Beilinson--Flach systems bound
Selmer groups in twisted rank-zero settings \cite[Thm.~11.6.4]{KLZ};
and the plectic points of Fornea--Gehrmann \cite[\S 1]{ForneaGehrmann}
are conjectural and degenerate to rank one over $\Q$.  No construction
is known that attaches an algebraic number to $L''(E,1)$.

\section{The inert half}\label{sec:inert}
Every statement of \S\S\ref{sec:thmA}--\ref{sec:criterion} is about
split primes, where the reduction is ordinary.  At an inert prime it is
supersingular \cite{Deuring}, \cite[Ch.~13, \S 4]{LangEF}, and three of
the objects the argument runs on are absent.  There is no unit root
$\alpha_p$, so the multiplier $(1-\alpha_p^{-1})^2$ of \eqref{eq:interp}
is not a unit and Definition~\ref{def:c2tilde} has no analogue as
stated.  The Mazur--Tate--Teitelbaum construction gives no element of
$\Lambda$, and the Selmer module is not $\Lambda$-torsion
\cite[Introduction]{PollackRubin}; what replaces them is Pollack's pair
$L^{\pm}_E \in \Lambda$ \cite[\S 7]{PollackRubin} and Kobayashi's
signed Selmer groups \cite[Def.~3.3]{PollackRubin}, whose main
conjecture is a theorem for CM curves at good supersingular $p > 2$
\cite[Thm.~7.3]{PollackRubin}.  And the generating object of
\S\ref{sec:questions} is not integral at $\fp$: Bannai and Kobayashi
find its two-variable expansion at a supersingular prime to have
$p$-adic radius of convergence $p^{-p/(p^2-1)} < 1$, so that no
two-variable measure exists \cite[\S 4]{BannaiKobayashi}, and record
Kurihara's expectation that the cyclotomic $\mu$-invariant there is
$p/(p^2-1)$ rather than $0$.  The inert zeroth-jet input will not be
the vanishing of a $\mu$-invariant.

The dictionary has inert counterparts.  Bernardi and Perrin-Riou
\cite{BernardiPerrinRiou} and Sprung \cite[Conj.~1.2 and \S 4.1]{Sprung}
formulate the $p$-adic Birch--Swinnerton-Dyer conjecture at
supersingular primes for the pair of $p$-adic $L$-functions, with the
height pairings of \cite{KobayashiHeights}; the inert form of the unit
condition is that statement with the minimum order of vanishing equal
to $2$ and the leading coefficient a $p$-adic unit.  On the arithmetic
side Castella proves the corresponding leading-term theorem for the
signed Selmer groups, given finiteness of $\Sha(E/\Q)[p^\infty]$ and
nondegeneracy of the regulator \cite[Thm.~A]{CastellaSigned}, which
with the main conjecture are the ingredients of an inert analogue of
Proposition~\ref{prop:dictionary}; we have not written one out.  The
anticyclotomic inert theory has moved: Rubin's conjecture on the local
units is a theorem of Burungale--Kobayashi--Ota for $p \ge 5$
\cite[Thm.~2.1]{BKO1}, who relate values of Rubin's inert $p$-adic
$L$-function to rational points \cite[Thm.~1.1]{BKO2}, the counterpart
of \cite{Rubin92}; and Kobayashi's $p$-adic Gross--Zagier formula at
supersingular primes \cite{KobayashiGZ} reaches rank one only, both
sides vanishing in rank two by Gross--Zagier \cite{GrossZagier} and
Kolyvagin \cite{Kolyvagin}.  Two things would have to be settled before
the scan of \S\ref{ssec:scan} could be repeated at inert primes: the
generic valuation of the pair of regulators for a curve of rank two,
computable by the $\sigma$-function algorithms of
\cite[\S 4]{SteinWuthrich}, and the rigidity input, for which no
supersingular counterpart of the generating object is known.

\section{Nonvanishing of the rank-two $\fp$-adic
regulator}\label{sec:bertrand}
One transcendence-theoretic result bears on the height pairing, and it
is a rank-one result: for a CM elliptic curve the $p$-adic height of a
single point of infinite order is nonzero, by Bertrand's work on theta
functions in several variables \cite{Bertrand}.  The form we quote is
Rubin's \cite[Thm.~A.1]{AgboolaHoward}: for a number field $F$ and a
point $P \in E(F)$ of infinite order, the $\fp$-adic height of $P$
along the $\Z_p$-extension of $K$ unramified outside $\fp$ is nonzero.
That height is not the cyclotomic one, and Rubin reaches the cyclotomic
height only for points that are universal norms in the anticyclotomic
tower \cite[Lem.~A.2 and~A.4]{AgboolaHoward}; for the cyclotomic
pairing on a CM curve over $\Q$, Burungale--Disegni attribute
nonvanishing to the same result \cite[\S 1.1]{BurungaleDisegni}.  Either
way the statement is about one point.  In rank two the quantity that
must be nonzero is a determinant.

\begin{problem}\label{prob:determinant}
Let $E/\Q$ be an elliptic curve with complex multiplication by the
maximal order of an imaginary quadratic field $K$.  Assume:
\begin{enumerate}
\item[(a)] $P_1, P_2 \in E(\Q)$ are independent points of infinite
      order;
\item[(b)] $p = \fp\bar\fp$ is a split prime of good reduction.
\end{enumerate}
Prove that the cyclotomic $\fp$-adic height pairing
$\langle\ ,\ \rangle_\fp$ on $E(\Q)$
\textup{(\S\ref{ssec:notation})} satisfies
\[
   \det\bigl(\langle P_i, P_j\rangle_\fp\bigr) \;\ne\; 0 .
\]
\end{problem}

In rank at least two no theorem gives the determinant nonzero at a
single prime of a single curve \cite[\S 1.1]{BurungaleDisegni}.
Part~\ref{part:evidence} gives it at finitely many primes, at two
standards.  The scans of \S\ref{ssec:scan}, at the standard of a
measurement, give it at all $8{,}052$ pairs, a finite valuation being a
nonzero value.  The bracket of \S\ref{ssec:bracketcert}, under the
gates described in \texttt{code/README.md}, gives $v_\fp(B(\fp)) = 2$
at fifteen pairs; by Proposition~\ref{prop:exactreduction}(3) that is
$\mathfrak{c}(\fp) \ne 0$, by Proposition~\ref{prop:grading}(2) with
Lemma~\ref{lem:decoupling}(2) and
Proposition~\ref{prop:normalisation} it is
$\tilde c_2(p) \in \Z_p^\times$, and by
Proposition~\ref{prop:dictionary}, direction $(1) \Rightarrow (2)$, it
is the nondegeneracy of the pairing together with
$v_p(\Reg_\gamma) = 0$.  Neither settles the Problem, which asks for
every split $\fp$ of good reduction, or for all but finitely many.

Nor would a solution give the unit condition.  By
Proposition~\ref{prop:dictionary}, at a split $p \notin S_E$ of good
reduction the unit condition is equivalent to
\[
   \text{$\langle\ ,\ \rangle_\fp$ nondegenerate on $E(\Q)$ with }
   v_p\bigl(\Reg_\gamma\bigr) = 0
   \qquad\text{and}\qquad
   \Sha(E/\Q)[p^\infty] = 0 ,
\]
and the first of these is strictly stronger than the nondegeneracy the
Problem asks for.  A solution would give the nondegeneracy, not the
valuation, and would say nothing about $\Sha$.  Transcendence methods
attack nonvanishing; the criterion of \S\ref{sec:criterion} decides
the valuation.

\section{Non-CM curves}\label{sec:noncm}
Nothing in \S\S\ref{sec:thmA}--\ref{sec:criterion} survives the
removal of complex multiplication.  Theorem~\ref{prop:consequence} uses
Rubin's main conjecture for $K$; Theorem~\ref{thm:criterion} runs on
the Katz measure and on the algebraic theta function of
\S\ref{ssec:BK}, both attached to the CM lattice; and the mechanism of
\S\ref{sec:questions} needs a single algebraic object generating the
measures at every split prime at once, which the CM lattice supplies.
The gap shows already at the zeroth jet.  For a CM curve the
zeroth-jet statement is Gillard's theorem.  Without complex
multiplication it is the vanishing of $\mu_{\mathrm{an}}(p)$, which is
conjectural, one prime at a time: Greenberg conjectures $\mu = 0$ for
the Selmer group over $\Q_\infty$ whenever $E[p]$ is irreducible
\cite[Conj.~1.11]{Greenberg}, and with the main conjecture that gives
$\mu_{\mathrm{an}}(p) = 0$ at that prime.  Uniformity in $p$ is what
the fixed generating object supplies and what the non-CM case lacks a
replacement for; a $p$-adic family of modular forms does not serve, a
Hida family being constructed after $p$ has been fixed.

\end{document}